\documentclass[10pt]{amsart}
\usepackage[margin=3cm]{geometry}
\usepackage{eepic}
\usepackage{epsfig}
\usepackage{graphicx}
\usepackage{eufrak}
\usepackage{mathrsfs}
\usepackage[english]{babel}
\usepackage{color}
\usepackage{mathabx}

\usepackage{palatino, mathpazo}
\usepackage{amscd}      
\usepackage{amssymb}
\usepackage{amsmath}
\usepackage{dsfont}
\usepackage{xypic}      
\LaTeXdiagrams          
\usepackage[all,v2]{xy}
\xyoption{2cell} \UseAllTwocells \xyoption{frame} \CompileMatrices
\usepackage{latexsym}
\usepackage{epsfig}
\usepackage{amsfonts}
\usepackage{enumerate}

\allowdisplaybreaks[3]

\newtheorem{prop}{Proposition}[section]
\newtheorem{lem}[prop]{Lemma}

\newtheorem{thm}[prop]{Theorem}
\newtheorem{rmk}[prop]{Remark}

\newtheorem{defn}[prop]{Definition}

            {\nolinebreak $\Box$ \end{trivlist}}

\newcommand{\noprint}[1]{}

\newcommand{\Ext}{\mbox{Ext}}

\newcommand{\Hom}{\mbox{Hom}}

\newcommand{\pt}{\mathop{pt}}

\newcommand{\E}{\mathop{\sf E}\nolimits}
\newcommand{\sfH}{\mathop{\sf H}\nolimits}

\newcommand{\N}{\mathcal{N}}

\newcommand{\Tt}{{\mathfrak t}}

\newcommand{\zz}{{\mathbb Z}}

\newcommand{\T}{{\mathbb T}}

\newcommand{\aaa}{{\mathbb A}}

\newcommand{\nn}{{\mathbb N}}

\renewcommand{\ll}{{\mathbb L}}
\newcommand{\qq}{{\mathbb Q}}
\newcommand{\pp}{{\mathbb P}}
\newcommand{\cc}{{\mathbb C}}

\newcommand{\Gm}{{{\mathbb G}_{\mbox{\tiny\rm m}}}}

\newcommand{\sT}{{\mathcal T}}
\newcommand{\sD}{{\mathcal D}}
\newcommand{\sC}{{\mathcal C}}
\newcommand{\sE}{{\mathcal E}}
\newcommand{\sI}{{\mathcal I}}

\newcommand{\sL}{{\mathcal L}}

\newcommand{\sS}{{\mathcal S}}

\newcommand{\sU}{{\mathcal U}}
\newcommand{\sO}{{\mathcal O}}

\newcommand{\sX}{{\mathcal X}}

\newcommand{\sM}{{\mathcal M}}

\newcommand{\yY}{{\mathcal Y}}
\newcommand{\sZ}{{\mathcal Z}}

\newcommand{\sF}{{\mathcal F}}

\newcommand{\Coh}{\mbox{Coh}}

\newcommand{\rE}{\mathscr{E}}
\newcommand{\rM}{\mathscr{M}}

\newcommand{\cEnd}{\mathscr{E}nd}
\newcommand{\cHom}{\mathscr{H}om}

\DeclareMathOperator{\loc}{loc}

\DeclareMathOperator{\id}{id}

\DeclareMathOperator{\Sch}{Sch}

\DeclareMathOperator{\Hilb}{Hilb}

\DeclareMathOperator{\Aut}{Aut}

\DeclareMathOperator{\Ob}{Ob}
\DeclareMathOperator{\VW}{VW}
\DeclareMathOperator{\vw}{vw}
\DeclareMathOperator{\At}{At}
\DeclareMathOperator{\Cone}{Cone}
\DeclareMathOperator{\vir}{vir}
\DeclareMathOperator{\mov}{mov}
\DeclareMathOperator{\Pic}{Pic}
\DeclareMathOperator{\vd}{vd}
\DeclareMathOperator{\Ch}{Ch}
\DeclareMathOperator{\CR}{CR}
\DeclareMathOperator{\PD}{PD}
\DeclareMathOperator{\AJ}{AJ}
\DeclareMathOperator{\Supp}{Supp}
\DeclareMathOperator{\Higg}{Higg}
\DeclareMathOperator{\instan}{instan}

\newcommand{\rk}{\mathop{\rm rk}}

\newcommand{\ev}{\mathop{\rm ev}\nolimits}

\newcommand{\tr}{\mathop{\rm tr}\nolimits}

\renewcommand{\Im}{\mathop{\rm Im}}
\renewcommand{\top}{\mathop{\rm top}}

\newcommand{\ord}{\mathop{\rm ord}\nolimits}

\newcommand{\rank}{\mathop{\rm rank}\nolimits}

\newcommand{\red}{\mathop{\rm red}\nolimits}

\newcommand{\Jac}{\mathop{\rm Jac}\nolimits}

\newcommand{\ob}{\mathop{\rm ob}}

\newcommand{\spec}{\mathop{\rm Spec}\nolimits}

\newcommand{\proj}{\mathop{\rm Proj}\nolimits}

\numberwithin{equation}{subsection}

\newcommand {\mat}      [1] {\left(\begin{array}{#1}}
\newcommand {\rix}          {\end{array}\right)}

\title[Tanaka-Thomas VW invariants via stacks]{The Tanaka-Thomas's Vafa-Witten invariants via surface Deligne-Mumford stacks}
\author{Yunfeng Jiang and Promit Kundu}
\address{Department of Mathematics\\ University of Kansas\\ 405 Snow Hall 1460 Jayhawk Blvd\\Lawrence KS 66045 USA} 
\email{y.jiang@ku.edu}

\address{Department of Mathematics\\ University of Kansas\\ 405 Snow Hall 1460 Jayhawk Blvd\\Lawrence KS 66045 USA} 
\email{kundupromit63@ku.edu}

\begin{document}
\sloppy \maketitle
\begin{abstract}
We provide a definition of Vafa-Witten invariants for projective surface Deligne-Mumford stacks, generalizing  the construction of Tanaka-Thomas on the Vafa-Witten invariants for projective surfaces  inspired by the S-duality conjecture.  We give calculations for a root stack over a  general type quintic surface, and quintic surfaces with ADE singularities.  The relationship  between  the Vafa-Witten invariants of quintic surfaces with ADE singularities and  the Vafa-Witten invariants of their crepant resolutions is also discussed. 
\end{abstract}

\maketitle

\tableofcontents

\section{Introduction}

In this paper we generalize  the Tanaka-Thomas's Vafa-Witten invariants for projective surfaces \cite{TT1}, \cite{TT2} to  two dimensional smooth Deligne-Mumford (DM) stacks.  

\subsection{Background}

The  motivation from physics is the S-duality conjecture, where by physical duality theory  Vafa and Witten  \cite{VW} predict that the generating function of the Euler characteristic of the moduli space of stable coherent sheaves on projective surfaces should be modular form.  In history the S-duality is a rich conjecture.  For instance,    the reduction of  the S-duality conjecture from projective surfaces to product of curves is given by Kapustin-Witten \cite{KW}, which is related  to the Langland duality in number theory. 
In the mathematics side, the moduli space of solutions of the Vafa-Witten equation on a projective surface $S$ has a partial compactification by Gieseker semistable Higgs pairs $(E,\phi)$ on $S$, where $E$ is a coherent sheaf with rank $\rk >0$, and $\phi\in \Hom_{S}(E, E\otimes K_S)$ is a section called a Higgs field.  

The formula in \cite{VW} (for instance Formula (5.38) of \cite{VW}) and some mathematical calculations as in \cite{GK} imply that the invariants in \cite{VW} may have other contributions except purely from the projective surfaces.  In \cite{TT1}, \cite{TT2} Tanaka and Thomas have developed  the Vafa-Witten theory using the moduli space $\N$ of Gieseker semi-stable Higgs pairs $(E,\phi)$ on $S$ with topological data $(\rk=\rank, c_1(E), c_2(E))$.   By spectral theory, the moduli space $\N$ of Gieseker semi-stable Higgs pairs $(E,\phi)$ on $S$ is isomorphic to the moduli space of Gieseker semi-stable torsion sheaves  
$\sE_\phi$ on the total space $X:=\mbox{Tot}(K_S)$.  In the case that the semistability and stability coincide,  the moduli space $\N$ admits a symmetric obstruction theory in \cite{Behrend}.  Therefore there exists a dimension zero virtual fundamental cycle 
$[\N]^{\vir}\in H_0(\N)$.  The moduli space $\N$ is not compact, but it admits a $\cc^*$-action induced by the 
$\cc^*$-action on $X$ by scaling the fibers of $X\to S$.  The $\cc^*$-fixed locus $\N^{\cc^*}$ is compact,  then from \cite{GP}, $\N^{\cc^*}$ inherits a perfect obstruction theory from $\N$ and  Tanaka-Thomas \cite{TT1} define the virtual localized invariant
\begin{equation}\label{eqn_localized_invariants_intr}
\widetilde{\VW}(S):=\widetilde{\VW}_{\rk, c_1, c_2}(S)=\int_{[\N^{\cc^*}]^{\vir}}\frac{1}{e(N^{\vir})}
\end{equation}
as the $U(\rk)$ Vafa-Witten invariants, where $N^{\vir}$ is the virtual normal bundle. 

By the property of the obstruction sheaf this invariant $\widetilde{\VW}(S)$ is zero unless 
$H^{0,1}(S)=H^{0,2}(S)=0$.  As mentioned in \cite{TT1}, most of the surfaces satisfying the condition admit a vanishing theorem such that the invariants are just the signed Euler characteristic of the moduli space of stable sheaves on $S$. 
The right invariants are defined by  the moduli space $\N_L^{\perp}$ of Higgs pairs with fixed determinant $L\in\Pic(S)$ and trace-free  $\phi$. 
Tanaka-Thomas have carefully studied the deformation and obstruction theory of the Higgs pairs instead of using the ones for the two dimensional sheaves, and constructed a {\em symmetric obstruction theory} on $\N^{\perp}_{L}$.  
The space $\N^{\perp}_{L}$ still admits a $\cc^*$-action, therefore inherits a perfect obstruction theory on the fixed locus. 
Then the Vafa-Witten invariants  are defined as:
\begin{equation}\label{eqn_localized_SU_invariants_intr}
\VW(S):=\VW_{\rk, c_1, c_2}(S)=\int_{[(\N_L^{\perp})^{\cc^*}]^{\vir}}\frac{1}{e(N^{\vir})}.
\end{equation}
This corresponds to the $SU(\rk)$ gauge group in Gauge theory.   They did explicit calculations for some surfaces of general type in \cite[\S 8]{TT1} and verified some parts of Formula (5.38) in \cite{VW}.  Since for such general type surfaces, the 
$\cc^*$-fixed loci contain components such that the Higgs fields are non-zero,  there are really contributions from the threefolds to the Vafa-Witten invariants.  This is the first time that the threefold contributions are made for the Vafa-Witten invariants. 
Some calculations and the refined version of the Vafa-Witten invariants have been studied in \cite{Thomas2}, \cite{MT}, \cite{GK}, \cite{Laarakker}. 

\subsection{Invariants using the Behrend function}

There is another way to do the localization on the moduli space $\N^{\perp}_{L}$.  In \cite{Behrend}, Behrend constructs an integer valued constructible function 
$$\nu_{\N}: \N^{\perp}_{L}\to \zz$$
called the Behrend function.  We can define 
\begin{equation}\label{eqn_localized_Behrend_invariants_intr}
\vw(S):=\vw_{\rk, c_1, c_2}(S)=\chi(\N_L^{\perp}, \nu_{\N})
\end{equation}
where $\chi(\N_L^{\perp}, \nu_{\N})$ is the weighted Euler characteristic by the Behrend function $\nu_{\N}$.  The $\cc^*$-action on $\N^{\perp}_{L}$ induces a cosection $\sigma: \Omega_{\N^{\perp}_{L}}\to \sO_{\N^{\perp}_{L}}$ in \cite{KL} by taking the dual of the associated vector field $v$ given by the $\cc^*$-action. The degenerate locus is the fixed locus 
$(\N^{\perp}_{L})^{\cc^*}$, therefore there exists a cosection localized virtual cycle 
$[\N^{\perp}_{L}]^{\vir}_{\loc}\in H_0((\N^{\perp}_{L})^{\cc^*})$, and 
$$\int_{[\N^{\perp}_{L}]^{\vir}_{\loc}}1=\chi(\N_L^{\perp}, \nu_{\N})$$
which is proved in \cite{JT}, \cite{Jiang}. 
Tanaka-Thomas prove that in the case $\deg K_{S}< 0$ and the case that $S$ is a K3 surface, 
$\VW(S)=\vw(S)$. They also prove their corresponding  generalized Vafa-Witten invariants in \cite{JS} agree, see \cite{TT2} for the Fano case and  \cite{MT} for the K3 surface case. 

\subsection{Motivation for surface DM stacks}

The S-duality conjecture is interesting for surfaces with orbifold singularities.  In \cite{VW}, Vafa-Witten discussed the ALE spaces, which are hyperkahler four manifolds and crepant resolutions of the orbifold $\cc^2/G$ with $G\subset SU(2)$ a finite subgroup in $SU(2)$ corresponding to ADE Dynkin diagrams.  It is interesting to directly study the Vafa-Witten invariants for such two dimensional DM stacks. 

On the other hand, the crepant resolution conjecture in both Gromov-Witten theory  \cite{Ruan}, \cite{Bryan-Graber} and Donaldson-Thomas theory \cite{Bryan-Young} has been attracting a lot of interests in the past ten years.  Given a local orbifold $\sX=[\cc^3/G]$ for $G\subset SU(2)$, Bryan-Young formulated the crepant resolution conjecture for the generating function of the Donaldson-Thomas invariants for $\sX$ and  their crepant resolutions which are $G$-Hilbert schemes.  The conjecture can be explained as wall crossing formula of the counting invariants in the derived category of coherent sheaves corresponding to different Bridgeland stability conditions, see \cite{Calabrese}, \cite{BD}, \cite{Toda2}, \cite{Jiang2}.
It is interesting to study the Vafa-Witten invariants and the generalized Vafa-Witten invariants  for orbifolds via wall crossing techniques of Joyce-Song and Bridgeland.  In such a situation, the invariants $\vw$ seem to be more suitable to be put into the wall crossing formula, and we leave it as a future work. 

Finally our study of the Vafa-Witten invariants for surface DM stacks is  motivated by the S-duality conjecture and Langlands duality in \cite{KW}. In  Tanaka-Thomas \cite{TT1}, \cite{TT2},  the gauge group is
$SU(\rk)$.  Since the Langlands dual group of $SU(\rk)$ is $SU(\rk)/\zz_{\rk}$, one hopes that the study of the Vafa-Witten invariants for orbifold surfaces will be related to  the $SU(\rk)/\zz_{\rk}$-Vafa-Witten invariants.  In particular, one has $SU(2)/\zz_2\cong SO(3)$, so the $SU(2)$-Vafa-Witten invariants for a global quotient surface DM stack 
$[S/\zz_2]$ is related to  the $SO(3)$-Vafa-Witten invariants for $S$.  This should be related to the $SO(3)$-Donaldson invariants for the surface $S$, see \cite{KM}, \cite{Gottsche}, \cite{MW}.

\subsection{The moduli space of Higgs pairs on surface DM stacks}

Let $\sS$ be a smooth two dimensional DM stack with projective coarse moduli space $S$, which we call it  ``a surface DM stack".   The moduli space of stable coherent sheaves with a fixed Hilbert polynomial $H\in \qq[m]$ has been constructed by F. Nironi \cite{Nironi}.  In \cite{Nironi}, in order to define suitable Hilbert polynomials the author  chooses a generating sheaf $\Xi$ for $\sS$ and defines the {\em modified} Hilbert polynomial associated with this generating sheaf.  Let $p: \sS\to S$ be the morphism to its coarse moduli space.  A locally free sheaf $\Xi$ on $\sS$ is $p$-very ample if for every geometric point of $\sS$ the representation of the stabilizer group at that point contains every irreducible representation. We define  
$\Xi$ as a {\em generating sheaf} of $\sS$.  
The readers may understand that the generating sheaf is introduced to deal with some stacky issues and in order not to loose stacky information like finite group gerbes over schemes. 

Let us fix a polarization $\sO_S(1)$ on the coarse moduli space $S$. 
Choose a generating sheaf $\Xi$, and for a coherent sheaf $E$ on $\sS$, the modified Hilbert polynomial is defined by:
$$H_{\Xi}(E,m) = \chi(\sS,E\otimes \Xi^{\vee}\otimes p^*\sO_S(m)).$$
Then we can write down 
$$H_{\Xi}(E,m) =\sum_{i=0}^d\alpha_{\Xi, i}\frac{m^i}{i!},$$
where $d=\dim(E)$ is the dimension of the sheaf $E$. 
The reduced Hilbert polynomial for pure sheaves, which 
we will denote by $h_{\Xi}(E)$;  is the monic polynomial $\frac{H_{\Xi}(E)}{\alpha_{\Xi, d}}$ with rational coefficients. 
Then let $E$ be a pure coherent sheaf, it is semistable if for every proper subsheaf
$F \subset E$ we have  $h_{\Xi}(F) \leq h_{\Xi}(E)$ and it is stable if the same is true with a strict inequality.
Fixing a modified Hilbert polynomial $H$, the moduli stack of semistable coherent sheaves 
 on $\sS$ is constructed in \cite{Nironi}.  If the stability and semistability coincide, the coarse moduli space 
$\rM:=\rM^{\Xi}_H(\sS)$ is a projective scheme. 

Our goal is to study the moduli stack of Higgs pairs $(E,\phi)$ on $\sS$, where $E$ is a torsion free coherent sheaf with 
rank $\rk >0$ and $\phi\in \Hom(E, E\otimes K_\sS)$ is a Higgs field.  The Hilbert polynomial $h_{\Xi}(E)$ can be similarly defined by choosing a generating sheaf $\Xi$ for $\sS$.  Then $(E,\phi)$  is semistable if for every proper $\phi$-invariant subsheaf
$F \subset E$ we have  $h_{\Xi}(F) \leq h_{\Xi}(E)$.  The moduli stack of stable Higgs pairs on $\sS$ with modified Hilbert polynomial $H$ has a coarse moduli space 
$\N:=\N_H$ which is a quasi-projective scheme.  

Let $\sX:=\mbox{Tot}(K_{\sS})$ be the total space of the  canonical line bundle of $\sS$, then $\sX$ is a smooth Calabi-Yau threefold DM stack.  
By spectral theory, the category of Higgs pairs on $\sS$ is equivalent to the category of torsion sheaves $\sE_\phi$  on $\sX$ supporting on $\sS\subset \sX$.
Let $\pi: \sX\to \sS$ be the projection, then  the bullback $\pi^*\Xi$ is a generating sheaf for $\sX$.  One can take a projectivization $\overline{\sX}=\proj (K_{\sS}\oplus\sO_{\sS})$, and consider the moduli space of stable torsion sheaves 
on $\overline{\sX}$ with modified Hilbert polynomial $H$. The  part that is supported on   $\sX$ is isomorphic to the moduli space of stable Higgs pairs $\N$ on $\sS$ with modified Hilbert polynomial $H$. 

\subsection{Perfect obstruction theory and the Vafa-Witten invariants}

To construct the perfect obstruction theory,  we also take the moduli space $\N^{\perp}_{L}$ of stable Higgs pairs 
$(E,\phi)$ with fixed determinant $L$ and trace-free  $\phi$.   For a Higgs pair $(E,\phi)$ on $\sS$, and the associated torsion sheaf 
$\sE_\phi$ on $\sX$, the deformation and obstruction of $\sE_\phi$ are controlled by 
$$\Ext^1_{\sX}(\sE_{\phi}, \sE_{\phi}), \quad \Ext^2_{\sX}(\sE_{\phi}, \sE_{\phi})$$
respectively; and also we have an exact triangle 
$$R\cHom(\sE_\phi,\sE_\phi)\to R\cHom_{\sS}(E, E)\stackrel{\circ\phi-\phi\circ}{\longrightarrow}R\cHom_{\sS}(E\otimes K_{\sS}^{-1},E)$$
relating the deformation and obstruction theory of $\sE_{\phi}$ to the Higgs pairs $(E,\phi)$.  We found that all the arguments as in 
\cite[\S 5]{TT1} work for smooth DM stacks $\sS$. We write down some parts in the appendix for $\sS$, and for  the more precise details we refer to \cite[\S 5]{TT1}. 
Therefore  $\N^{\perp}_{L}$  admits a symmetric obstruction theory and the $\cc^*$ acts on $\N^{\perp}_{L}$ with compact fixed loci,   we define 
\begin{equation}\label{eqn_localized_SU_DM_invariants_intr}
\VW^L_{H}(\sS)=\int_{[(\N_L^{\perp})^{\cc^*}]^{\vir}}\frac{1}{e(N^{\vir})}
\end{equation}
as the $SU(\rk)$-Vafa-Witten invariants. 
We have the Behrend function in this case 
and define
\begin{equation}\label{eqn_localized_SU_DM_invariants_vw_intr}
\vw^L_{H}(\sS)=\chi(\N^{\perp}_{L}, \nu_{\N^{\perp}_{L}})
\end{equation}
as the weighted Euler characteristic. 

\subsection{Calculations}

For surface DM stacks, the essential part is to calculate the Vafa-Witten invariants for some examples to see if one can get different phenomenon comparing with the case of  smooth projective surfaces.  First the moduli space $\N_L^{\perp}$ admits a $\cc^*$-action induced by the $\cc^*$-action on the total space $\sX$ of the canonical line bundle $K_{\sS}$.  There are two type of $\cc^*$-fixed loci on $\N_L^{\perp}$ .  The first one corresponds to the $\cc^*$-fixed Higgs pairs $(E,\phi)$ such that the Higgs fields $\phi=0$.  The fixed locus is just the moduli space $\rM_{L}(\sS)$ of stable torsion free sheaves $E$ on $\sS$. This is called the {\em Instanton Branch} as in \cite{TT1}.  The second type corresponds to $\cc^*$-fixed Higgs pairs $(E,\phi)$ such that the Higgs fields $\phi\neq 0$.  This case mostly happens when the surfaces $\sS$ are general type, and this component is called the {\em Monopole} branch.  See \S \ref{subsec_CStar_fixed_locus} for more details.

We include a short calculation  for 
$$\pp:=\pp(1,1,2),$$
the weighted projective plane with only one stacky point $[0,0,1]\in \pp(1,1,2)$. The  local orbifold structure around this point is given by type $A_1$ singularity 
$[\cc^2/\mu_2]$. The inertia stack $I\pp=\pp\sqcup \pp(2)$, where $\pp(2)=B\mu_2$ corresponds to the nontrivial element $\zeta\in \mu_2$. We can choose generating sheaf 
$\Xi=\sO_{\pp}\oplus \sO_{\pp}(1)$.  The moduli space of stable torsion free sheaves can be studied by toric method in \cite{GJK}. 
In this case the vanishing theorem as in Proposition \ref{prop_K_S_fixed_locus} shows that $\VW(\pp)$ is just the signed virtual Euler number.  Moreover since $K_{\pp}<0$, the obstruction sheaf $\Ext^2(E,E)_0=0$ and the moduli space 
$\rM_L$ is smooth,  so $\VW(\pp)=\vw(\pp)=(-1)^{\dim T_{\rM_L}}\chi(\rM_L)$ is the signed Euler number.

When fixing a $K$-group class in $K_0(\pp)$, the Hilbert polynomial is fixed.  We introduce orbifold Chern character 
$\widetilde{\Ch}: K_0(\pp)\to H^*(I\pp)$ by taking values in the Chen-Ruan orbifold cohomology and let $\widetilde{\Ch}=(\widetilde{\Ch}_1, \widetilde{\Ch}_\zeta)$ be the components.
We use the notation of codegree 
$$\left(\widetilde{\Ch}_g\right)^k:=\left(\widetilde{\Ch}_g\right)_{\dim \pp_{g}-k}\in H^{\dim \pp_g-k}(\pp_g),$$
as in (\ref{eqn_Chern_character_degree}). Here $\pp_g=\pp$ if $g=1$ and $\pp_g=B\mu_2$ if $g=\zeta$.
So if we fix $\left(\widetilde{\Ch}_1\right)^2=2$, the rank, $\left(\widetilde{\Ch}_1\right)^1=c_1(\pp)$ the first Chern class, and let $q_1$ be the variable keeping track of $E$ such that  $\left(\widetilde{\Ch}_1\right)^0=c_2(E)$, and $q_2$ the variable keeping track of 
$\left(\widetilde{\Ch}_\zeta\right)^0$.  From Example 2 of \cite[\S 7.2]{GJK}, 
$\left(\widetilde{\Ch}_\zeta\right)^0$ only takes values $-2, 0, 2$, we choose $0$. 
The tangent dimension 
$\dim \Ext^1_{\pp}(E,E)$ is $4c_2(E)-c_1^2-3=4n-c_1^2-3$. 
For simplicity we let $q_1=q, q_2=1$, and  let $\VW_n(\pp)=(-1)^{-c_1^2-3}\chi(\rM^n_{L})$ be the Vafa-Witten invariant for the moduli space of stable rank $2$ torsion free sheaves $E$ on $\pp$ with $c_2(E)=n$. Then from Lemma 7.4, Example 1 in 
\cite[\S 7.1]{GJK},  Example 2 of \cite[\S 7.2]{GJK},  up to a sign we have:
\begin{prop}
We have the following generating series for the rank two Vafa-Witten invariants for $\pp$:
\begin{equation}\label{eqn_VW_P112}
\sum_{n\geq 0}\VW_n(\pp) q^n=
\left(\frac{q^{\frac{1}{6}}}{\eta(q)^4}\theta_3(q)\right)^2\cdot \sum_{(w_1, w_2, w_3)\in C_{c_1}}q^{\frac{1}{4}c_1^2+\frac{1}{4}\sum_{i=1}^3w_i^2-\frac{1}{2}\sum_{1\leq i\leq j\leq 3}w_iw_j},
\end{equation}
where 
$$C_{c_1}:=\left\{(w_1,w_2,w_3)\in \zz_{>0}^3 :   2| c_1+\sum_{i=1}^3w_i; 2| w_2; w_i<w_j+w_k\right\}, \forall \{i,j,k\}=\{1,2,3\};$$
and here
$$\frac{q^{\frac{1}{6}}}{\eta(q)^4}\theta_3(q)=\frac{1}{\prod_{k>0}(1-q^k)^4}\sum_{k\in\zz}q^{k^2},$$
and $\theta_3$ is the Jacobi theta function and $\eta(q)$ is the Dedekind eta function. 
\end{prop}

From \cite{GJK}, the second part of 
(\ref{eqn_VW_P112}) is a holomorphic part of a modular form. For any $\Delta > 0$, let $H(\Delta)$ be the number of (equivalence classes of) positive definite integer binary quadratic forms $AX^2 + BXY + CY^2$ with
discriminant $B^2-4AC =-\Delta$ and weighted by the size by the following:  forms equivalent to $\lambda(X^2 + Y^2)$ and $\lambda(X^2 + XY + Y^2)$ are counted with weights $1/2$ and $1/3$ respectively. Then 
Theorem 1.2 of \cite{GJK} tells us that the sum in (\ref{eqn_VW_P112}) is: 
$$\sum_{n=1}^{\infty}2H(8n-1) q^{\frac{1}{4} - 2n}$$
if $c_1$ is odd and 
$$\sum_{n=1}^{\infty}(H(4n) + 2H(n)- \frac{1}{2}\sigma_0(n))q^{-n}- \sum_{n=1}^{\infty}\sigma_0(n)q^{-4n}$$
if $n$ is even, where $\sigma_0(n)$ is the sum of divisors function.  They are the holomorphic part of a modular form.  

We also perform calculations in the rank $\rk=2$ case for an $r$-th root stack $\sS$ over a general type quintic surface $S$ along a smooth divisor 
$C\in |K_{S}|$;  and quintic surfaces $\sS$ with ADE isolated singularities.  The calculations are for the second type $\cc^*$-fixed components such that in the fixed loci $(E,\phi)\in \rM^{(2)}$ the Higgs field $\phi\neq 0$.  The fixed loci $\rM^{(2)}$ is expressed as union of nested Hilbert schemes, see Proposition \ref{prop_second_fixed_loci_Hilbert_scheme}, and Proposition \ref{prop_second_fixed_loci_Hilbert_scheme_isolated}. 

In the first special case that one ideal sheaf is the structure sheaf, the nested Hilbert scheme is just the Hilbert scheme of points on the surface DM stack $\sS$ with fixed $K$-group class $\mathbf{c}\in K_0(\sS)$.  We do the calculations for such Hilbert schemes on orbifolds and get similar  results as in \cite[Proposition 8.22,  \S 8.5]{TT1}, see Theorem \ref{thm_generating_function_root_quintic} and Theorem \ref{thm_generating_function_ADE_quintic}.

Let $\sS$ be a quintic orbifold surface with isolated ADE singularities $P_1,\cdots, P_s$. 
In this case we get the same result as in \cite[\S 8.5]{TT1}.
The reason that the results from the component of Hilbert scheme of points  for such $\sS$ are the same as in 
 \cite[\S 8.5]{TT1}  may be explained as deformation invariance since smooth quintic surfaces can be deformed into quintic surfaces with ADE isolated singularities.  
The canonical line bundle $K_{\sS}$ for $\sS$  will have trivial $G_i$ representation around any isolated ADE singularity point 
$P_i$ so that locally $\sS$ is isomorphic to $[\cc^2/G_i]$.  This can be taken as 
the other  reason for the equality of the results  as in  \cite[\S 8.5]{TT1} since the canonical divisor curve $C\in |K_{\sS}|$ will not touch the isolated singular points $P_1,\cdots, P_s$.
We also do a second special case calculations for the nested Hilbert schemes for $\sS$ with ADE isolated singularities, and this time the result of the integral  is different from \cite[\S 8.7]{TT1}.

Finally we discuss the relation of the Vafa-Witten invariants of a quintic orbifold surface with ADE singularities and their  {\em crepant resolutions}. 

\subsection{Related and future research}

As we have already pointed out, our study for the Vafa-Witten invariants for surface DM stacks is definitely motivated by the beautiful work of Tanaka-Thomas \cite{TT1}, \cite{TT2}.  Our original motivation is to expect that surface DM stacks can give different phenomenon for the invariants. 
As mentioned earlier, it is interesting to define generalized invariants $\vw$ as in \cite{TT2} and study the wall-crossing formula of Joyce-Song, and we leave this as a future research. 
Special examples of surface DM stacks are cyclic gerbes over smooth surfaces, and in  \cite{Jiang_2019} Jiang  has developed a theory of twisted Vafa-Witten invariants  using gerbe twisted sheaves and proved the S-duality conjecture of Vafa-Witten for K3 surfaces.   

Let $S$ be a projective surface with isolated ADE type singularities.  It is interesting to calculate the Vafa-Witten invariants using the techniques developed in \cite{GT}, and compare with the Vafa-Witten invariants with its crepant resolution $\widetilde{S}$. 
In particular, if $S$ is a smooth projective K3 surface, and let $[S/\mu_N]$ be the quotient stack, where $\mu_N$ is a cyclic group of order $N$ acting as symplectic automorphisms on $S$.  From \cite{Huybrechts}, the action only has finite isolated rational double points singularities and its crepant resolution is a smooth K3 surface. It is interesting to study the Vafa-Witten invariants of the quotient stack $[S/\mu_N]$ and compare with the Vafa-Witten invariants for $S$ calculated in \cite{TT2}. 
We leave these as future projects too. 

We want to make an extra effort to the $r$-th root stacks on surfaces.  Let $(S,D)$ be a pair, where $D\subset S$ be a smooth divisor curve.  Let $\sS:=\sqrt[r]{(S,D)}$ be the  $r$-th root stack and let $p: \sS\to S$ be the map to its coarse moduli space.  Let $\sD=p^{-1}(D)$, and we choose a generating sheaf $\Xi=\oplus_{i=0}^{r}\sO_{\sS}(\sD^{\frac{i}{r}})$.
The moduli space of stable sheaves $E$ on $\sS$ with a Hilbert polynomial $H$ is isomorphic to the moduli space of parabolic stable sheaves $(E_\bullet)$ on $(S,D)$ constructed by  Maruyama and Yokogawa in \cite{MY}.  In 
\cite{Jiang3} we expect that the same result holds for moduli space of Higgs pairs $(E,\phi)$.  Parabolic stable sheaves on $(S,D)$ has some relations with the S-duality conjecture, which has been studied by Kapranov \cite{Kapranov}, and we hope to put these ideas into the Vafa-Witten theory to check  the S-duality results of Kapranov.

\subsection{Outline} 
The paper is organized as follows.  We review the construction of the moduli space of semistable sheaves and the moduli space of Higgs sheaves on a surface DM stack $\sS$  in \S \ref{sec_moduli_space}, where in \S \ref{subsec_surface_DM} we recall the basic properties of surface DM stacks and give several interesting examples;  and in 
\S \ref{subsec_moduli_Higgs_pairs} we study the moduli construction of stable sheaves and stable Higgs pairs on a smooth projective DM stacks $\sS$. 
We work on the deformation theory of Higgs pairs on $\sS$ in \S \ref{sec_VW} and define the Vafa-Witten invariants $\VW$ on a surface DM stack $\sS$ by the obstruction theory constructed in the Appendix.   
In \S \ref{sec_calculations} we do the main calculations on the surface DM stacks, where in \S \ref{subsec_root_stacks} we calculate the case of the $r$-th root stack $\sS$ over a smooth quintic surface $S$; and in \ref{subsec_quintic_ADE} we deal with the quintic surfaces with ADE singularities. 

\subsection{Convention}
We work over $\cc$ throughout of the paper.  For a surface DM stack $\sS$, we mean a smooth two dimensional Deligne-Mumford stack over $\cc$. 
Let us fix some notations.   We always use Roman letter $E$ to represent a coherent sheaf on a projective DM stack or a surface DM stack $\sS$, and use curl latter $\sE$ to represent the sheaves on the total space Tot$(\sL)$ of a line bundle $\sL$ over $\sS$. 

We reserve {\em $\rk$} for the rank of the torsion free coherent sheaves $E$, and use $\sqrt[r]{(S,C)}$ for the $r$-th root stack associated with the pair $(S, C)$ for a smooth projective surface and $C\subset S$ a smooth connected divisor.

\subsection*{Acknowledgments}

Y. J. would like to thank Huai-Liang Chang, Amin Gholampour,  Martijn Kool, Richard Thomas and Hsian-Hua Tseng for valuable discussions on the Vafa-Witten invariants.    
The first author would like to thank Hong Kong University of Science and Technology for hospitality where part of the work is done. 
This work is partially supported by  NSF DMS-1600997.


\section{Moduli space of semistable Higgs pairs on surface DM stacks}\label{sec_moduli_space}

\subsection{Surface DM stacks, examples}\label{subsec_surface_DM}
The basic knowledge of stacks can be found in the book \cite{LMB}.  

Let $\sX$ be a smooth DM stack. Roughly speaking the stack $\sX$ is a fibre  category  by groupoids.  An interesting DM stack is the global quotient stack $[M/G]$, where $M$ is a smooth scheme and $G$ is a finite group acting on $M$. The quotient stack $[M/G]$ classifies principal $G$-bundles over the category of schemes, together with a $G$-equivariant morphism to $M$.  Not every stack is a global quotient, but most interesting DM stacks are locally global quotient, i.e., for any point $x\in\sX$, there exists an open \'etale neighborhood $x\in \sU$ such that $\sU\cong [M/G]$ is a global quotient.  We mainly work on locally quotient DM stacks. 

Let $f: \sX\to \yY$ be a morphism between two DM stacks. $f$ is called ``representable" if every morphism $g: S\to \yY$ from a scheme $S$, the fibre product $S\times_{g, \yY, f}\yY$ is a scheme. In particular, any morphism from a scheme to $\yY$ is representable. 

\begin{defn}\label{defn_inertia_stack}
Let $\sX$ be a smooth DM stack. The inertia stack $I\sX$ associated to $\sX$ is defined to be the fibre product:
$$I\sX:=\sX\times_{\Delta, \sX\times \sX}\sX$$
where $\Delta: \sX\to \sX\times\sX$ is the diagonal morphism.
\end{defn}

The objects in the category underlying $I\sX$ is: 
$$\Ob(I\sX)=\{(x,g)| x\in \Ob(\sX), g\in \Aut_{\sX}(x)\}.$$
\begin{rmk}
\begin{enumerate}
\item There exists a morphism $q: I\sX\to \sX$ given by $(x,g)\mapsto x$;
\item There exists a decomposition 
$$I\sX=\bigsqcup_{r\in\nn}\Hom(\mbox{Rep}(B\mu_r, \sX)),$$
hence $I\sX$ can be decomposed into connected components, and 
$$I\sX=\bigsqcup_{i\in\mathcal{I}}\sX_i$$
Here $\mathcal{I}$ is the index set.  We denote the component $\{(x, \id)|x\in \Ob(\sX), \id\in \Aut_{\sX}(\sX)\}$ by $\sX_0=\sX$. 
\end{enumerate}
\end{rmk}
For instance let $\sX=[M/G]$ be the global quotient DM stack.  Then $\mathcal{I}=\{(g)| g\in G\}$, where $(g)$ is the conjugacy class.  The inertia stack $I\sX=\bigsqcup_{(g)}\sX_{(g)}$, and $\sX_{(g)}=[M^g/C(g)]$, where $M^g$ is the $g$ fixed locus of $M$ and $C(g)$ is the centralizer of $g$. \\

\textbf{Surface DM stack examples}
Throughout of the paper we let $\sS$ be a smooth two dimensional DM stack, which we call a ``surface DM stack". 
We review some interesting examples of surface DM stacks. 

\subsubsection{Orbifold K3 surfaces}

An example of projective smooth K3 surface $S$ is the smooth quartic surface in $\pp^3$.  Let $\zz_2$ or $\zz_n$ act on 
$\pp^3$ by $\zeta[x_0:x_1:x_2:x_3]=[x_0:\zeta\cdot x_1, \zeta\cdot x_2:\zeta\cdot x_3]$.  Then $\sS=[S/\zz_2]$ is an orbifold K3 surface. 

Another interesting orbifold K3 surface is orbifold Kummer surface $\sS=[\mathbb{T}^4/\zz_2]$, where 
$\mathbb{T}^4$ is the $4$-dimensional real torus and $\tau\in \zz_2$ acts on $\mathbb{T}^4$ by 
$$\tau(e^{it_1}, e^{it_2}, e^{it_3}, e^{it_4})=(e^{-it_1}, e^{-it_2}, e^{-it_3}, e^{-it_4}).$$
This is a complex two dimensional orbifold K3 surface with orbifold fundamental group $\zz^4\rtimes \zz_2$, the semi-direct product.  
This orbifold Kummer surface has $16$ isolated orbifold points.  After taking resolution $\widetilde{S}\to \overline{S}$ there are $16$ exceptional $(-2)$-curves on the smooth K3 surface $\widetilde{S}$.
Of course the finite group $\zz_2$ can be placed by $\zz_n$, $D_n$, $E_6, E_7, E_8$ to get ADE singularities on $\sS$.  Orbifold Kummer surfaces and their resolutions have applications in mirror symmetry, see \cite{NW}.

\subsubsection{Automorphism of K3 surfaces}

Let $S$ be a smooth projective K3 surface with a holomorphic symplectic form 
$\sigma\in H^0(S, \Omega_S^2)$. An automorphism $g$ of $S$ is called a {\em symplectic automorphism} if $\sigma$ preserves the symplectic form of $S$, i.e. $g^*\sigma=\sigma$.  Our reference is Daniel Huybrechts's book \cite{Huybrechts}.  The classification of finite order symplectic automorphisms of $S$ gives abelian group actions on  
$S$.  Let $G$ be such a finite abelian group, then the quotient stack $\sS:=[S/G]$ is an orbifold K3 surface, and from \cite{Huybrechts},  there are only finite isolated stacky points in $\sS$, which are rational double points.  
A list of the number of fixed points and the corresponding finite abelian groups can be found in 
\cite[\S 15.1]{Huybrechts}. 
For instance, in the case of cyclic group $\mu_2$ of order two, the orbifold K3 surface $[S/\mu_2]$ is called the 
{\em Nikulin involution}. There are eight isolated $A_1$-singular points. 

The global Torelli theorem tells us that  the symplectic automorphism $g$ is uniquely determined by its induced action on $H^2(S,\zz)$. By  \cite{Huybrechts} the action of $g$ on the abstract lattice $H^2(S,\zz)$ depends up to an orthogonal transformation of the lattice only on the order $|g|$. 
Let $U=\mat{cc}0&1\\1&0\rix $ denote by the hyperbolic lattice. Recall that
$$
\Lambda:=H^2 (S , \zz)=U^3\oplus E_8(-1)^2.$$
The invariant lattice with respect to $g$ is
$$\Lambda^g =\{v\in \Lambda |g\cdot v=v\}.$$
The coinvariant lattice of $g$ is the orthogonal complement of the invariant lattice:
$$\Lambda_g =(\Lambda^g)^{\perp}\subset H^2(S,\zz).$$
In the case of  Nikulin involution $\sS:=[S/\mu_2]$.  The action of the generator $g\in \mu_2$ on $\Lambda$ is trivial on $U^{3}$ and interchanges the two copies of $E_8(-1)$. The invariant and co-invariant lattices are
$$\Lambda^g = U^3 \oplus E_8(-2), \Lambda_g = E_8(-2),$$
where  $E_8(-2)$ represents for the diagonal and the anti-diagonal in $E_8(-1)^2$ respectively. 
These invariants of lattices play an important role for the study of Donaldson-Thomas invariants of 
$(S\times E)/G$ in \cite{BO}, where $E$ is an elliptic curve with trivial $G$ action, and the authors call it a CHL Calabi-Yau threefold. 
We hope that the  invariant lattice and co-invariant lattice play a role in the calculation of Vafa-Witten invariants.

\subsubsection{Quintic surfaces with ADE singularities}\label{subsec_quintic_surface}

Let $S\subset \pp^3$ be a smooth quintic surface.  Then $S$ is a Horikawa surface with invariants
\begin{equation}\label{eqn_invariants_Horikava_surface}
\begin{cases}
p_g=h^0(K_S)=4;\\
q=h^1(\sO_S)=0;\\
c_1(S)^2=5
\end{cases}
\end{equation}
These are simple cases of surface of general type. The moduli space of Horikawa surfaces with invariants (\ref{eqn_invariants_Horikava_surface}) is very complicated.  From \cite{Gallardo}, \cite{Horikava}, the moduli space of quintic surfaces forms an irreducible component of the moduli space of general type surfaces with invariants (\ref{eqn_invariants_Horikava_surface}), where there are two irreducible components for the moduli space. The complete KSBA \cite{KSB}, \cite{Alexeev} moduli of the quintic surface is still unknown, see \cite{Rana} for some construction of boundary divisors. The component containing smooth quintic surfaces in $\pp^3$ has some quintic surfaces with at worst ordinary double point singularities, which are classified by ADE-singularities:
\begin{equation}\label{eqn_ADE_singularities}
\begin{cases}
A_n: x^2+y^2+z^{n+1}=0 (n\geq 1);\\
D_n: x^2+y(y^{n-2}+z^2)=0. (n\geq 4);\\
E_6: x^2+y^3+z^4=0;\\
E_7:  x^2+y^3+z^3y=0;\\
E_8:  x^2+y^3+z^5=0.
\end{cases}
\end{equation}
We can take such quintic surfaces  as DM stacks, with the singular points $p$ by $ADE$ type finite subgroups in $SU(2)$ acting on the local neighborhood $\cc^2$ around $p$.  For instance, let $(\sS, p_1, \cdots, p_l)$ be a surface DM stack, such that around the point $p_i$, we have $A_{n_i}$-type singularity, i.e., there exists an open neighborhood $\sU_{p_i}\subset \sS$ and $\sU_{p_i}\cong [\cc^2/\mu_{n_i+1}]$, where $\mu_{n_i+1}$ acts on $\cc^2$ by
$$\lambda\cdot (x,y)=(\lambda x, \lambda^{n_i}y).$$
Then $(\sS, p_1, \cdots, p_l)$ is a surface DM stack with stacky points $p_1,\cdots, p_l$.  As in \cite{Toda}, \cite{GNS}, the Hilbert scheme of points on such surface DM stacks $\sS$ are studied and the generating function of the Euler characteristic of the Hilbert scheme has been calculated directly in \cite{GNS}, and by wall crossing in \cite{Toda} motivated by S-duality.  We will see later in \cite{Jiang3}, that the their calculations can go into the invariants 
$\vw$ for such DM stacks.  Smooth quintic surfaces provide examples in the last section of \cite{TT1}, where they did explicit calculations for the Vafa-Witten invariants.

\subsubsection{Weighted projective planes}\label{subsec_weighted_projective_plane}

Let $\pp(a,b,c):=[\cc^3\setminus\{0\}/\cc^*]$ be the quotient stack, where $\cc^*$ acts on $\cc^3$ by 
$$\lambda[x:y:z]=[\lambda^a x: \lambda^b y: \lambda^c z].$$
Then $\pp(a,b,c)$ is the weighted projective stack.  These are interesting surface DM stacks such that their canonical line bundles are negative. For instance $\pp(2,2,2)$ is a $\mu_2$-gerbe over $\pp^2$ with canonical line bundle $\sO(-6)$. 

The moduli space of stable torsion free sheaves on $\pp(a,b,c)$ has been studied in \cite{GJK}.  The formula calculated there is also related to the S-duality, therefore to the Vafa-Witten invariants of such DM stacks. 

\subsubsection{Root stacks}

Root stacks provide a class of interesting DM stacks whose construction is obtained by taking the roots of line bundles with sections.  In \cite{Jiang3} we will study  the small Vafa-Witten invariants for root stacks and their relations to the S-duality conjecture. 

Let $S$ be a smooth projective surface, and $D\subset S$ be a smooth or simple normal crossing divisor. 
Fix an integer $r\geq 1$, the root stack $\sS:=\sqrt[r]{(S,D)}$ is constructed in \cite{Cadman}. 
Let $\sO_S(D)$ be the line bundle associated with $X$. Recall that there is an equivalence of categories between the category of line bundles over $S$ and the category of morphisms $S\to B\Gm$. Also there is an equivalence between the category of $(L, s)$ with $L$ a line bundle on $S$ and $s$ a global section on $L$, and the category of morphisms
$$S\to [\aaa^1/\Gm]$$
where $\Gm$ acts on $\aaa^1$ by multiplication, see \cite[Example 5.13]{OS03}. 

Then the line bundle $(\sO_S(D), s_D)$ defines a morphism 
$$S\to  [\aaa^1/\Gm].$$
Let $\Theta_r:  [\aaa^1/\Gm]\to  [\aaa^1/\Gm]$ be the morphism of stacks given by 
$$x\in \aaa^1\mapsto x^r\in \aaa^1; \quad  t\in \Gm\mapsto t^r\in \Gm,$$
which sends $(\sO_S(D), s_D)$ to $(\sO_S(D)^{\otimes r}, s^r_D)$.
\begin{defn}(\cite{Cadman})\label{defn_root_stack}
Let $\sS:=\sqrt[r]{(S,D)}$ be the stack obtained by the fibre product 
\[
\xymatrix{
\sqrt[r]{(S,D)}\ar[r]\ar[d]_{\pi}& [\aaa^1/\Gm]\ar[d]^{\Theta_r}\\
S\ar[r]^--{(\sO_S(D), s_D)}& [\aaa^1/\Gm].
}
\]
We call $\sS=\sqrt[r]{(S,D)}$ the root stack obtained from $S$ by the $r$-th root construction. 
\end{defn}
\begin{rmk}
$\sS=\sqrt[r]{(S,D)}$ is a smooth DM stack with stacky locus $\sD:=\pi^{-1}(D)$, and $\sD\to D$ is a $\mu_r$-gerbe over $D$ coming from the line bundle $\sO_S(D)|_{D}$. 

For example, the weighted projective stack $\pp(1, r, r)$ in Section \ref{subsec_weighted_projective_plane} is a root stack by taking the $r$-th root construction on $\pp^2$ with divisor $\pp^1\subset \pp^2$. 
\end{rmk}

\subsection{Moduli space of semistable sheaves on surface DM stacks}\label{subsec_moduli_Higgs_pairs}

In this section we review the construction of the moduli space of semistable sheaves and the moduli space of stable Higgs sheaves $(E,\phi)$ on a projective DM stack $\sS$. The construction of the moduli space was studied by F. Nironi \cite{Nironi}. 

\subsubsection{Generating sheaf and the modified Hilbert polynomial}

Let $S$ be a smooth projective scheme.  Recall that to define the Hilbert polynomial $H$ on $S$ we need a polarization $\sO_S(1)$. In the case of a smooth DM stack $\sS$, one can choose the polarization $\sO_S(1)$ on its coarse moduli space $p: \sS\to S$. However, there are no very ample invertible sheaves on a stack.  We need to loose the condition to choose some locally free sheaves on $\sS$ so that they behave like very ample sheaves. This is the notion of {\em generating sheaves} in \cite{OS03}. 

\begin{defn}\label{defn_very_ample}
A locally free sheaf $\Xi$ on $\sS$ is $p$-very ample if  for every geometric point of $\sS$ the representation of the stabilizer group at that point contains every irreducible representation of the stabilizer group. 
\end{defn}

\begin{defn}\label{defn_functors}
Let $\Xi$ be a locally free sheaf  on $\sS$. We define a functor
$$F_{\Xi}: D\Coh_{\sS}\to D\Coh_{S}$$ by
$$F\mapsto p_*\cHom_{\sO_S}(\Xi, F)$$
and a functor 
$$G_{\Xi}: D\Coh_{S}\to D\Coh_{\sS}$$ by
$$F\mapsto p^*F\otimes \Xi.$$
\end{defn}
\begin{rmk}
As in \cite{Nironi}, \cite{OS03}, the functor $F_{\Xi}$ is exact since the dual $\Xi^{\vee}$ is locally free and the pushforward $p_*$ is exact.  The functor $G_{\Xi}$ is not exact unless $p$ is flat. For instance, if $p$ is a flat gerbe or a root stack, it is flat. 
\end{rmk}

\begin{defn}\label{defn_generator}
A locally free sheaf $\Xi$ is said to be a generator for the quasi-coherent sheaf $F$ if the left adjunction morphism of the identity $p_*F\otimes \Xi^{\vee}\stackrel{\id}{\longrightarrow}p_*F\otimes \Xi^{\vee}$
$$\theta_{\Xi}(F): p^*p_*\cHom_{\sO_{\sS}}(\Xi, F)\otimes \Xi^{\vee}\to F$$ 
is surjective. The sheaf $\Xi$ is a {\em generating sheaf} of $\sS$ if it is a generator for every quasi-coherent sheaf on $\sS$. 
\end{defn}

\begin{prop}(\cite[\S 5.2]{OS03})
A locally sheaf $\Xi$ on a DM stack $\sS$ is a generating sheaf if and only if it is $p$-very ample. 
\end{prop}

So later we will use the $p$-very ampleness to define generating sheaves.

\subsubsection{Gieseker stability and the moduli space by F. Nironi}\label{subsec_Gieseker_stability}

Let us fix again $p: \sS\to S$ a smooth DM stack and the map to its coarse moduli space $S$.  Let 
$\sO_S(1)$ be the very ample invertible sheaf on $S$, and $\Xi$ a generating sheaf on $\sS$. We call the pair $(\Xi, \sO_S(1))$ a polarization of $\sS$. 

\begin{defn}\label{defn_support_sheaves}
Let $F$ be a coherent sheaf on $\sS$, we define the support of $F$ to be the closed substack associated with the ideal
$$0\to \sI\to \sO_{\sS}\to \cEnd_{\sO_{\sS}}(F).$$
So $\dim(\Supp F)$ is the dimension of the substack associated with the ideal $\sI\subset \sO_{\sS}$ since $\sS$ is a DM stack. 
\end{defn}

\begin{defn}\label{defn_pure_sheaf}
A pure sheaf of dimension $d$ is a coherent sheaf $F$ such that for every non-zero subsheaf $F^\prime\subset F$ the support of $F^\prime$ is of pure dimension $d$. 
\end{defn}
So for any coherent sheaf $F$, we have the torsion filtration:
$$0\subset T_0(F)\subset \cdots \subset T_d(F)=F$$
where every $T_i(F)/T_{i-1}(F)$ is pure of dimension $i$ or zero, see \cite[\S 1.1.4]{HL}.

Let us define the Gieseker stability condition:
\begin{defn}\label{defn_Gieseker}
The modified Hilbert polynomial of a coherent sheaf $F$ on $\sS$ is defined as:
$$H_{\Xi}(F,m)=\chi(\sS, F\otimes\Xi^{\vee}\otimes p^*\sO_{\sS}(m))
=H(F_{\Xi}(F)(m))=\chi(S, F_{\Xi}(F)(m)).$$
\end{defn}
\begin{rmk}
\begin{enumerate}
\item  Let $F$ be of dimension $d$, then we can write:
$$H_{\Xi}(F,m)=\sum_{i=0}^{d}\alpha_{\Xi, i}(F)\frac{m^i}{i!}$$
which is induced by the case of schemes.
\item Also the modified Hilbert polynomial is additive on short exact sequences since the functor $F_{\Xi}$ is exact. 
\item If we don't choose the generating sheaf $\Xi$, the Hilbert polynomial $H$ on $\sS$ will be the same as the Hilbert polynomial on the coarse moduli space $S$.  In order to get interesting information on the DM stack $\sS$, the sheaf $\Xi$ is necessary.  For example, in \cite[\S 7]{Nironi}, \cite{Jiang2} the modified Hilbert polynomial on a root stack $\sS$ will corresponds to the parabolic Hilbert polynomial on the pair $(S,D)$ with $D\subset S$ a smooth divisor. 
\end{enumerate}
\end{rmk}

\begin{defn}\label{defn_reduced_hilbert}
The {\em reduced modified Hilbert polynomial} for the pure sheaf $F$ is defined as 
$$h_{\Xi}(F)=\frac{H_{\Xi}(F)}{\alpha_{\Xi, d}(F)}.$$
\end{defn}

\begin{defn}\label{defn_stability}
Let $F$ be a pure coherent sheaf.  We call $F$ semistable if for every proper subsheaf $F^\prime\subset F$, 
$$h_{\Xi}(F^\prime)\leq h_{\Xi}(F).$$
We call $F$ stable if $\leq$ is replaced by $<$ in the above inequality.
\end{defn}

\begin{defn}\label{defn_slope_stability}
 We define the slope of  $F$ by 
 $$\mu_{\Xi}(F)=\frac{\alpha_{\Xi, d-1}(F)}{\alpha_{\Xi,d}(F)}.$$
 Then $F$ is
semistable if for every proper subsheaf $F^\prime\subset F$, 
$$\mu_{\Xi}(F^\prime)\leq \mu_{\Xi}(F).$$
We call $F$ stable if $\leq$ is replaced by $<$ in the above inequality.
\end{defn}

\begin{rmk}\label{rem_stability}
\begin{enumerate}
\item The notion of $\mu$-stability and semistability is related to the Gieseker stability and semistability in the same way as schemes, i.e.,
$$\mu-\text{stable}\Rightarrow \text{Gieseker stable}\Rightarrow \text{Gieseker semistable}\Rightarrow \mu-\text{semistable}$$
\item
The stability really depends on the generating sheaf $\Xi$.  This stability is not necessarily the same as the ordinary Gieseker stability even when $\sS$ is a scheme.  
\item 
One can define the rank 
$$\rk F_{\Xi}(F)=\frac{\alpha_{\Xi, d}(F)}{\alpha_{d}(\sO_{S})}.$$
\end{enumerate}
\end{rmk}

Let us fix a polarization $(\Xi, \sO_{S}(1))$ on $\sS$, and a modified Hilbert polynomial $H$. We define the moduli functor 
$$\sM:=\sM^H_{\Xi}(\sS): (\Sch_{\kappa})\to (\mbox{Groupoids})$$
by
$$T\mapsto \left\{F \left|
 \begin{array}{l}
  \text{$F\to T$ is a flat family of semistable sheaves} \\
  \text{on $S_{T}/T$ with modified Hilbert polynomial $H$.}
   \end{array}
 \right\}\right/\cong.
 $$
In \cite[\S 4, \S 5]{Nironi}, Nironi studies in detail the boundedness of the flat  family and prove that the moduli functor $\sM$ is a global GIT quotient stack.  
We briefly review it here. 
Fix an integer $m$ such that every semistable sheaf on $\sS$ is $m$-regular.  Let $N$ be the positive integer $N:=H_{\Xi}(F,m)$, and denote with $V$ the linear space $\cc^{\oplus N}$. 
There is an open subscheme $\mathcal{Q}$ in the quot scheme 
$\mbox{Quot}_{\sS}(V\otimes\Xi\otimes p^*\sO_{S}(-m), H)$ such that 
$\sM=[\mathcal{Q}/GL_N]$.  The coarse moduli space $\rM$ of $\sM$ is 
 is a projective scheme.  Moreover, the stable locus $\rM^{s}\subset \rM$ is an open quasi-projective scheme.

\subsection{The moduli space of Higgs pairs}\label{subsec_moduli_Higgs}

Now let 
$$\sX:=\mbox{Tot}(K_{\sS})$$
be the total space of the canonical line bundle $K_{\sS}$ on $\sS$.  Since $\sS$ is a smooth two dimensional DM stack, $K_{\sS}$ exists as a line bundle.  The total space $\sX$ is a Calabi-Yau threefold DM stack. For example, if $\sS=\mathbb{P}(1,2,2)$ is a weighted projective plane, then $\sX=\mbox{Tot}(\sO_{\mathbb{P}(1,2,2)}(-5))$. If $\sS$ is a quintic surface in $\mathbb{P}^3$ with isolated ADE singularities $P_1, \cdots, P_n$, then locally around the point $P_i$, there exists an open neighborhood 
$$U_i\cong [V_i/G_i]$$
where $V_i\cong \aaa_{\kappa}^2$ and $G_i\subset SU(2)$ is a finite ADE type subgroup.  Then $K_{\sS}$ is an orbifold line bundle, which around $[V_i/G_i]$ is a $G_i$-equivariant line bundle on $V_i$ for each $i$. 

\subsubsection{Spectral construction of Tanaka-Thomas}

Let us fix a line bundle $\sL$ on $\sS$. 
A $\sL$-Higgs pair on $\sS$ is given by $(E,\phi)$, where $E\in \Coh(\sS)$ is a coherent sheaf and 
$$\phi\in \Hom(E, E\otimes \sL)$$
is a section. 

\begin{prop}\label{prop_equivalent_categories}
There exists an abelian category $\Higg_{\sL}(\sS)$ of Higgs pairs on $\sS$ and an equivalence:
\begin{equation}\label{eqn_equivalence_categories}
\Higg_{\sL}(\sS)\stackrel{\sim}{\longrightarrow} \Coh_{c}(\sS)
\end{equation}
where $\Coh_{c}(\sS)$ is the category of compactly supported coherent sheaves on $\sX$. 
\end{prop}
\begin{proof}
We generalize the case of projective surface $S$ and $X=\mbox{Tot}(K_S)$ as in \cite[Proposition 2.2 ]{TT1}. 
The morphisms between Higgs pairs $(E,\phi)$ and $(E^\prime, \phi^\prime)$ are given in an obvious way by the the diagram
\[
\xymatrix{
E\ar[r]^{\phi}\ar[d]_{f}& E\otimes \sL\ar[d]^{f\otimes \id}\\
E^\prime \ar[r]^{\phi^\prime}& E^\prime\otimes \sL
}
\]
The kernels and cokernels of $f$ and $f\otimes \id$ define the kernel and cokernel Higgs pairs.  Hence the $\sL$-Higgs pairs form an abelian category. 

Let 's prove the equivalence.   The morphism $\pi: \sX\to \sS$ is also affine in the category of DM stacks.  It is still true that $\pi_*$, taken as a functor,  is an equivalence between the category of coherent $\sO_{\sX}$-modules and the abelian category of $\pi_*(\sO_{\sX})$-modules on $\sS$, for example see \cite{Stack_Project}.  One can see this locally that $\sS$ behaves like a quotient $[V/G]$ for $V\cong \cc^2$  and $G$ a finite group scheme acting on $V$, the the morphism
$\pi: [V\times \cc/G]\to [V/G]$ can be taken as a $G$-equivariant morphism $V\times \cc\to V$.  

Then in this case 
\begin{equation}\label{eqn_pi_OX}
\pi_*\sO_{\sX} = \bigoplus_{i\geq 0}\sL^{-i}\eta^i, i\geq 0
\end{equation}
where we take $\eta$ as the tautological section of $\pi^*\sL$ which is linear on the fibers and cuts out the zero section $\sS\subset \sX$. The sheaves $\sE$ on $\sX$ are equivalent to sheaves of modules $\pi_*\sE$ over $\pi_*\sO_{\sX}$.

Still (\ref{eqn_pi_OX}) is generated by $\sO_{\sS}$ and $\sL^{-1}  \eta$, so a module over $\pi_*\sO_{\sX}$ is equivalent to an $\sO_{\sS}$ -module $E$ together with a commuting action of $\sL^{-1}\cdot \eta$, i.e. an $\sO_{\sS}$-linear map
$$E\otimes \sL^{-1} \stackrel{\pi_*\eta}{\longrightarrow} E.$$ 
Thus we get an $\sL$-Higgs pair
\begin{equation}\label{eqn_Higgs_pair}
(E,\phi) = (\pi_*\sE,\pi_*\eta).
\end{equation}

On the other hand,  given a Higgs pair $(E, \phi)$ we get an action of $\sL^{-i}\cdot \eta$ by
$$E\otimes \sL^{-i}\stackrel{\phi^i}{\longrightarrow}E.$$ 
We sum over 
all $i$ and get   an action of $\pi_*\sO_{\sX}$  on $E$.  We denote by $\sE_{\phi}$ for this sheaf. 
This defines a functor from $\Higg_{\sL}(\sS)$ to $\Coh_{c}(\sX)$, which is an equivalence. 
Finally if $\sE$ is coherent, then $\pi_*\sE$ is coherent if and only if $\pi|_{\Supp \sE}$ is proper and if and only if 
$\sE$ is compactly supported. 
\end{proof}

\subsubsection{Moduli of Higgs pairs}

We can define the Gieseker stability on the Higgs pairs $(E,\phi)$.  Let us fix a generating sheaf $\Xi$ on $\sS$. Then for any coherent sheaf $E\in\Coh(\sS)$ we have the modified Hilbert polynomial $h_{\Xi}(E)$.

\begin{defn}\label{defn_Gieseker_Higgs}
The $\sL$-Higgs pair $(E,\phi)$ is said to be Gieseker stable with respect to the polarization $(\Xi, \sO_{S}(1))$ if and only if 
$$h_{\Xi}(F)<h_{\Xi}(E)$$
for every proper $\phi$-invariant subsheaf $F\subset E$. 
\end{defn}

We define the moduli functor of Higgs pairs as:
$$\N:=\N^H_{\Xi}(\sS): (\Sch_{\kappa})\to (\mbox{Groupoids})$$
by
$$T\mapsto \left\{(E,\phi) \left|
 \begin{array}{l}
  \text{$(E,\phi)\to T$ is a flat family of stable Higgs sheaves} \\
  \text{on $S_{T}/T$ with modified Hilbert polynomial $H$.}
   \end{array}
 \right\}\right/\cong.
 $$
Then $\N$ is also represented by a GIT quotient stack with coarse moduli space a  quasi-projective scheme which we still denote by $\N$. We see this from the following:

First consider the diagram:
\begin{equation}\label{eqn_diagram}
\xymatrix{
\sX\ar[r]^{p}\ar[d]_{\pi}& X\ar[d]^{\pi}\\
\sS\ar[r]^{p}& S
}
\end{equation}
where $p$ are the morphisms from the DM stacks to their coarse moduli spaces, and $\pi$ are the morphisms from the line bundles to the base. 
\begin{prop}\label{prop_Gieseker_matches}
For the projection $\pi: \sX\to \sS$, the pullback $\pi^*\Xi$ is a generating sheaf on $\sX$. 
Moreover, under the equivalence (\ref{eqn_equivalence_categories}), the Gieseker (semi)stability of the Higgs sheaves 
$(E,\phi)$ with respect to $(\Xi, \sO_{S}(1))$ is equivalent to the Gieseker stability of the torsion sheaves 
$\sE_{\phi}$ with respect to the pair $(\pi^*\Xi, \pi^*\sO_{S}(1))$. 
\end{prop}
\begin{proof}
First we show that $\pi^*\Xi$ is a generating sheaf on $\sX$.  For any local chart $[\cc^2/G]$ of $\sS$, where $G$ is a finite group acting on $\cc^2$, the generating sheaf $\Xi$ contains all the irreducible representations of $G$.  All the local charts of $\sX$ are of the form $[\cc^2\times\cc/G]$, where  the  finite group $G$-action on the extra factor  $\cc$ in $\cc^2\times\cc$ is induced by the action on $\cc^2$ (taking the extra $\cc$ as $\Lambda^2\cc^2$), thus $\pi^*\Xi$ also contains all the irreducible representations of $G$. Therefore $\pi^*\Xi$ is a generating sheaf. 

From the equivalence (\ref{eqn_equivalence_categories}), the $\phi$-invariant subsheaves $F\subset E$ are equivalent to subsheaves $\sF\subset \sE_{\phi}$ on $\sX$.  We also have 
$$\chi(\sS, E\otimes \Xi^{\vee})=\chi(\sS, \pi_*\sE_{\phi}\otimes \Xi^{\vee})
=\chi(\sX, \sE_{\phi}\otimes \pi^*\Xi^{\vee}).$$
So the Gieseker stability in Definition \ref{defn_Gieseker_Higgs} is equivalent to 
$$h_{\pi^*\Xi}(\sF)< h_{\pi^*\Xi}(\sE_{\phi})$$
for all proper subsheaves $\sF\subset \sE_{\phi}$.   This is the Gieseker stability of the torsion sheaves on $\sX$ with respect to $\pi^*\Xi$ and $\pi^*\sO_{S}(1)$. 
\end{proof}

\begin{rmk}
Hence from Proposition \ref{prop_Gieseker_matches}, the moduli space of stable Higgs pairs $\N$ is isomorphic to the moduli space of stable coherent sheaves on $\sX$ with fixed modified Hilbert polynomial $H$, which is a GIT quotient stack with coarse moduli space a  quasi-projective scheme by Nironi's result in \cite[\S 6]{Nironi}. In the Appendix we prove after modulo out the  global $\cc^*$-action the moduli stack $\N$ admits a perfect obstruction theory.
\end{rmk}

\section{Deformation theory  and the Vafa-Witten invariants}\label{sec_VW}

Fix $\pi: \sX:=\mbox{Tot}(\sL)\to \sS$, the projection from the total space of the line bundle $\sL$ to $\sS$. Then from the spectral theory a coherent sheaf $\sE$ on $\sX$ is equivalent to a 
$\pi_{*}\sO_{\sX}=\bigoplus_{i\geq 0}\mathcal{L}^{-i}\eta^i$-module, where $\eta$ is the tautological section of $\pi^*\mathcal{L}$. 

From \cite[\S 2.2]{TT1}, given a Higgs pair $(E, \phi)$, we have the torsion sheaf $\sE_{\phi}$ of $\sX$ supported on $\sS$.  $\sE_\phi$ is generated by its sections down on $\pi$ and we have a natural surjective morphism
\begin{equation}\label{eqn_rE_quotient}
0\to \pi^*(E\otimes\sL^{-1})\stackrel{\pi^*\phi-\eta}{\longrightarrow}\pi^*E=\pi^*\pi_*\sE_{\phi}\stackrel{\ev}{\longrightarrow}\sE_\phi\to 0
\end{equation}
with kernel $\pi^*(E\otimes\sL^{-1})$ as in Proposition 2.11 of \cite{TT1}.  All the arguments in 
\cite[Proposition 2.11]{TT1} work for smooth DM stack $\sS$ and $\sX$. The reason is that since $\pi_*\eta=\phi$, $\pi^*E$ is divided by the submodule $\pi^*(E\otimes\sL^{-1})$ to make sure $\eta$ acts as $\pi^*\phi$ on the quotient, hence we have
$\sE_\phi$. 

\subsection{Deformation theory}

The deformation of $\sE$ on $\sX$ is governed by $\Ext^*_{\sX}(\sE, \sE)$, while the Higgs pair $(E,\phi)$ is governed by the cohomology groups of the total complex
$$R\cHom_{\sS}(E, E)\stackrel{[\cdot, \phi]}{\longrightarrow}R\cHom_{\sS}(E,E\otimes\sL).$$
By  homological algebra proof as in \cite[Proposition 2.14]{TT1}, we have the exact triangle:
\begin{equation}\label{eqn_deformation1}
R\cHom(\sE_\phi,\sE_\phi)\to R\cHom_{\sS}(E, E)\stackrel{\circ\phi-\phi\circ}{\longrightarrow}R\cHom_{\sS}(E\otimes\sL^{-1},E).
\end{equation}
Taking cohomology of $(\ref{eqn_deformation1})$ we get 
\begin{equation}\label{eqn_coh_deformation1}
\cdots\to\Hom(E,E\otimes K_{\sS})\to \Ext^1(\sE_\phi, \sE_{\phi})\to \Ext^1(E,E)\to \cdots
\end{equation}
which relates the automorphisms, deformations and obstructions of $\sE_\phi$ to those of $(E,\phi)$. 

\subsection{Families and the moduli space}

Let $\sS\to B$ be a family of surface DM stacks $\sS$, i.e., a smooth projective morphism with the fiber a surface DM stack, and let $\sX\to B$ be the total space of the a line bundle $\sL=K_{\sS/B}$. 

Let $\N^{H}$ denote the moduli space of Gieseker stable Higgs pairs on the fibre of $\sS\to B$ with fixed rank $r> 0$ and Hilbert polynomial $H$ (a fixed generating sheaf $\Xi$).    

We pick a (twisted by the $\cc^*$-action) universal sheaf $\rE$ over $\N\times_{B}\sX$.  We use the same $\pi$ to represent the projections
$$\pi: \sX\to \sS; \quad  \pi: \N\times_{B}\sX\to \N\times_{B}\sS.$$
Since $\rE$  is flat over $\N$ and $\pi$ is affine, 
$$\E:=\pi_*\rE \text{~on~} \N\times_{B}\sS$$
is flat over $\N$.  $\E$ is also coherent because  $\rE$ is coherent and compactly supported. Therefore it defines a classifying map:
$$\Pi: \N\to \rM$$
by
$$\sE\mapsto \pi_*\sE; \quad  (E,\phi)\mapsto E,$$
where $\rM$ is the moduli space of coherent sheaves on the fiber of $\sS\to B$ with Hilbert polynomial $H$. 
For simplicity, we use the same $\E$ over $\rM\times \sS$ and $\E=\Pi^*\E$ on $\N\times \sS$. 
Let 
$$p_{\sX}:  \N\times_{B}\sX\to \N; \quad   p_{\sS}:  \N\times_{B}\sS\to \N$$
be the projections.  Then (\ref{eqn_deformation1}) becomes:
\begin{equation}\label{eqn_deformation2}
R\cHom_{p_{\sX}}(\rE, \rE)\stackrel{\pi_*}{\longrightarrow}R\cHom_{p_{\sS}}(\E, \E)\stackrel{[\cdot, \phi]}{\longrightarrow}R\cHom_{p_{\sS}}(\E, \E\otimes\sL).
\end{equation}
Let $\sL=K_{\sS/B}$ and
taking the relative Serre dual of the above exact triangle we get 
$$R\cHom_{p_{\sS}}(\E, \E)[2]\to R\cHom_{p_{\sS}}(\E, \E\otimes K_{\sS/B})[2]\to R\cHom_{p_{\sX}}(\rE, \rE)[3].$$
\begin{prop}(\cite[Proposition 2.21]{TT1})\label{prop_self_dual}
The above exact triangle  is the same as (\ref{eqn_deformation2}), just shifted.
\end{prop} 
\begin{proof}
We check that the right hand arrow is the same as in (\ref{eqn_deformation2}). 
Look at the diagram: 
\[
\xymatrix{
R\cHom_{p_{\sX}}(\rE, \rE)\ar[d]^{\pi_*}&\otimes&R\cHom_{p_{\sX}}(\rE, \rE)[3]\ar[r]& \sO_{\N}\\
R\cHom_{p_{\sS}}(\E, \E)&\otimes &R\cHom_{p_{\sS}}(\E, \E\otimes K_{\sS/B})[2]\ar[r]\ar[u]_{\partial}& \sO_{\N}
}
\]
The horizontal morphisms intertwine the vertical pairs given by cup product and trace. The morphism $\partial$ is the coboundary morphism of (\ref{eqn_deformation2}), which is the cup product with the canonical extension class 
$$e\in \Ext^1_{\N\times \sX}(\rE, \pi^*\E\otimes K_{\sS/B}^{-1}).$$
Since 
$$R\cHom_{p_{\sS}}(\E, \E\otimes K_{\sS/B})\cong R\cHom_{p_{\sS}}(\E\otimes K_{\sS/B}^{-1}, \pi_*\rE)\cong 
R\cHom_{p_{\sX}}(\pi^*\E\otimes K_{\sS/B}^{-1}, \rE),$$
the result can be seen from the two pairings:
\[
\xymatrix{
R\cHom_{p_{\sX}}(\rE, \rE)\otimes^{L}R\cHom_{p_{\sX}}(\pi^*\E\otimes K_{\sS/B}^{-1}, \rE)[2]
\ar@{=}[dd]\ar[dr]^--{\tr((-)\circ(-)\circ e)}&  \\
&\sO_{\N}\\
R\cHom_{p_{\sX}}(\rE, \rE)\otimes^{L}R\cHom_{p_{\sS}}(\E, \E\otimes K_{\sS/B})[2]\ar[ur]_--{\tr(\pi_*(-)\circ(-))}&
}
\]
are the same. The proof is from \cite[Propositioon 2.21]{TT1}, where the authors use the Grothendieck operators
$\pi^{!}, \pi_{*}$ such that $\pi_* \pi^{!}\to \id$ is counit, and $\pi^{!}=\pi^*\otimes K_{\sS/B}^{-1}[1]$, which is the dualizing complex of $\pi$. All these proofs work for smooth DM stacks $\sX$ and $\pi: \sX\to \sS$. 
\end{proof}

Then the exact triangle  (\ref{eqn_deformation2}) fits into the following commutative diagram 
(\cite[Corollary 2.22]{TT1}):
\[
\xymatrix{
R\cHom_{p_{\sS}}(\E, \E\otimes K_{\sS/B})_{0}[-1]\ar[r]\ar@{<->}[d] &R\cHom_{p_{\sX}}(\rE, \rE)_{\perp}
\ar[r]\ar@{<->}[d] & R\cHom_{p_{\sS}}(\E, \E)_{0}\ar@{<->}[d]\\
R\cHom_{p_{\sS}}(\E, \E\otimes K_{\sS/B})[-1]\ar[r]\ar@{<->}[d]_{\id}^{\tr} &R\cHom_{p_{\sX}}(\rE, \rE)
\ar[r]\ar@{<->}[d] & R\cHom_{p_{\sS}}(\E, \E)\ar@{<->}[d]_{\id}^{\tr}\\
Rp_{\sS *}K_{\sS/B}[-1]\ar@{<->}[r]& Rp_{\sS *}K_{\sS/B}[-1]\oplus Rp_{\sS *}\sO_{\sS}\ar@{<->}[r]& Rp_{\sS *}\sO_{\sS}
}
\]
where $(-)_0$ denotes the trace-free Homs.  The $R\cHom_{p_{\sX}}(\rE, \rE)_{\perp}$ is the co-cone of the middle column. It will provide the symmetric obstruction theory of the moduli space $\N_{L}^{\perp}$ of stable trace free fixed determinant Higgs pairs.

\subsection{The $U(\rk)$ Vafa-Witten invariants}

From Proposition \ref{prop_self_dual}, in the appendix we review that the truncation $\tau^{[-1,0]}R\cHom_{p_{\sX}}(\rE, \rE)$ defines a symmetric perfect obstruction theory on the moduli space $\N$. 

The total space $\sX=\mbox{Tot}(K_{\sS})\to \sS$ admits a $\cc^*$-action which has weight one on the fibers.  The obstruction theory in (\ref{eqn_Atiyah_deformation_obstruction2})  in the Appendix is naturally $\cc^*$-equivariant. 
From \cite{GP}, the $\cc^*$-fixed locus $\N^{\cc^*}$ inherits a perfect obstruction theory
\begin{equation}\label{eqn_deformation_obstruction_fixed_locus}
\left(\tau^{[-1,0]}(R\cHom_{p_{\sX}}(\rE,\rE)[2])\Tt^{-1}\right)^{\cc^*}\to  \ll_{\N^{\cc^*}}
\end{equation}
by taking the fixed part of (\ref{eqn_Atiyah_deformation_obstruction2}).  Therefore it induces a virtual fundamental cycle 
$$[\N^{\cc^*}]^{\vir}\in H_*(\N^{\cc^*}).$$
The virtual normal bundle is given by
$$N^{\vir}:=\left(\tau^{[-1,0]}(R\cHom_{p_{\sX}}(\rE,\rE)[2]\Tt^{-1})^{\mov}\right)^{\vee}
=\tau^{[0,1]}(R\cHom_{p_{\sX}}(\rE,\rE)[1])^{\mov}$$
which is the derived dual of the moving part of (\ref{eqn_Atiyah_deformation_obstruction2}).

Consider the localized invariant
$$\int_{[\N^{\cc^*}]^{\vir}}\frac{1}{e(N^{\vir})}.$$
We explain this a bit.  Represent $N^{\vir}$ as a $2$-term complex $[E_0\to E_1]$ of locally free $\cc^*$-equivariant sheaves with non-zero weights and define 
$$e(N^{\vir}):=\frac{c_{\top}^{\cc^*}(E_0)}{c_{\top}^{\cc^*}(E_1)}\in H^*(\N^{\cc^*}, \zz)\otimes \qq[t, t^{-1}],$$
where $t=c_1(\Tt)$ is the generator of $H^*(B\cc^*)=\zz[t]$, and $c_{\top}^{\cc^*}$ denotes the $\cc^*$-equivariant top Chern class lying in $H^*(\N^{\cc^*}, \zz)\otimes_{\zz[t]} \qq[t, t^{-1}]$. 

\begin{defn}\label{defn_VW1}
Let $\sS$ be a smooth projective surface DM stack. Fixing a generating sheaf $\Xi$ on $\sS$, and a Hilbert polynomial $H$ associated with $\Xi$. Let $\N:=\N_H$ be the moduli space of stable Higgs pairs with  Hilbert polynomial $H$.  Then the primitive Vafa-Witten invariants of $\sS$ is defined as:
$$\widetilde{\VW}_{H}(\sS):=\int_{[\N^{\cc^*}]^{\vir}}\frac{1}{e(N^{\vir})}\in \qq.$$
\end{defn}

\begin{rmk}
\begin{enumerate}
\item The invariant $\widetilde{\VW}_H$  is a constant in $\qq[t,t^{-1}]$ since $\N$ has virtual dimension zero. 
\item For stable sheaves $\sE_\phi$, we have 
$$\Ext^\bullet_{\sX}(\sE_\phi, \sE_\phi)=H^{\bullet-1}(K_{\sS})\oplus H^{\bullet}(\sO_{\sS})\oplus 
\Ext^{\bullet}_{\sX}(\sE_\phi, \sE_\phi)_{\perp},$$
where $\Ext^{\bullet}_{\sX}(\sE_\phi, \sE_\phi)_{\perp}$ is the trace zero part with determinant $L\in \Pic(\sS)$. 
Hence the obstruction sheaf in the obstruction theory (\ref{eqn_Atiyah_deformation_obstruction2}), and (\ref{eqn_deformation_obstruction_fixed_locus}) has a trivial summand $H^2(\sO_{\sS})$. 
So $[\N^{\cc^*}]^{\vir}=0$ if $h^{0,2}(\sS)>0$.  If $h^{0,1}(\sS)\neq 0$, then tensoring with flat line bundle makes the obstruction theory invariant. Therefore the integrand is the  pullback from $\N/\Jac(\sS)$, which is a lower dimensional space, hence zero. 
\end{enumerate}
\end{rmk}

\subsection{$SU(\rk)$ Vafa-Witten invariants}

Let us now fix $(L, 0)\in \Pic(\sS)\times \Gamma(K_{\sS})$, and let $\N^{\perp}_{L}$ be the fibre of 
$$\N/\Pic(\sS)\times\Gamma(K_{\sS}).$$
The moduli space $\N^{\perp}_{L}$ of stable Higgs sheaves $(E,\phi)$ with $\det(E)=L$ and trace-free $\phi\in \Hom(E,E\otimes K_{\sS})_0$ admits a symmetric obstruction theory 
$$R\cHom_{p_{\sX}}(\rE, \rE)_{\perp}[1]\Tt^{-1}\longrightarrow \ll_{\N^{\perp}_{L}}$$
from Proposition \ref{prop_symmetric_POT_Perp}. 

\begin{defn}\label{defn_SU_VW_invariants}
Let $\sS$ be a smooth projective surface DM stack.  Fix a generating sheaf $\Xi$ for $\sS$, and a Hilbert polynomial $H$ associated with $\Xi$.  Let $\N^{\perp}_{L}:=\N_{L}^{\perp, H}$ be the moduli space of stable Higgs sheaves $(E,\phi)$ with Hilbert polynomial $H$, $\det(E)=L$ and trace free $\phi$.  Then define
$$\VW_{H}(\sS):=\int_{[(\N^{\perp}_{L})^{\cc^*}]^{\vir}}\frac{1}{e(N^{\vir})}.$$
\end{defn}
\begin{rmk}
\begin{enumerate}
\item Since we work on a surface DM stack $\sS$, it maybe better to fix the K-group class $\mathbf{c}\in K_0(\sS)$ such that the Hilbert polynomial of $\mathbf{c}$ is $H$.  Then 
$\VW_{\mathbf{c}}(\sS)=\int_{[(\N^{\perp}_{L})^{\cc^*}]^{\vir}}\frac{1}{e(N^{\vir})}$ is Vafa-Witten invariant corresponding to $\mathbf{c}$. 
\item This definition $\VW_{H}(\sS)$ agrees with $\widetilde{\VW}_{H}(\sS)$ defined before for $\sS$ if  
$h^{0,1}(\sS)=h^{0,2}(\sS)=0$.
\item As in \cite[\S 6.1]{TT1}, one can define more general Higgs pairs by replacing $K_{\sS}$ by any line bundle 
$\sL\to \sS$ requiring $\deg \sL\geq \deg K_{\sS}$. But the obstruction theory will not be symmetric. 
\end{enumerate}
\end{rmk}

\subsection{$\cc^*$-fixed loci}\label{subsec_CStar_fixed_locus}

For the moduli space $\N^{\perp}_{L}$, we discuss the $\cc^*$-fixed loci. 

\subsubsection{The first locus $\phi=0$}\label{subsubsec_first_type}

For the Higgs pairs $(E,\phi)$ such that $\phi=0$, the $\cc^*$-fixed locus is exactly the moduli space $\rM_{L}$ of Gieseker stable sheaves on $\sS$ with fixed determinant $L$ and Hilbert polynomial $H$ associated with the generating sheaf $\Xi$.  The exact triangle in (\ref{eqn_deformation2}) splits the obstruction theory
$$R\cHom_{p_{\sX}}(\rE, \rE)_{\perp}[1]\Tt^{-1}\cong R\cHom_{p_{\sS}}(\E, \E\otimes K_{\sS})_{0}[1]\oplus 
R\cHom_{p_{\sS}}(\E, \E)_{0}[2]\Tt^{-1}$$
where $\Tt^{-1}$ represents the moving part of the $\cc^*$-action. Then the $\cc^*$-action induces a perfect obstruction theory 
$$E_{\rM}^{\bullet}:=R\cHom_{p_{\sS}}(\E, \E\otimes K_{\sS})_{0}[1]\to \ll_{\rM_{L}}.$$
The virtual normal bundle 
$$N^{\vir}=R\cHom_{p_{\sS}}(\E, \E\otimes K_{\sS})_{0}\Tt=E_{\rM}^{\bullet}\otimes \Tt[-1].$$
So the invariant contributed from $\rM_{L}$ (we can let $E_{\rM}^{\bullet}$ is quasi-isomorphic to $E^{-1}\to E^0$) is: 
\begin{align*}
\int_{[\rM_{L}]^{\vir}}\frac{1}{e(N^{\vir})}&=\int_{[\rM_{L}]^{\vir}}\frac{c_s^{\cc^*}(E^0\otimes \Tt)}{c_r^{\cc^*}(E^{-1}\otimes \Tt)}\\
&=\int_{[\rM_{L}]^{\vir}}\frac{c_s(E^0)+\Tt c_{s-1}(E^0)+\cdots}{c_r(E^{-1})+\Tt c_{r-1}(E^{-1})+\cdots}
\end{align*}
Here we assume $r$ and $s$ are the ranks of $E^{-1}$ and $E^0$ respectively, and $r-s$ is the virtual dimension of 
$\rM_{L}:=\rM_{L,H}$.  By the virtual dimension consideration, only $\Tt^0$ coefficient contributes and we may let $\Tt=1$, so
\begin{align}\label{eqn_virtual_Euler_number}
\int_{[\rM_{L}]^{\vir}}\frac{1}{e(N^{\vir})}&=\int_{[\rM_{L}]^{\vir}}\Big[\frac{c_{\bullet}(E^0)}{c_{\bullet}(E^{-1})}\Big]_{\vd}\nonumber \\
&=\int_{[\rM_{L}]^{\vir}}c_{\vd}(E_{\rM}^{\bullet})\in \zz.
\end{align}
This is the signed virtual Euler number of Ciocan-Fontanine-Kapranov/Fantechi-G\"ottsche. 

\begin{prop}\label{prop_K_S_fixed_locus}
Let us fix a generating sheaf $\Xi$ on $\sS$.  If $\deg K_{\sS}\leq 0$, then any stable $\cc^*$-fixed Higgs pair 
$(E,\phi)$ has Higgs field $\phi=0$. Therefore if we fix some $K$-group class $\mathbf{c}\in K_0(\sS)$, then 
$\VW^{L}_{\mathbf{c}}(\sS)$ is the same as the signed virtual Euler number in (\ref{eqn_virtual_Euler_number}).
\end{prop}
\begin{proof}
Consider 
$$0\to \ker \phi\longrightarrow E\stackrel{\phi}{\longrightarrow} E\otimes K_{\sS}$$
where $\ker \phi$ and $\Im \phi$ are $\phi$-invariant, and  for the generating sheaf $\Xi$, Gieseker stability gives:
$$p_{E, \Xi}(n)<p_{\Im\phi, \Xi}(n)< p_{E\otimes K_{\sS}, \Xi}(n),  \forall n>>0$$
unless $\Im \phi=0$ or $E\otimes K_{\sS}$.  The $\deg(K_{\sS})\leq 0$ implies that 
$$p_{E\otimes K_{\sS}, \Xi}(n)\leq p_{E,\Xi}(n)$$
for $n>>0$. 
So either $\phi=0$ or $\phi$ is an isomorphism.  But $\phi$ is $\cc^*$-fixed and it has determinant zero, it can not be an isomorphism. 
\end{proof}

Also we have:

\begin{prop}
If $\deg K_{\sS}< 0$, then any semistable $\cc^*$-fixed Higgs pair $(E,\phi)$ has Higgs field $\phi=0$. 
\end{prop}
\begin{proof}
This is the same as Proposition \ref{prop_K_S_fixed_locus}.
\end{proof}

\subsubsection{The second fixed locus $\phi\neq 0$}\label{subsec_second_fixed_loci}

The second component $\rM^{(2)}$ corresponds to the $\cc^*$-fixed Higgs pairs $(E,\phi)$ such that the  Higgs fields $\phi\neq 0$. Let $(E,\phi)$ be a $\cc^*$-fixed stable Higgs pair.  Since the $\cc^*$-fixed stable sheaves $\sE_{\phi}$ are simple,  we  use \cite[Proposition 4.4]{Kool}, \cite{GJK} to make this stable sheaf 
$\cc^*$-equivariant.  The cocycle condition in the  $\cc^*$-equivariant definition for the Higgs pair $(E,\phi)$ corresponds to a $\cc^*$-action 
$$\psi: \cc^*\to \Aut(E)$$
such that 
\begin{equation}\label{eqn_cocycle_condition}
\psi_{t}\circ \phi\circ \psi_t^{-1}=t\phi
\end{equation}  
With respect to the $\cc^*$-action on $E$,  it splits into a direct sum of eigenvalue subsheaves
$$E=\oplus_{i}E_i$$
where $E_i$ is the weight space such that $t$ acts by $t^i$, i.e., $\psi_t=\mbox{diag}(t^i)$. 
The action acts on the Higgs field with weight one by (\ref{eqn_cocycle_condition}).
Also for a Higgs pair $(E,\phi)$, if a $\cc^*$-action on $E$ induces weight one action on $\phi$, then it is a fixed point of the $\cc^*$-action.

Since the $\cc^*$-action on the canonical line bundle $K_{\sS}$ has weight $-1$, $\phi$ decreases the weights, and it maps the lowest weight torsion subsheaf to zero, hence zero by stability. 
So each $E_i$ is torsion free and have rank $> 0$.  Thus $\phi$ acts blockwise through morphisms
$$\phi_i: E_i\to E_{i-1}.$$
These are flags of torsion-free sheaves on $\sS$, see \cite{TT1}. 

In the case that $E_i$ has rank $1$, they are twisted by line bundles, and $\phi_i$ defining nesting of ideals. Then this is the nested Hilbert scheme on $\sS$. Very little is known of nested Hilbert schemes for surface DM stacks.

\section{Calculations}\label{sec_calculations}

In this section we do some calculations on two type of  general type surface DM stacks, one is for a $r$-root stack  over a smooth quintic surface, and the other is for quintic surface with ADE singularities.  

\subsection{Root stack on quintic surfaces}\label{subsec_root_stacks}

Let $S\subset \pp^3$ be a smooth quintic surface in $\pp^3=\proj(\cc[x_0:x_1:x_2,x_3])$, given by a homogeneous degree $5$ polynomial.  Let $C\subseteq |K_{S}|$ be a smooth connected canonical divisor such that 
$\sO_S$ is the only line bundle $L$ satisfying $0\leq \deg L\leq \frac{1}{2}\deg K_{S}$ where the degree is defined  by 
$\deg L=c_1(L)\cdot c_1(\sO_S(1))$.  Then we have the following topological invariants:
\begin{equation}\label{eqn_topology_invariants_quintic}
\begin{cases}
g_{C}=1+c_1(S)^2=1+5=6;\\
h^0(K_{S})=p_{g}(S)=\frac{1}{12}(c_1(S)^2+c_2(S))-1=\frac{1}{12}(5+55)-1=4;\\
h^0(K_S^2)=p_g(S)+g_C=10.
\end{cases}
\end{equation}
Let 
$$\sS:=\sqrt[r]{(S,C)}$$
be the root stack associated with the divisor $C$.  The stack  $\sS=\sqrt[r]{(S,C)}$ is the $r$-th root stack associated with the line bundle $\sO_S(C)$.  Let 
$$p: \sS\to S$$
be the projection to its coarse moduli space $S$, and let 
$\sC:=p^{-1}(C)$.
We still use $p: \sC\to C$ to represent the projection and it is a $\mu_r$-gerbe over $C$. 
The canonical line bundle $K_{\sS}$ satisfies the formula 
$$K_{\sS}=p^*K_{S}+\frac{r-1}{r}\sO_{\sS}(\sC)=\sO_{\sS}(\sC).$$

Recall that $\sX=\mbox{Tot}(K_{\sS})$, and $X:=\mbox{Tot}(K_{S})$, and let 
$$\pi: \sX\to \sS; \quad  \pi: X\to S$$
be the projections. 
We pick the generating sheaf ``$\Xi=\oplus_{i=0}^{r}\sO_{\sS}(i\sC^{\frac{1}{r}})$", and a Hilbert polynomial $H$, and let 
$\N_H$ be the moduli space of stable Higgs sheaves on $\sS$ with Hilbert polynomial $H$. 
\begin{rmk}
For the choice of the generating sheaf $\Xi$, the modified Hilbert polynomial $H$ of torsion free sheaves on $\sS$ corresponds to the parabolic Hilbert polynomial $H$ for parabolic sheaves on $(S,C)$. 
The moduli space of stable sheaves with modified Hilbert polynomial $H$ is actually isomorphic to the moduli space of stable parabolic sheaves on $(S,C)$ with parabolic Hilbert polynomial $H$ defined in \cite{MY}. 
The moduli space of stable Higgs sheaves on $\sS$ with modified Hilbert polynomial $H$ is also isomorphic to the 
moduli space of parabolic stable Higgs sheaves on $(S,C)$  with Hilbert polynomial $H$, see \cite{JT}. 

The perfect obstruction theory constructed in the Appendix actually implies that there exists a perfect obstruction theory on the moduli space of stable  sheaves with modified Hilbert polynomial $H$.  Therefore there exists a perfect obstruction theory on the moduli space of stable parabolic sheaves on $(S,C)$. One can study the perfect obstruction theory and the corresponding defining  invariants  for parabolic sheaves on $(S,C)$ by sheaves on the  root stack $\sS$.
\end{rmk}

\subsubsection{$\cc^*$-fixed Higgs pairs on $\rM^{(2)}$}\label{subsec_fixed_pairs}
The $\cc^*$ acts on $\sX$ by scaling the fibers of $\sX\to \sS$.  Let $(E,\phi)$ be a $\cc^*$-fixed rank $2$ Higgs pair with fixed determinant $L=K_{\sS}$ in the second component $\rM^{(2)}$ in \S \ref{subsec_second_fixed_loci}.   
Then since all the $E_i$ have rank bigger than zero, 
$$E=E_i\oplus E_j.$$
Without loss of generality, we may let 
$E=E_0\oplus E_{-1}$ since tensoring $E$ by $\Tt^{-i}$,  $E_i$ goes to $E_0$, where $\Tt$ is the standard one dimensional $\cc^*$-representation of weight one. 
Then considering $\phi$ as a weight zero element of $\Hom(E,E\otimes K_{\sS})\otimes \Tt$, we have 
$$E=E_0\oplus E_{-1}, \text{~and~} \phi= \left(
\begin{array}{cc}
0&0\\
\iota&0
\end{array}
\right)$$
for some 
$\iota: E_0\to E_{-1}\otimes K_{\sS}\otimes \Tt$. Then $E_{-1}\hookrightarrow E$ is a $\phi$-invariant subsheaf, and by semistability (Gieseker stable implies $\mu$-semistable)
we have 
$$\mu_{\Xi}(E_{-1})\leq \mu_{\Xi}(E_0)=\mu_{\Xi}(K_{\sS})-\mu_{\Xi}(E_{-1}).$$
The existence of the nonzero map $\Phi: E_0\to E_{-1}\otimes K_{\sS}$ implies:
$$\mu_{\Xi}(E_{-1})+\mu_{\Xi}(K_{\sS})\geq \mu_{\Xi}(E_0)=\mu_{\Xi}(K_{\sS})-\mu_{\Xi}(E_{-1}).$$
So 
\begin{equation}\label{eqn_inequality}
0\leq \mu_{\Xi}(E_{-1})\leq \frac{1}{2}\mu_{\Xi}(K_{\sS}).
\end{equation}

\begin{lem}\label{lem_E0E1}
The inequality (\ref{eqn_inequality}) implies that 
$$\det(E_{-1})=\sO_{\sS}; \text{~and~} \det(E_{0})=K_{\sS}.$$
\end{lem}
\begin{proof}
Since the generating sheaf $\Xi=\oplus_{i=0}^{r}\sO_{\sS}(\sC^{\frac{i}{r}})$, and $K_{\sS}\cong \sO_{\sS}(\sC)$, we look at the modified Hilbert polynomial 
$$H_{\Xi}(K_{\sS},m) = \chi(\sS,K_{\sS}\otimes \Xi^{\vee}\otimes p^*\sO_S(m)).$$
Recall that $p: \sS\to S$ is the morphism to its coarse moduli space. 
We have:
$$ \chi(\sS,K_{\sS}\otimes \Xi^{\vee}\otimes p^*\sO_S(m))=
  \chi(S,p_*K_{\sS}\otimes p_*\Xi^{\vee}\otimes \sO_S(m))$$
From Definition \ref{defn_slope_stability} and the definition of the rank in (3) of Remark \ref{rem_stability}, since $E_{-1}$ and $K_{\sS}$ all have rank one, 
$\mu_{\Xi}(K_{\sS})$ and $\mu_{\Xi}(E_{-1})$ can be calculated by the modified degree of 
$p_*K_{\sS}=p_*(\sO_{\sS}(\sC))=\sO_{S}(C)=K_S$ and $p_*E_{-1}$. 
We have 
\begin{equation}\label{lem_eqn_1}
\deg_{\Xi}(K_S)=\deg(K_S)+\sum_{i=1}^{r}\deg(K_{S}\otimes p_*\sO_{\sS}(-\sC^{\frac{i}{r}}))
\end{equation}
and 
\begin{equation}\label{lem_eqn_2}
\deg_{\Xi}(p_*E_{-1})=\deg(p_*E_{-1})+\sum_{i=1}^{r}\deg(p_*E_{-1}\otimes p_*\sO_{\sS}(-\sC^{\frac{i}{r}})).
\end{equation}

Since for $1\leq i\leq r$, 
$$\deg(K_{S}\otimes p_*\sO_{\sS}(-\sC^{\frac{i}{r}}))=\deg(K_{S})+\deg(p_*\sO_{\sS}(-\sC^{\frac{i}{r}}))$$
and 
$$\deg(E_{-1}\otimes p_*\sO_{\sS}(-\sC^{\frac{i}{r}}))=\deg(E_{-1})+\deg(p_*\sO_{\sS}(-\sC^{\frac{i}{r}})),$$
(\ref{lem_eqn_1}) and  (\ref{lem_eqn_2}) actually determine the modified degree of $p_*K_{\sS}$ and 
 $p_*E_{-1}$ respectively.  We  calculate and  get:
 $$0\leq (r+1)\cdot \deg(p_*E_{-1})\leq \frac{r+1}{2}\deg(K_S).$$
 So 
 $$0\leq \deg(p_*E_{-1})\leq \frac{1}{2}\deg(K_S).$$
 The only line bundle  $L$ on $S$ that 
$0\leq \deg (L)\leq \frac{1}{2}\deg(K_S)$ is the trivial line bundle.  
Then the determinant of the rank one sheaf $p_*E_{-1}$ must be trivial. 
From \cite{Cadman}, any line bundle on $\sS$ is a tensor product of a pullback line bundle from $S$ and a power of the  tautological  line bundle $\sO_{\sS}(\sC^{\frac{1}{r}})$ and the pushforward $p_*\sO_{\sS}(\sC^{\frac{1}{r}})=\sO_S$. 
The only possibility of the rank one sheaf 
$E_{-1}$ is $\sO_{\sS}(\sC^{\frac{j}{r}})$ for $0\leq j<r$. 
Since the $\cc^*$ action on the tautological line bundle $\sO_{\sS}(\sC^{\frac{j}{r}})$  has weight $j$ (where we make the $\cc^*$-action compatible with the $\mu_r$-action), and our $E_{-1}$ and $E_{0}$ have weights $-1$ and $0$.  This is impossible. 
Then  the determinant of 
the rank one sheaf $E_{-1}$ must be the trivial sheaf $\sO_{\sS}$.
\end{proof}

Therefore we have:
$$E_0=\sI_0\otimes K_{\sS},$$
$$E_{-1}=\sI_1\otimes \Tt^{-1}$$
for some ideal sheaves $\sI_i$.  The morphism $\sI_0\to \sI_1$ is nonzero, so we must have:
$$\sI_0\subseteq \sI_1. $$
So there exist $\sZ_1\subseteq \sZ_0$ two zero-dimensional subsheaves parametrized by $\sI_0\subseteq \sI_1$. 

\subsubsection{Components in terms of $K$-group class}\label{subsec_components_K-group}

Let $K_0(\sS)$ be the Grothendieck $K$-group of $\sS$, and we want to use Hilbert scheme on $\sS$ parametrized by 
$K$-group classes.  We fix the filtration
$$F_0K_0(\sS)\subset F_1K_0(\sS)\subset F_2K_0(\sS)$$
where $F_iK_0(\sS)$ is the subgroup of $K_0(\sS)$ such that the support of the elements in $F_iK_0(\sS)$ has dimension $\leq i$. 
The orbifold Chern character morphism is defined by:
\begin{equation}\label{eqn_orbifold_Chern_character}
\widetilde{\Ch}: K_0(\sS)\to H^*_{\CR}(\sS,\qq)=H^*(I\sS, \qq)
\end{equation}
where $H^*_{\CR}(\sS,\qq)$ is the Chen-Ruan cohomology of $\sS$.  The inertia stack 
$$I\sS=\sS\bigsqcup \sqcup_{i=1}^{r-1}\sC_i$$
where each $\sC_i=\sC$ is the stacky divisor of $\sS$.   We should understand that the inertia stack 
is indexed by the element $g\in\mu_r$, $\sS_{g}\cong \sC$ is the component corresponding to $g$.  It is clear that 
$\sS_{1}=\sS$ and $\sS_{g}=\sC$ if $g\neq 1$. 
Let $\zeta\in \mu_r$ be the generator of $\mu_r$. 
Then 
$$H^*(I\sS, \qq)=H^*(\sS)\bigoplus \oplus_{i=1}^{r-1}H^*(\sC_i),$$
where $\sC_i$ corresponds to the element $\zeta^i$. 
The cohomology of $H^*(\sC_i)$ is isomorphic to $H^*(C)$. 
For any coherent sheaf $E$, the restriction of $E$ to every $\sC_i$ has a $\mu_r$-action. 
We assume that 
$$E=E_{\zeta^i}^1\oplus E_{\zeta^i}^2$$
is the decomposition of eigen-subsheaves such that it acts by 
$e^{2\pi i\frac{f_{i1}}{r}}$ on $E_{\zeta^i}^1$ and $e^{2\pi i\frac{f_{i2}}{r}}$ on $E_{\zeta^i}^2$. 
We  let 
\begin{equation}\label{eqn_Chern_character_value1}
\widetilde{\Ch}(E)=(\Ch(E), \oplus_{i=1}^{r-1}\Ch(E|_{\sC_i})),
\end{equation}
where 
$$\Ch(E)=(\rk(E), c_1(E), c_2(E))\in H^*(\sS),$$
and 
$$
\Ch(E|_{\sC_i})=\left(e^{2\pi i\frac{f_{i1}}{r}}+e^{2\pi i\frac{f_{i2}}{r}}, e^{2\pi i\frac{f_{i1}}{r}}c_1(E_{\zeta^i}^1)
+e^{2\pi i\frac{f_{i2}}{r}}c_1(E_{\zeta^i}^2)\right)\in H^*(\sC_i).$$

In order to write down the generating function later.  We introduce some notations.  We roughly write 
$$\widetilde{\Ch}(E)=(\widetilde{\Ch}_{g}(E))$$
where $\widetilde{\Ch}_{g}(E)$ is the component in $H^*(\sS_{g})$ as in (\ref{eqn_Chern_character_value1}). 
Then define:
\begin{equation}\label{eqn_Chern_character_degree}
\left(\widetilde{\Ch}_{g}\right)^k:=\left(\widetilde{\Ch}_{g}\right)_{\dim \sS_{g}-k}\in H^{\dim \sS_{g}-k}(\sS_{g}).
\end{equation}
The $k$ is called the codegree in \cite{GJK}.  In the inertia stack $I\sS$, the component  $\sS_g$ is either the whole $\sS$, or $\sC$,
therefore if we have a rank $2$ $\cc^*$-fixed Higgs pair $(E,\phi)$ with fixed $c_1(E)=-c_1(\sS)$, then 
$\left(\widetilde{\Ch}_{g}\right)^2(E)=2$, the rank; while 
$$
\left(\widetilde{\Ch}_{g}\right)^1(E)=
\begin{cases}
-c_1(\sS), & g=1;\\
 e^{2\pi i\frac{f_{i1}}{r}}+ e^{2\pi i\frac{f_{i2}}{r}}, & g=\zeta^i \neq 1.
\end{cases}
$$
Also we have 
$$
\left(\widetilde{\Ch}_{g}\right)^0(E)=
\begin{cases}
c_2(E), & g=1;\\
e^{2\pi i\frac{f_{i1}}{r}}c_1(E_{\zeta^i}^1)
+e^{2\pi i\frac{f_{i2}}{r}}c_1(E_{\zeta^i}^2), & g=\zeta^i \neq 1.
\end{cases}
$$

Therefore we have the following proposition:
\begin{prop}\label{prop_second_fixed_loci_Hilbert_scheme}
In the case that the rank of stable Higgs sheaves is $2$, we fix a $K$-group class 
$\mathbf{c}\in K_0(\sS)$ such that $\left(\widetilde{\Ch}_{1}\right)^1(\mathbf{c})=-c_1(\sS)$.  Then 
\begin{enumerate}
\item If $c_2(E)<0$, then the $\cc^*$-fixed locus is empty by the assumption of Bogomolov inequality.  
\item If $c_2(E)\geq 0$,  then 
$$\rM^{(2)}\cong \bigsqcup_{\alpha\in F_0K_0(\sS)}\Hilb^{\alpha, \mathbf{c}_0-\alpha}(\sS)$$
where $\mathbf{c_0}\in F_0K_0(\sS)$ such that $\left(\widetilde{\Ch}_{g}\right)^0(\mathbf{c}_0)=\left(\widetilde{\Ch}_{g}\right)^0(\mathbf{c})$; and $\Hilb^{\alpha, \mathbf{c}_0-\alpha}(\sS)$ is the nested Hilbert scheme of zero-dimensional substacks of $\sS$:
$$\sZ_1\subseteq \sZ_0$$
such that $[\sZ_1]=\alpha$, $[\sZ_0]=\mathbf{c}_0-\alpha$.  $\square$
\end{enumerate}
\end{prop}
\begin{proof}
For $(1)$, the Bogomolov inequality holds for root gerbes over schemes, see \cite{Lieblich}.   Since the slope semistability of the sheave $E$ for $\sS$ corresponds to parabolic semistability for the corresponding parabolic sheaf $E_{*}$ as in \cite{MY},  and \cite[Theorem 7.1]{Anchouche} has proved the Bogomolov inequality for parabolic sheaves, we can use it here. 
Since $E$ is Gieseker stable (with fixed generating sheaf), it is slope semistable.  The Bogomolov inequality for $E$ says that 
$\Delta(E)=2\rk c_2(E)-(\rk-1)c_1(E)^2\geq 0$.  Thus if $c_2(E)<0$, the Bogomolov inequality will fail which is impossible. 
The general case of Bogomolov inequality is treated in  \cite{JKT} for surface DM stacks.

$(2)$ is from the arguments in  \S \ref{subsec_fixed_pairs}, since any $\cc^*$-fixed Higgs pair $(E,\phi)$ with determinant $K_{\sS}$ has the form 
$E=E_0\oplus E_{-1}=\sI_0\otimes K_{\sS}\oplus \sI_1\otimes \Tt^{-1}$.  The result in  \S \ref{subsec_fixed_pairs} works in families of $\cc^*$-fixed Higgs pairs with $\phi\neq 0$, which means if we have a family of $\cc^*$-fixed Higgs pairs $(E, \phi)$ on $\sS\times T$ for a scheme $T$, then these Higgs pairs are isomorphic to $\sI_0\otimes K_{\sS}\oplus \sI_1\otimes \Tt^{-1}$ where $\sI_0$ and $\sI_1$ are families of ideal sheaves on $\sS\times T$.  Thus the moduli space decomposition result in (2) follows from the definition of moduli spaces as functors. 
\end{proof}

\subsubsection{The case $\sZ_1=\emptyset$}\label{subsec_case_Z1_empty}

Therefore in this case 
$$E=\sI_0\otimes K_{\sS}\oplus \sO\cdot \Tt^{-1}.$$
So the nested Hilbert scheme $\Hilb^{\alpha, \mathbf{c}_0-\alpha}(\sS)$ is just the Hilbert scheme 
$\Hilb^{\mathbf{c}_0}(\sS)$ on $\sS$.  The deformation theory of $(E,\phi)$ is given by 
$$R\Hom_{\sS}(E, E)_0\stackrel{[\cdot, \phi]}{\longrightarrow}R\Hom_{\sS}(E,E\otimes K_{\sS\otimes \Tt^{-1}})_0$$
with $\cc^*$-action 
in \S \ref{sec_VW} and in the appendix, where the Higgs field $\phi$ has weight $0$.  Then $\Hom(E, E)$ splits into:
$$\mat{cc}\Hom(\sI_0 K_{\sS}, \sI_0 K_{\sS})&\Hom(\sO\cdot \Tt^{-1}, \sI_0 K_{\sS})\\
\Hom( \sI_0 K_{\sS}, \sO\cdot \Tt^{-1})& \Hom(\sO\cdot \Tt^{-1}, \sO\cdot \Tt^{-1})\rix
=\mat{cc}\cc\cdot\id_{\sI_0}&H^0(\sI_0 K_{\sS})\Tt\\
0& \cc\cdot \id_{\sO}\rix
$$
and 
$\Hom(E, E\otimes K_{\sS}\cdot \Tt)$ splits into:
$$\mat{cc}\Hom(\sI_0 K_{\sS}, \sI_0 K^2_{\sS}\cdot \Tt)&\Hom(\sO\cdot \Tt^{-1}, \sI_0 K_{\sS}^2\cdot \Tt)\\
\Hom( \sI_0 K_{\sS}, K_{\sS})& \Hom(\sO\cdot \Tt^{-1}, K_{\sS})\rix
=\mat{cc}H^0(K_{\sS})\cdot \Tt&H^0(\sI_0 K^2_{\sS})\Tt^2\\
\cc\cdot \iota& H^0(K_{\sS})\cdot \Tt\rix
$$
We have $\phi=\mat{cc} 0&0\\
\iota&0\rix
$, where we recall that 
$\iota: \sI_0\otimes K_{\sS}\to \sO\cdot K_{\sS}$,  and 
the map $[\cdot, \phi]$ between them acts by:
$$
\mat{cc} a&s\\
0&b\rix\mapsto \mat{cc} s\iota&0\\
(b-a)\iota&-\iota s\rix.
$$
($b=-a$ gives the map on trace-free groups.) The morphism 
$$\Hom(E,E)_0\to \Hom(E, E\otimes K_{\sS}\otimes \Tt)_0$$
is injective and has cokernel:
$$\frac{H^0(K_{\sS})}{\iota\cdot H^0(\sI_0 K_{\sS})}\cdot \Tt\oplus H^0(\sI_0 K_{\sS})\Tt^2.$$
The $\Ext^1(E,E)=\Ext^1(E,E)_0$ is:
$$\mat{cc}\Ext^1(\sI_0 K_{\sS}, \sI_0 K_{\sS}\cdot \Tt)&\Ext^1(\sO\cdot \Tt^{-1}, \sI_0 K_{\sS})\\
\Ext^1( \sI_0 K_{\sS}, \sO\cdot \Tt^{-1})& \Ext^1(\sO\cdot \Tt^{-1}, \sO\cdot \Tt^{-1})\rix
=\mat{cc}T_{\sZ_0}\Hilb^{\mathbf{c}_0}(\sS)&H^1(\sI_0 K_{\sS})\Tt\\
H^1(\sI_0\cdot K^2_{\sS})^*\cdot \Tt^{-1}& 0\rix.
$$
And the $\Ext^1(E,E\otimes K_{\sS}\cdot \Tt)=\Ext^1(E,E\otimes K_{\sS})_0\cdot \Tt$ is
$$\mat{cc}\Ext^1(\sI_0 K_{\sS}, \sI_0 K^2_{\sS}\cdot \Tt)&\Ext^1(\sO\cdot \Tt^{-1}, \sI_0 K_{\sS}^2\cdot \Tt)\\
\Ext^1( \sI_0 K_{\sS}, K_{\sS})& \Ext^1(\sO\cdot \Tt^{-1}, K_{\sS})\rix
=\mat{cc}T_{\sZ_0}\Hilb^{\mathbf{c}_0}(\sS)\cdot \Tt&H^1(\sI_0 K^2_{\sS})\Tt^2\\
H^1(\sI_0\cdot K_{\sS})^*& 0\rix
$$
and the map $[\cdot, \phi]$ between them:
$$
\mat{cc} v&s\\
f&0\rix\mapsto \mat{cc} s\iota&0\\
-\iota v&-\iota s\rix.
$$Similar to Lemma 8.7 as in \cite{TT1}, 
\begin{lem}\label{lem_ls0}
The above morphism is zero. 
\end{lem}
\begin{proof}
The proof is the same as in Lemma 8.7  \cite{TT1}, and 
$\iota s$ lies in $H^1(K_{\sS})$ which is zero since $h^1(\sO_{\sS})=0$, which is from  
$h^1(\sO_S)=0$ and $p_*\sO_{\sS}=\sO_S$.
\end{proof}

\subsubsection{Deformation theory}\label{subsec_deformation_theory_case_Z1}

Let 
\begin{equation}\label{eqn_deformation_cone}
R\Hom(E,E)_0\stackrel{[\cdot, \phi]}{\longrightarrow} R\Hom(E, E\otimes K_{\sS}\otimes\Tt)_0
\end{equation}
be the cone, and let 
$\sT^i$ be the cohomology of the cone (\ref{eqn_deformation_cone}). Then we have the following exact sequence of cohomology:
\begin{multline*}
0\to \Hom(E,E)_0\stackrel{[\cdot, \phi]}{\longrightarrow} \Hom(E, E\otimes K_{\sS})_0 \Tt\longrightarrow \sT^1\longrightarrow \Ext^1(E,E)\stackrel{[\cdot, \phi]}{\longrightarrow} \\
\Ext^1(E,E\otimes K_{\sS})\Tt\longrightarrow \sT^2\longrightarrow \Ext^2(E,E)_0 \stackrel{[\cdot, \phi]}{\longrightarrow}\Ext^2(E,E\otimes K_{\sS})\Tt\to \sT^3\to 0.
\end{multline*}
We know that the first morphism $[\cdot, \phi]$ is injective, so $\sT^0=0$. By Serre duality, the third $[\cdot, \phi]$ is surjective, so $\sT^3=0$. Therefore we have:
\begin{equation}\label{eqn_calculation_exact_sequence_deformation1}
0\to \frac{H^0(K_{\sS})}{\iota\cdot H^0(\sI_0 K_{\sS})}\cdot \Tt\oplus H^0(\sI_0 K_{\sS}^2)\Tt^2\longrightarrow
\sT^1\longrightarrow 
\mat{cc}T_{\sZ_0}\Hilb^{\mathbf{c}_0}(\sS)&H^1(\sI_0 K_{\sS})\Tt\\
H^1(\sI_0\cdot K_{\sS})^*\cdot\Tt^{-1}& 0\rix
\to 0
\end{equation}
and its Serre dual for $\sT^2$:
\begin{equation}\label{eqn_calculation_exact_sequence_obstruction2}
0\to \mat{cc}T^*_{\sZ_0}\Hilb^{\mathbf{c}_0}(\sS)\Tt&H^1(\sI_0 K_{\sS})\Tt^2\\
H^1(\sI_0\cdot K_{\sS})^*& 0\rix
\longrightarrow
\sT^2\longrightarrow 
\left(\frac{H^0(K_{\sS})}{\iota\cdot H^0(\sI_0 K_{\sS})}\right)^*\oplus H^0(\sI_0 K_{\sS}^2)^*\Tt^{-1}
\to 0.
\end{equation}

\subsubsection{Virtual cycle}\label{subsec_virtual_cycle_case_Z1}

We can see that the fixed weight zero part of $\sT^1$ is $T_{\sZ_0}\Hilb^{\mathbf{c}_0}(\sS)$.  The weight $1$ part of $\sT^1$ is putting together of 
$$ \frac{H^0(K_{\sS})}{\iota\cdot H^0(\sI_0 K_{\sS})}; \text{~and~} H^1(\sI_0 K_{\sS}).$$
These data put together to give $\Gamma(K_{\sS}|_{\sZ_0})$.  The proof is the same as in 
\cite[\S 8.2]{TT1}.  Here we only explain a bit from the exact sequence
$$0\to \sI_0\longrightarrow \sO_{\sS}\longrightarrow \sO_{\sZ_0}\to 0.$$
Tensoring with $K_{\sS}$ we get
$$0\to \sI_0\cdot K_{\sS}\longrightarrow K_{\sS}\longrightarrow K_{\sS}|_{\sZ_0}\to 0,$$
and taking cohomology
$$0\to H^0(\sI_0\cdot K_{\sS})\longrightarrow H^0(K_{\sS})\longrightarrow H^0(K_{\sS}|_{\sZ_0})
\longrightarrow H^1(\sI_0\cdot K_{\sS})\to H^1(K_{\sS})=0.$$

Consider the following diagram:
\[
\xymatrix{
\sZ^{\mathbf{c}_0}(\sS)\ar@{^(->}[r]& \Hilb^{\mathbf{c}_0}(\sS)\times \sS\ar[d]_{p_1} \ar[dr]^{p_2} & \\
& \Hilb^{\mathbf{c}_0}(\sS)& \sS
}
\]
where $\sZ^{\mathbf{c}_0}(\sS)$ is the universal zero dimensional substack in $\sS$ with $K$-group class 
$\mathbf{c}_0$. 
Let 
$$K_{\sS}^{[\mathbf{c}_0]}:=p_{1*}(p_2^*K_{\sS}\otimes \sO_{\sZ^{\mathbf{c}_0}(\sS)})$$
be the tautological vector bundle on $\Hilb^{\mathbf{c}_0}(\sS)$. Let 
$$\overline{\pi}: K_{\sS}^{[\mathbf{c}_0]}\to \Hilb^{\mathbf{c}_0}(\sS)$$
be the projection from the vector bundle, then $\overline{\pi}^{-1}(\sZ_0)=\Gamma(K_{\sS}|_{\sZ_0})$.
This bundle is the weight $1$ part of $\sT^1$, and by duality, the part of the obstruction bundle on this component 
$\rM^{(2)}$ is $(K_{\sS}^{[\mathbf{c}_0]})^*$. 

Then the virtual cycle on $\rM^{(2)}$ which is induced from virtual localization in  \cite{GP}, is the Euler class of the obstruction bundle
\begin{equation}\label{eqn_Euler_obstruction_bundle}
[\rM^{(2)}]^{\vir}=e((K_{\sS}^{[\mathbf{c}_0]})^*)\cap [\Hilb^{\mathbf{c}_0}(\sS)]
=(-1)^{\rk}e(K_{\sS}^{[\mathbf{c}_0]})\cap [\Hilb^{\mathbf{c}_0}(\sS)].
\end{equation}

Let us look at the canonical line bundle $K_{\sS}$ which is 
$\sO_{\sS}(\sC)$. 
There is a section $s$ of $K_{\sS}$ cutting out of the curve $\sC$.  Then this section $s$ induces a section $s^{[\mathbf{c}_0]}$ on the tautological vector bundle $K_{\sS}^{[\mathbf{c}_0]}$, which cut out the Hilbert scheme $\sC^{[\mathbf{c}_0]}:=\Hilb^{\mathbf{c}_0}(\sC)$.  Therefore 
\begin{equation}\label{eqn_Euler_obstruction_bundle2}
[\rM^{(2)}]^{\vir}
=(-1)^{\rk}\sC^{[\mathbf{c}_0]}\subset  \Hilb^{\mathbf{c}_0}(\sS)=\rM^{(2)}.
\end{equation}

\begin{rmk}
The Hilbert scheme $\sC^{[\mathbf{c}_0]}$ depends on the $K$-group class $\mathbf{c}_0\in K_0(\sS)$. 
\end{rmk}

\subsubsection{Virtual normal bundle}\label{subsec_virtual_normal_case_Z1}

The calculation of the virtual normal bundle $N^{\vir}$  of $\rM^{(2)}$ is the same as in \cite[\S 8.3]{TT1}, which is given by the moving part of (\ref{eqn_calculation_exact_sequence_obstruction2}):
$$\Gamma(K_{\sS}|_{\sZ_0})\Tt\oplus R\Gamma(\sI_0 K_{\sS}^2)\Tt^2\oplus R\Gamma(\sI_0 K^2_{\sS})^{\vee}\Tt^{-1}[-1]\oplus T_{\sZ_0}^*\Hilb^{\mathbf{c}_0}(\sS)\Tt[-1]$$
at $\sZ_0\in \rM^{(2)}$.  Then we calculate the virtual normal bundle $N^{\vir}$ by noting that
$$R\Gamma(\sI_0 K_{\sS}^2)=H^0(K_{\sS^2})-H^0(K_{\sS}|_{\sZ_0});$$
and $N^{\vir}$ is:
$$[K_{\sS}^{[\mathbf{c}_0]}]\Tt+(\Tt^2)^{\oplus\dim H^0(K_{\sS}^2)}
-[(K_{\sS}^2)^{[\mathbf{c}_0]}]\Tt^2-(\Tt^{-1})^{\oplus\dim H^0(K_{\sS}^2)}
+[((K_{\sS}^2)^{[\mathbf{c}_0]})^*]\Tt^{-1}-\Big[T_{\Hilb^{\mathbf{c}_0}(\sS)}\Big]\Tt.$$

Since $\sC\to C$ is a $\mu_r$-gerbe, we can just write the Hilbert scheme $\sC^{[\mathbf{c}_0]}$ as 
$\sC^{[n]}$ for some integer $n\in \zz_{\geq 0}$. 
So we calculate the virtual Euler class:
\begin{align*}
\frac{1}{e(N^{\vir})}&=\frac{e((K_{\sS})^{[n]})\Tt^2)\cdot e(\Tt^{-1})^{\oplus\dim H^0(K_{\sS}^2)}\cdot e(T^*_{\Hilb^{\mathbf{c}_0}(\sS)}\Tt)}
{e(K_{\sS}^{[n]}\Tt)\cdot e((\Tt^2)^{\oplus \dim H^0(K_{\sS}^2)})\cdot e(((K_{\sS}^2)^{[n]})^* \Tt^{-1})}\\
&=\frac{(2t)^n\cdot c_{\frac{1}{2t}}((K_{\sS}^2)^{[n]})\cdot (-t)^{\dim H^0(K_{\sS}^2)}\cdot t^{2n}\cdot c_{\frac{1}{t}}(T^*_{\Hilb^n(\sS)})}
{t^n\cdot c_{\frac{1}{t}}((K_{\sS})^{[n]})\cdot (2t)^{\dim H^0(K_{\sS}^2)}\cdot (-1)^n\cdot t^n c_{\frac{1}{t}}((K_{\sS}^2)^{[n]})}\\
&=(-2)^{n-\dim}\cdot t^n\cdot 
\frac{c_{\frac{1}{2t}}((K_{\sS}^2)^{[n]})\cdot c_{-\frac{1}{t}}(T_{\Hilb^n(\sS)})}
{c_{\frac{1}{t}}((K_{\sS})^{[n]})\cdot c_{\frac{1}{t}}((K_{\sS}^2)^{[n]})}
\end{align*}
where 
$$c_s(E):=1+sc_1(E)+\cdots+s^r c_r(E),$$
and when $s=1$, $c_s(E)$ is the total Chern class of $E$.  By the arguments of the degree, we calculate the case $t=1$. 
Also since $\sC^{[n]}$ is cut out of the section $s^{[n]}$ on $K_{\sS}^{[n]}$, we have 
$$T_{\Hilb^{\mathbf{c}_0}}(\sS)|_{\sC^{[n]}}=T_{\sC^{[n]}}\oplus K_{\sS}^{[n]}|_{\sC^{[n]}}$$
in $K$-theory. 
Therefore
\begin{equation}\label{eqn_rM_integral}
\int_{[\rM^{(2)}]^{\vir}}\frac{1}{e(N^{\vir})}=(-2)^{-\dim}\cdot 2^n\cdot 
\int_{[\sC^{[n]}]}\frac{c_{\frac{1}{2}}((K_{\sS}^2)^{[n]})\cdot c_{-1}(T_{\sC^{[n]}})\cdot c_{-1}(K_{\sS}^{[n]})}
{c_{\bullet}(K_{\sS}^{[n]})\cdot c_{\bullet}((K_{\sS}^2)^{[n]})}.
\end{equation}

\subsubsection{The Hilbert scheme of points on gerby curves}

Before calculating further for the integral, we prove several statements of Hilbert scheme of points on the $\mu_r$-gerby curve $p: \sC\to C$. 

\begin{prop}\label{prop_Hilbert_scheme_gerbe}
Let $ \sC^{[n]}$ be the Hilbert scheme of $n$-points on $\sC$.  Then we have 
$$ \sC^{[n]}\cong  C^{[n]},$$
where $C^{[n]}$ is the Hilbert scheme of $n$-points on the coarse moduli space $C$. 
\end{prop}
\begin{proof}
By definition of the Hilbert scheme $\sC^{[n]}$, an element $\sZ\in \sC^{[n]}$ is given by an exact sequence 
\begin{equation}\label{eqn_gerby)_curve_1}
0\to I_{\sZ}\rightarrow \sO_{\sC}\rightarrow \sO_{\sZ}\to 0
\end{equation}
where $\sZ$ is a zero-dimensional closed substack in $\sC$  of length $n$ with ideal sheaf $I_{\sZ}$.  Since $\sC$ is a 
$\mu_r$-gerbe over $C$, the above exact sequence immediately gives the exact sequence 
\begin{equation}\label{eqn_gerby)_curve_2}
0\to I_{Z}\rightarrow \sO_{C}\rightarrow \sO_{Z}\to 0
\end{equation}
for $Z=p(\sC)\subset C$, and 
$p_*\sO_{\sZ}=\sO_Z$. 
Since $p_*$ is exact and $p_*\sO_{\sC}=\sO_C$,  we have $p_{*}I_{\sZ}=I_Z$. 
Thus the second exact sequence (\ref{eqn_gerby)_curve_2}) is the pushforward of the first one (\ref{eqn_gerby)_curve_1}). 
The zero dimensional closed subscheme $Z$ also has length $n$ which defines an element in $C^{[n]}$. 

Suppose we have an exact sequence (\ref{eqn_gerby)_curve_2}), then taking pullback $p^*$ we get the first one (\ref{eqn_gerby)_curve_1}), which shows 
the above correspondence is reversed. 
The above proof works in families of ideal sheaves in $\sC$ and $C$, 
therefore we have an isomorphism between Hilbert schemes. 
\end{proof}

\subsubsection{Calculations via tautological classes}

From Proposition \ref{prop_Hilbert_scheme_gerbe}, one can use the calculation for the Hilbert scheme of points on $C$ in \cite[\S 8.4]{TT1} to calculate the integral on the Hilbert scheme of points on the gerby curve $\sC$.
Let us first review the tautological classes on $C^{[n]}$ for the smooth curve $C$. 
Let 
$$\omega:=\PD[C^{[n-1]}]\in H^2(C^{[n]},\zz)$$
where $C^{[n-1]}\subset C^{[n]}$ is a smooth divisor given by $Z\mapsto Z+x$ for a base point $x\in C$, and $\PD$ represents the Poincare dual. 
The second one is given by the Abel-Jacobi map:
$$\AJ: C^{[n]}\to \Pic^n(C); \quad   Z\mapsto \sO(Z).$$
Since tensoring with power of $\sO(x)$ makes the $\Pic^n(C)$ isomorphic for all $n$, the pullback of the theta divisor from $\Pic^{g-1}(C)$  gives a cohomology class 
$$\theta\in H^2(\Pic^n(C),\zz)\cong \Hom(\Lambda^2H^1(C,\zz),\zz).$$
Still let $\theta$ to denote its pullback 
$\AJ^*\theta$, so 
$$\theta\in H^2(C^{[n]},\zz),$$
which is the second tautological class.  The basic property (\cite[\S I.5]{ACGH}) is:
\begin{equation}\label{eqn_basic_formula1}
\int_{C^{[n]}}\frac{\theta^i}{i!}\omega^{n-i}=\mat{c}g\\
i\rix,
\end{equation}
and 
$$
\begin{cases}
c_t(T_{C^{[n]}})=(1+\omega t)^{n+1-g}\exp\left(\frac{-t\theta}{1+\omega t}\right);\\
c_t(\sL^{[n]})=(1-\omega t)^{n+g-1-\deg \sL}\exp\left(\frac{t\theta}{1-\omega t}\right).
\end{cases}
$$

The canonical line bundle $K_{\sS}|_{\sC}=\sO_{\sS}(\sC)|_{\sC}$ has the same degree of the restriction 
$K_{S}|_{C}=\sO_{S}(C)|_{C}$.  Therefore the vector bundle $K_{\sS}^{[n]}|_{\sC}$ on $\sC^{[n]}$ is the same as the 
vector bundle $K_{S}^{[n]}|_{C}$ on $C^{[n]}$. 
Thus 
\begin{align}\label{eqn_key_formula1}
&\text{Right side of ~} (\ref{eqn_rM_integral})=   \\
&(-2)^{\dim}\cdot 2^n\int_{\sC^{[n]}}
\frac{(1-\frac{\omega}{2})^{n+1-g}\cdot e^{\frac{\theta}{2-\omega}}\cdot (1-\omega)^{n+1-g}\cdot e^{\frac{\theta}{1-\omega}}\cdot (1+\omega)^n\cdot e^{\frac{-\theta}{1+\omega}}}
{(1-\omega)^n\cdot e^{\frac{\theta}{1-\omega}}\cdot (1-\omega)^{n+1-g}\cdot e^{\frac{\theta}{1-\omega}}}\nonumber \\
&=(-2)^{\dim}\cdot 2^n\int_{\sC^{[n]}}
\frac{(2-\omega)^{n+1-g}}{2^{n+1-g}}\cdot \frac{(1+\omega)^n}{(1-\omega)^n}\cdot 
e^{\frac{\theta}{2-\omega}-\frac{\theta}{1+\omega}-\frac{\theta}{1-\omega}} \nonumber \\
&=(-2)^{\dim}\cdot 2^{g-1}(-1)^{n+1-g}\cdot \int_{\sC^{[n]}}
(\omega-2)^{n+1-g}\cdot \frac{(1+\omega)^n}{(1-\omega)^n}\cdot 
e^{\frac{\theta}{2-\omega}-\frac{\theta}{1+\omega}-\frac{\theta}{1-\omega}}. \nonumber
\end{align}

Now we use (\ref{eqn_basic_formula1}) and Proposition \ref{prop_Hilbert_scheme_gerbe} to get 
\begin{align*}
\int_{\sC^{[n]}}\frac{(\theta)^i}{i!}\cdot (\omega)^{n-i}
=\int_{C^{[n]}} \frac{(\theta)^i}{i!}\cdot (\omega)^{n-i}=\mat{c}g\\i\rix.
\end{align*}
Hence whenever we have $\frac{(\theta)^i}{i!}$ in the integrand involving only power of $\omega$ we can replace it by 
$\mat{c}g\\i\rix (\omega)^{n-i}$.  Therefore for $\alpha$ a power series of $\omega$, 
$$e^{\alpha (\theta)}=\sum_{i=0}^{\infty}\alpha^i \frac{(\theta)^i}{i!}\sim
\sum_{i=0}^{\infty}\alpha^i \mat{c}g\\i\rix (\omega)^{n-i}=(1+\alpha (\omega))^g.$$

When we do the integration against $\sC^{[n]}$, $\sim$ becomes equality, and 
(\ref{eqn_key_formula1}) is: 
\begin{align}\label{eqn_key_formula2}
&(-2)^{\dim}\cdot 2^{g-1}(-1)^{n+1-g}\cdot \int_{\sC^{[n]}}
(\omega-2)^{n+1-g}\cdot \frac{(1+\omega)^n}{(1-\omega)^n}\cdot 
(1+\frac{\omega}{2-\omega}-\frac{\omega}{1+\omega}-\frac{\omega}{1-\omega})^g
 \\
&=(-2)^{\dim}\cdot (-2)^{g-1}(-1)^{n}\cdot \int_{\sC^{[n]}}
(\omega-2)^{n+1-2g}\cdot \frac{(1+\omega)^{n-g}}{(1-\omega)^{n+g}}\cdot (4\omega-2)^g. \nonumber
\end{align}

\subsubsection{Writing the generating function}\label{subsec_generating_function}

Recall that for the $\cc^*$-fixed Higgs pair $(E,\phi)$, we fix 
$$(\widetilde{\Ch}_1)^2(E)=2. \quad (\widetilde{\Ch}_1)^1(E)=-c_1(\sS).$$
For $g_i=\zeta^i\in\mu_r (1\leq i\leq r-1)$, $E=E_{g_i}^1\oplus E_{g_i}^2$ is the decomposition under the $g_i$-action into eigen-subsheaves.  We calculate and denote by 
\begin{equation}\label{eqn_Ci_C1}
\beta_{g_i}:=(\widetilde{\Ch}_{g_i})^1(E)=e^{2\pi i \frac{f_{i1}}{r}}+e^{2\pi i \frac{f_{i2}}{r}}.
\end{equation}
Then we let
\begin{equation}\label{eqn_Ci_C2}
n_{i}:=(\widetilde{\Ch}_{g_i})^0(E)=e^{2\pi i \frac{f_{i1}}{r}}c_1(E_{g_i}^1)+e^{2\pi i \frac{f_{i2}}{r}}c_1(E_{g_i}^2).
\end{equation}

 We introduce variables $q$ to keep track of the second Chern class $c_2(E)$ of the torsion free sheaf $E$, 
$q_1, \cdots, q_{r-1}$ to keep track of the classes $n_i$ for $i=1,\cdots, r-1$. 
Then we write 
\begin{align*}
&(-2)^{-\dim}\cdot (-2)^{1-2g}(-1)^{n}\cdot \sum_{n=0}^{\infty}q^n\cdot q_1^{n_1}\cdots q_{r-1}^{n_{r-1}} \cdot \int_{\sC^{[n]}}\frac{1}{N^{\vir}}\\
&=\sum_{n=0}^{\infty}q^n \cdot q_1^{n_1}\cdots q_{r-1}^{n_{r-1}} \cdot \int_{\sC^{[n]}}(\omega-2)^{n+1-2g}\cdot \frac{(1+\omega)^{n-g}}{(1-\omega)^{n+g}}\cdot (1-2\omega)^g. 
\end{align*}
\begin{rmk}
Since the moduli of stable Higgs pairs on $\sS$ is isomorphic to the moduli space of  parabolic  Higgs pairs on 
$(S,C)$, see \cite{Jiang3}. We will see that the variables $q_1, \cdots, q_{r-1}$ will keep track of the parabolic degree of the sheaf $E$ on the curve $C\subset S$. 
\end{rmk}

For simplicity, we deal with the case $q_1=\cdots=q_{r-1}=1$. 
So 
\begin{align}\label{eqn_key_formula3}
&(-2)^{-\dim}\cdot (-2)^{1-2g}(-1)^{n}\cdot \sum_{n=0}^{\infty}q^n \int_{\sC^{[n]}}\frac{1}{N^{\vir}}\\
&=\sum_{n=0}^{\infty}q^n \int_{\sC^{[n]}}(\omega-2)^{n+1-2g}\cdot \frac{(1+\omega)^{n-g}}{(1-\omega)^{n+g}}\cdot (1-2\omega)^g. \nonumber
\end{align}

Since $\sC^{[n]}\cong C^{[n]}$,  we have 
\begin{align}\label{eqn_key_formula4}
&(-2)^{-\dim}\cdot (-2)^{1-2g}(-1)^{n}\cdot \sum_{n=0}^{\infty}q^n \int_{\sC^{[n]}}\frac{1}{N^{\vir}}\\
&=\sum_{n=0}^{\infty}q^n \int_{C^{[n]}}(\omega-2)^{n+1-2g}\cdot \frac{(1+\omega)^{n-g}}{(1-\omega)^{n+g}}\cdot (1-2\omega)^g. \nonumber
\end{align}

Then we perform the same careful Contour integral calculations as in \cite[\S 8.5]{TT1} by using \cite[\S 6.3]{St2}. 
Introduce the sum series
\begin{equation}\label{eqn_key_series}
\sum_{\substack{i=0\\
n=0}}^{\infty}x^i\cdot t^n\cdot \int_{C^{[i]}}(\omega-2)^{n+1-2g}\cdot \frac{(1+\omega)^{n-g}}{(1-\omega)^{n+g}}\cdot (1-2\omega)^g.
\end{equation}

Then (\ref{eqn_key_formula4}) is just the diagonal part of (\ref{eqn_key_series}). 
Summing over $n$ on (\ref{eqn_key_series}) gives 
$$\sum_{i=0}^{\infty}x^i\int_{\sC^{[i]}}
\frac{(1-\omega)^g}{(\omega-2)^{2g-1}\cdot (1-\omega^2)^g}
\left(1-t\frac{(\omega-2)(1+\omega)}{1-\omega}\right)^{-1}.
$$
Using the fact $\int_{C^{[i]}}\omega^j=\delta_{ij}$, and  replacing $\omega$ by $x$, the above sum is:
$$\frac{(1-2x)^g}{(x-2)^{2g-1}(1-x^2)^g}\cdot \frac{1-x}{1-x-t(x^2-x-2)}.$$
To find the diagonal of (\ref{eqn_key_formula3}), let $t=q/x$, and consider the integral 
$$\frac{1}{2\pi i}\oint \frac{(1-2x)^g}{(x-2)^{2g-1}(1-x^2)^g}\cdot \frac{1-x}{1-x-\frac{q}{x}(x^2-x-2)}\frac{dx}{x}$$
around a loop containing only those poles which tend to zero as $q\to 0$.  
So (\ref{eqn_key_formula4}), hence (\ref{eqn_key_formula3}) is the residue of 
$$ \frac{(1-2x)^g}{(x-2)^{2g-1}(1-x^2)^g}\cdot \frac{-(1-x)}{(1+q)x^2-(1+q)x-2q}$$
at the root 
$$x_0=\frac{1}{2}\left(1-\sqrt{1+\frac{8q}{1+q}}\right)$$
of $(1+q)x^2-(1+q)x-2q$ in $x$, which is 
$$\frac{(1-2x_0)^g}{(x_0-2)^{2g-1}(1+x_0)^g(1-x_0)^{g-1}}\cdot \frac{-1}{(1+q)(x_0-x_1)}$$
where $x_1=\frac{1}{2}\left(1+\sqrt{1+\frac{8q}{1+q}}\right)$ is the other root.  Then the same calculation as in 
\cite[\$ 8.5]{TT1} gives the following result:

\begin{thm}\label{thm_generating_function_root_quintic}
We have 
\begin{equation}\label{eqn_prop_formula}
\sum_{n=0}^{\infty}q^n \int_{\sC^{[n]}}\frac{1}{N^{\vir}}=
A\cdot (1-q)^{g-1}\left(1+\frac{1-3q}{\sqrt{(1-q)(1-9q)}}\right)^{1-g},
\end{equation}
where 
$A:=(-2)^{\dim}\cdot (-2)^{2g-1}$.    $\square$
\end{thm}

\begin{rmk}
Since we are calculating the case that  in the Higgs pair $(E,\phi)$, $E=\sI_0\otimes K_{\sS}\oplus \sO\cdot \mathbf{t}^{-1}$, the action of $\mu_r$ can be classified. The canonical line bundle $K_{\sS}=\sO_{\sS}(\sC)$, which makes the action of 
$\mu_r$ on $K_{\sS}$ by $\zeta^i\cdot K_{\sS}$ for $\zeta^i\in\mu_r$.  The cyclic group $\mu_r$ is naturally embedded into 
$\cc^*$, and the $\cc^*$-action on $\sO\cdot \mathbf{t}^{-1}$ induces the action of $\mu_r$ on it by 
$\zeta^{r-i}\cdot \sO$ for $\zeta^i\in\mu_r$.  So in this case for $g_i=\zeta^i (1\leq i\leq r-1)$ 
$$\beta_{g_i}=e^{2\pi i \frac{i}{r}}+e^{2\pi i \frac{r-i}{r}}$$
and 
$$n_i=e^{2\pi i \frac{i}{r}}(-c_1).$$
A concrete calculation for the generating series with $q_1,\cdots, q_{r-1}$ nonzero is given for the case $r=2$ in the next section. 
\end{rmk}

\subsubsection{The case of $r=2$}

Let us include an easy case that $r=2$ so that $\sS=\sqrt[2]{(S,C)}$ is a $2$-th root stack over $S$. 
The inertia stack 
$$I\sS=\sS\sqcup \sC$$
where $\sC$ is the component corresponding to the nontrivial element $\zeta\in \mu_2$.  Then in this case $\beta_{\zeta}$ has three possibilities:
$$\beta_{\zeta}=-2, 0, 2.$$
In this case $E=\sI_0\otimes K_{\sS}\oplus \sO\cdot \mathbf{t}^{-1}$, the corresponding actions of $\mu_2$ on the first  and second summands are all nontrivial.  Hence $\beta_{\zeta}=-2$, and  the corresponding $n_1=c_1:=c_1(\sS)$. 

Recall that the variable $q$ keeps track of the second Chern class $c_2(E)$, and 
$q_1$ keeps track of the classes $n_1$.  Since $n_1=c_1$, we have the same formula as in Theorem \ref{thm_generating_function_root_quintic}:
$$
\sum_{n=0}^{\infty}q^n\cdot q_1^{c_1} \int_{\sC^{[n]}}\frac{1}{N^{\vir}}=q_1^{c_1}\cdot 
A\cdot (1-q)^{g-1}\left(1+\frac{1-3q}{\sqrt{(1-q)(1-9q)}}\right)^{1-g},
$$
where 
$A:=(-2)^{\dim}\cdot (-2)^{2g-1}$.

\subsection{Quintic surfaces with ADE singularities}\label{subsec_quintic_ADE}

In this section we consider the quintic surface $\sS$ with isolated ADE singularities as in \S \ref{eqn_ADE_singularities}.  We take $\sS$ as a surface DM stack.  From \cite{Horikava}, the coarse moduli space of the DM stack $\sS$ lies in the 
component consisting of  smooth quintic surfaces in the moduli space of general type surfaces with topological invariants in 
(\ref{eqn_topology_invariants_quintic}).  This means that there exists a deformation family such that the smooth quintic surfaces can be deformed into quintic surfaces with ADE singularities.

Let us fix a quintic surface $\sS$, with $P_1,\cdots, P_s\in \sS$ the isolated singular points with ADE type. 
Let $G_1,\cdots, G_s$ be the local ADE finite group in $SU(2)$ corresponding to $P_1,\cdots, P_s\in \sS$. 
We use the notation $|G_i|$ to represent the set of conjugacy classes for $G_i$.

From \cite{Horikava}, the canonical sheaf  $K_{\sS}$ has no base point, therefore it is a line bundle over $\sS$ which is the canonical line bundle.  All the theory we constructed works for $\sS$. 
Still let $\sX=\mbox{Tot}(K_{\sS})$ be the total space of $K_{\sS}$, which is a Calabi-Yau smooth DM stack.  Choose a generating sheaf $\Xi$ on $\sS$ such that it contains all the irreducible representations of the local group $G_i$ of $P_i$.  
Fix a $K$-group class $\mathbf{c}\in K_0(\sS)$ (determining a Hilbert polynomial $H$),  and let $\N$ be the moduli space of stable Higgs pairs with $K$-group class $\mathbf{c}$. We work on the Vafa-Witten invariants 
$\VW$ for the moduli space $\N^{\perp}_{\mathbf{c}}$ of stable fixed determinant $K_{\sS}$ and 
trace-free  Higgs pairs with $K$-group class $\mathbf{c}$.

For $\sS$, the trivial line bundle  
$\sO_{\sS}$ is the only line bundle $L$ satisfying $0\leq \deg L\leq \frac{1}{2}\deg K_{\sS}$ where the degree is defined  by 
$\deg L=c_1(L)\cdot c_1(\sO_{\sS}(1))$.
In fact all the arguments in \S \ref{subsec_fixed_pairs} work for such a surface DM stack $\sS$.  The $\cc^*$-fixed Higgs pair $(E,\phi)$ in the component $\rM^{(2)}$ satisfies the following properties:
$$E=E_{0}\oplus E_{-1}$$
where $E_0$ and $E_{-1}$ all have rank one.  
For such a surface DM stack $\sS$, let 
$$N:=\max\{|G_i|:  G_i \text{~are local orbifold groups of~}P_1, \cdots, P_s\},$$
then the generating sheaf $\Xi$ we choose is $\Xi=\bigoplus_{j=0}^{N-1}\sO_{\sS}(j)$. 
We have a lemma similar to Lemma \ref{lem_E0E1}. 
\begin{lem}\label{lem_E0E1_ADE}
We still have 
the inequality (\ref{eqn_inequality}) for this generating sheaf and it  implies that 
$$\det(E_{-1})=\sO_{\sS}; \text{~and~} \det(E_{0})=K_{\sS}.$$
\end{lem}
\begin{proof}
The argument of inequality (\ref{eqn_inequality}) is the same as in the root stack case.  
Since the choice of the generating sheaf is $\Xi=\bigoplus_{j=0}^{N-1}\sO_{\sS}(j)$, we will see that the modified Hilbert polynomials 
$$H_{\Xi}(E,m) = \chi(\sS,E\otimes \Xi^{\vee}\otimes p^*\sO_S(m))$$
will be the same as $H(E, m)$ for sufficiently large $m$ for any $E$.  So this is reduced to the generating sheaf $\Xi=\sO_{\sS}$ case.  The result is true from the arguments as in \cite[\S 8]{TT1}.
\end{proof}

Therefore we have 
$E_0=\sI_0\otimes K_{\sS}$, and $E_{-1}=\sI_{1}\otimes \Tt^{-1}$
for some ideal sheaves $\sI_i$ and $\Tt$ is the standard one dimensional representation of $\cc^*$.
The morphism $\sI_0\to \sI_1$ is nonzero, therefore must satisfy  
$\sI_0\subseteq \sI_1$.  There exist two zero dimensional substacks parametrized by $\sI_0\subseteq \sI_1$.

The inertia stack  $I\sS$
is indexed by the element $g_i\in |G_i|$, $\sS_{g_i}=BG_i$ is the component corresponding to $g_i$. Then 
$$H^*(I\sS, \qq)=H^*(\sS)\bigoplus \oplus_{i=1}^{s}H^*(BG_i).$$
The cohomology of $H^*(BG_i)$ is isomorphic to $H^*(\pt)$. 
The orbifold Chern character morphism is defined by:
\begin{equation}\label{eqn_orbifold_Chern_character_isolated}
\widetilde{\Ch}: K_0(\sS)\to H^*_{\CR}(\sS,\qq)=H^*(I\sS, \qq)
\end{equation}
where $H^*_{\CR}(\sS,\qq)$ is the Chen-Ruan cohomology of $\sS$.  

For any coherent sheaf $E$, the restriction of $E$ to every $\sS_{g_i}=BG_i$ has a $G_i$-action such that it acts by 
$e^{2\pi i\frac{f_i}{\ord(G_i)}}$.  Let  
\begin{equation}\label{eqn_Chern_character_value1_isolated}
\widetilde{\Ch}(E)=(\Ch(E), \oplus_{i=1}^{s}\oplus_{g_i\in |G_i|}\Ch(E|_{BG_i})),
\end{equation}
where 
$$\Ch(E)=(\rk(E), c_1(E), c_2(E))\in H^*(\sS), \quad 
\Ch(E|_{BG_i})=(e^{2\pi i\frac{f_i}{\ord(G_i)}}\rk(E))\in H^*(BG_i).$$

We still let 
$$\widetilde{\Ch}(E)=(\widetilde{\Ch}_{g_i}(E))$$
where $\widetilde{\Ch}_{g_i}(E)$ is the component in $H^*(\sS_{g_i})=H^*(BG_i)$ as in (\ref{eqn_Chern_character_value1_isolated}). 
Still define:
\begin{equation}\label{eqn_Chern_character_degree_isolated}
\left(\widetilde{\Ch}_{g_i}\right)^k:=\left(\widetilde{\Ch}_{g_i}\right)_{\dim \sS_{g_i}-k}\in H^{\dim \sS_{g_i}-k}(\sS_{g_i}).
\end{equation}
If we have a rank $2$ $\cc^*$-fixed Higgs pair $(E,\phi)$ with fixed $c_1(E)=-c_1(\sS)$, then 
$\left(\widetilde{\Ch}_{g_i}\right)^2(E)=2$ only when $g_i=1$, which is  the rank; 
and 
$$
\left(\widetilde{\Ch}_{g_i}\right)^1(E)=-c_1(\sS)
$$
only when $g_i=1$, and 
$$
\left(\widetilde{\Ch}_{g_i}\right)^0(E)=
\begin{cases}
c_2(E), & g_i=1;\\
2e^{2\pi i\frac{f_i}{\ord(G_i)}}, & g_i \neq 1.
\end{cases}
$$

Therefore we have a similar  proposition as in Proposition (\ref{prop_second_fixed_loci_Hilbert_scheme}):
\begin{prop}\label{prop_second_fixed_loci_Hilbert_scheme_isolated}
In the case that the ranks of stable Higgs sheaves are $2$, we fix a $K$-group class 
$\mathbf{c}\in K_0(\sS)$ such that $\left(\widetilde{\Ch}_{\xi}\right)^1(\mathbf{c})=-c_1(\sS)$.  Then 
\begin{enumerate}
\item If $c_2(E)<0$, then the $\cc^*$-fixed locus is empty by the Bogomolov inequality which holds for surface DM stacks.
\item If $c_2(E)\geq 0$,  then 
$$\rM^{(2)}\cong \bigsqcup_{\alpha\in F_0K_0(\sS)}\Hilb^{\alpha, \mathbf{c}_0-\alpha}(\sS)$$
where $\mathbf{c_0}\in F_0K_0(\sS)$ such that $\left(\widetilde{\Ch}_{g}\right)^0(\mathbf{c}_0)=\left(\widetilde{\Ch}_{g}\right)^0(\mathbf{c})$; and $\Hilb^{\alpha, \mathbf{c}_0-\alpha}(\sS)$ is the nested Hilbert scheme of zero-dimensional substacks of $\sS$:
$$\sZ_1\subseteq \sZ_0$$
such that $[\sZ_1]=\alpha$, $[\sZ_0]=\mathbf{c}_0-\alpha$.  $\square$
\end{enumerate}
\end{prop}
\begin{proof}
This case of the Bogomolov inequality is treated in  \cite{JKT}.
\end{proof}

All the statement in (\ref{subsec_case_Z1_empty}) works for this $\sS$, until Formula (\ref{eqn_Euler_obstruction_bundle})
\begin{equation}\label{eqn_Euler_obstruction_bundle_isolated}
[\rM^{(2)}]^{\vir}=e((K_{\sS}^{[\mathbf{c}_0]})^*)\cap [\Hilb^{\mathbf{c}_0}(\sS)]
=(-1)^{\rk}e(K_{\sS}^{[\mathbf{c}_0]})\cap [\Hilb^{\mathbf{c}_0}(\sS)].
\end{equation}

There is a section $s$ of $K_{\sS}$ cutting out of a smooth curve $C$ and $g:=g_C=6$ as in (\ref{eqn_topology_invariants_quintic}).  
Also fixing $\mathbf{c}_0\in K_{0}(\sS)$ means that the second Chern class $c_2=n$ is fixed. 
Then this section $s$ induces a section $s^{[\mathbf{c}_0]}$ on the tautological vector bundle $K_{\sS}^{[\mathbf{c}_0]}$, which cuts out the Hilbert scheme $C^{[n]}:=\Hilb^{\mathbf{c}_0}(C)$.  Therefore 
\begin{equation}\label{eqn_Euler_obstruction_bundle2_isolated}
[\rM^{(2)}]^{\vir}
=(-1)^{\rk}C^{[n]}\subset  \Hilb^{n}(\sS)=\rM^{(2)}.
\end{equation}

The arguments of virtual normal bundle is similar as in  (\ref{subsec_case_Z1_empty}), and 
we also have 
$$T_{\Hilb^{n}(S)}|_{C^{[n]}}=T_{C^{[n]}}\oplus K_{\sS}^{[n]}|_{C^{[n]}}$$
in $K$-theory. 
Therefore
\begin{equation}\label{eqn_rM_integral_isolated}
\int_{[\rM^{(2)}]^{\vir}}\frac{1}{e(N^{\vir})}=(-2)^{-\dim}\cdot 2^n\cdot 
\int_{[C^{[n]}]}\frac{c_{\frac{1}{2}}((K_{\sS}^2)^{[n]})\cdot c_{-1}(T_{C^{[n]}})\cdot c_{-1}(K_{\sS}^{[n]})}
{c_{\bullet}(K_{\sS}^{[n]})\cdot c_{\bullet}((K_{\sS}^2)^{[n]})}.
\end{equation}

As in \S \ref{subsec_generating_function}, 
 we still use variables $q$ to keep track of the second Chern class $c_2(E)$ of the torsion free sheaf $E$, 
$q_1, \cdots, q_{\ord{|G_i|}-1}$ to keep track of the classes $c_0(E|_{BG_i})$. 
Let $q_1=\cdots =q_{\ord{|G_i|}-1}=1$, we can write down the generating function in 
\S \ref{subsec_generating_function}.
Thus all the calculations as in \cite[\S 8.4, \S 8.5]{TT1}  go through in this case, and we get the same result:

\begin{thm}\label{thm_generating_function_ADE_quintic}
Let $\sS$ be a quintic surface with isolated ADE singularities.  Then we the generating function: 
\begin{equation}\label{eqn_prop_formula}
\sum_{n=0}^{\infty}q^n \int_{C^{[n]}}\frac{1}{N^{\vir}}=
A\cdot (1-q)^{g-1}\left(1+\frac{1-3q}{\sqrt{(1-q)(1-9q)}}\right)^{1-g},
\end{equation}
where 
$A:=(-2)^{\dim}\cdot (-2)^{2g-1}$.    $\square$
\end{thm}

\begin{rmk}
For surface DM stacks $\sS$ with only ADE singularities, the  Vafa-Witten invariants in Theorem \ref{thm_generating_function_ADE_quintic} are the same as in \cite[\S 8.5]{TT1}.  This may explain  that such surfaces $\sS$ can be put into a deformation family $\sS\to \aaa^1$ such that the generic fibers are  smooth quintic surfaces, and the central fibre is the surface with ADE singularities. The deformation invariance of the virtual cycle in (\ref{eqn_Euler_obstruction_bundle_isolated}) and (\ref{eqn_Euler_obstruction_bundle2_isolated}) implies that the invariants are the same.
\end{rmk}

\subsubsection{One vertical term}

We perform one more step to calculate one vertical term as in \cite{TT1}, and explain this time it will not give the same invariants as in the smooth case. 

This case is that $[\sZ_1]=[\sZ_0]\in L_0(\sS)$ component in $\rM^{(2)}$ in Proposition \ref{prop_second_fixed_loci_Hilbert_scheme_isolated}. 
So in this case $\Phi: \sI_0\to \sI_1$ is an isomorphism. Therefore
$$\Hilb^{\mathbf{c}_0, \mathbf{c}_0}(\sS)=\Hilb^{\mathbf{c}_0}(\sS)$$
and 
$$E=\sI_{\sZ}\otimes K_{\sS}\oplus \sI_{\sZ}\cdot \Tt^{-1}; \quad
\phi=\mat{cc}0&0\\
1&0\rix : E\to E\otimes K_{\sS}\cdot \Tt,$$
where $\sZ\subset \sS$ is a zero dimensional substack with $K$-group class $\mathbf{c}_0$. 
We use the same arguments as in \cite[\S 8.7]{TT1} for the torsion sheaf 
$\sE_{\phi}$ on $\sX$, which is the twist 
\begin{equation}\label{eqn_FZ}
\sF_{\sZ}:=(\pi^*\sI_{\sZ}\otimes \sO_{2\sS})
\end{equation}
by $\pi^*K_{\sS}$. Look at the following exact sequence:
$$0\to \pi^*\sI_{\sZ}(-2\sS)\longrightarrow \pi^*\sI_{\sZ}\longrightarrow \sF_{\sZ}\to 0,$$
we have
$$R\Hom(\sF_{\sZ},\sF_{\sZ})\to R\Hom(\pi^*\sI_{\sZ}, \sF_{\sZ})\to 
R\Hom(\pi^*\sI_{\sZ}, \pi_*\sF_{\sZ}(2\sS)).$$
The second arrow is zero since the section $\sO(2\sS)$ cutting out $2\sS\subset \sX$ annihilates $\sF_{\sZ}$. So by adjunction and the formula $\pi_*\sF_{\sZ}=\sI_{\sZ}\oplus \sI_{\sZ}\otimes K_{\sS}^{-1}\cdot \Tt^{-1}$, we have 
\begin{multline*}
R\Hom(\sF_{\sZ},\sF_{\sZ})\cong R\Hom_{\sS}(\sI_{\sZ},\sI_{\sZ})\oplus R\Hom_{\sS}(\sI_{\sZ},\sI_{\sZ}\otimes K_{\sS}^{-1})\Tt^{-1} \\
\oplus  R\Hom_{\sS}(\sI_{\sZ},\sI_{\sZ}\otimes K^2_{\sS})\Tt^{2}[-1]\oplus  R\Hom_{\sS}(\sI_{\sZ},\sI_{\sZ}\otimes K_{\sS})\Tt[-1].
\end{multline*}
We calculate the perfect obstruction theory
$$R\Hom_{\sX}(\sF_{\sZ},\sF_{\sZ})_{\perp}[1]$$
which comes from taking trace-free parts of the first and last terms and we have $\Hom_{\perp}=\Ext^3_{\perp}=0$. 
We have:
\begin{multline}\label{eqn_FZ_deformation}
\Ext^1_{\sX}(\sF_{\sZ},\sF_{\sZ})_{\perp}= \Ext^1_{\sS}(\sI_{\sZ},\sI_{\sZ})\oplus \Ext^1_{\sS}(\sI_{\sZ},\sI_{\sZ}\otimes K_{\sS}^{-1})\Tt^{-1} \\
\oplus  \Hom_{\sS}(\sI_{\sZ},\sI_{\sZ}\otimes K^2_{\sS})\Tt^{2}\oplus  \Hom_{\sS}(\sI_{\sZ},\sI_{\sZ}\otimes K_{\sS})_0\Tt.
\end{multline}
The obstruction $\Ext^2_{\perp}$ is just the dual of the above tensored with $\Tt^{-1}$.  The first term in (\ref{eqn_FZ_deformation}) is $T_{\sZ}\Hilb^{\mathbf{c}_0}(\sS)$, the fixed part of the deformations. The last term 
$\Hom_{\sS}(\sI_{\sZ},\sI_{\sZ}\otimes K_{\sS})_0\Tt=0$ since $\sI_{\sZ}\to \sI_{\sZ}$ is an isomorphism and taking trace-free we get zero. 
So the fixed part of the obstruction vanishes by duality. This tells us that 
$$[\Hilb^{\mathbf{c}_0}(\sS)]^{\vir}=[\Hilb^{\mathbf{c}_0}(\sS)].$$
Then the virtual normal bundle is: 
\begin{multline}\label{eqn_FZ_virtual_normal_bundle}
N^{\vir}= \Big[\Ext^1_{\sS}(\sI_{\sZ},\sI_{\sZ}\otimes K_{\sS}^{-1})\Tt^{-1} \oplus  \Hom_{\sS}(\sI_{\sZ},\sI_{\sZ}\otimes K^2_{\sS})\Tt^{2}\Big]-\\
\Big[\Ext^1_{\sS}(\sI_{\sZ},\sI_{\sZ}\otimes K_{\sS})\Tt\oplus \Ext^1_{\sS}(\sI_{\sZ},\sI_{\sZ}\otimes K_{\sS}^2)\Tt^2\oplus   \Ext^2_{\sS}(\sI_{\sZ},\sI_{\sZ}\otimes K^{-1}_{\sS})\Tt^{-1}\Big].
\end{multline}

It is quite complicated to integrate  equivariant Chern class on $\Hilb^{\mathbf{c}_0}(\sS)$, but we can do an easy case. 
Let $(\sS, P)$ be a quintic surface with only a singular point $P\in \sS$ with $A_1$-type singularity. 
By choosing a constant modified Hilbert polynomial $1$ on $\sS$ such that under the orbifold Chern character morphism:
$$K_0(\sS)\to H^*(I\sS)=H^*(\sS)\oplus H^*(B\mu_2)$$
the class 
$$1\mapsto (1,1)$$
where the second $1$ means the trivial one dimensional $\mu_2$-representation.  Then in this case 
the Hilbert scheme 
$$\Hilb^{1}(\sS)=\pt$$
which is a point.  This can be seen as follows. Around the singular point $P$, there is an open affine neighborhood 
$P\in U\subset \sS$ such that 
$$U\cong [\cc^2/\mu_2]$$
where $\zeta\in \mu_2$ acts on $\cc^2$ by 
$$\zeta\cdot (x,y)=(\zeta x, \zeta^{-1}y).$$
The Hilbert scheme of one point on $P\in U$ corresponds to invariant $\mu_2$-representation of length $1$, which must be trivial.  Then integration in this case must be: 
\begin{align*}
\int_{\Hilb^1(\sS)}\frac{1}{e(N^{\vir})}=\int_{\pt}1=1.
\end{align*}

Next we perform a degree two calculation. 
Let $2$ be a constant modified Hilbert polynomial  on $\sS$ such that under the orbifold Chern character morphism:
$$K_0(\sS)\to H^*(I\sS)=H^*(\sS)\oplus H^*(B\mu_2)$$
the class 
$$2\mapsto (2,2)$$
where the second $2$ means the regular two dimensional $\mu_2$-representation.  Then in this case 
the Hilbert scheme 
$$\Hilb^{2}(\sS)=\widetilde{S}$$
where 
$$\sigma: \widetilde{S}\to S$$
is the crepant resolution of the coarse moduli space $S$ of $\sS$.  Then $\widetilde{S}$ is still a smooth surface. 
Then the  integration in this case can be written as: 
\begin{align*}
\int_{\widetilde{S}}\frac{1}{e(N^{\vir})}&=\int_{\widetilde{S}}
\frac{e(T^*_{\widetilde{S}}\Tt)\cdot e(T_{\widetilde{S}}\otimes K_{\sS}^2\Tt^2)\cdot e(H^0(K_{\sS}^2)^*\Tt^{-1})}
{e(T_{\widetilde{S}}\otimes K_{\sS}^{-1}\Tt^{-1})\cdot e(H^0(K_{\sS}^2)\Tt^2)}\\
&=\int_{\widetilde{S}}\frac{t^2 c_{\frac{1}{t}}(T^*_{\widetilde{S}})\cdot (2t)^2 c_{\frac{1}{2t}}(T_{\widetilde{S}}\otimes K_{\sS}^2)\cdot (-t)^{\dim}}
{(-t)^2 c_{-\frac{1}{t}}(T_{\widetilde{S}}\otimes K_{\sS}^{-1})\cdot (2t)^{\dim}}
\end{align*}
Only the $t^0$ term contributes and we let $t=1$, and get 

$$(-2)^{\dim}\int_{\widetilde{S}}
\frac{(1-\widetilde{c}_1+\widetilde{c}_2)4(1+\frac{1}{2}c_1(T_{\widetilde{S}}\otimes K_{\sS}^2)+\frac{1}{4}c_2(T_{\widetilde{S}}\otimes K_{\sS}^2))}
{1-c_1(T_{\widetilde{S}}\otimes K_{\sS}^{-1})+c_2(T_{\widetilde{S}}\otimes K_{\sS}^{-1})}$$
where $\widetilde{c}_i=c_i(\widetilde{S})$.  
Let 
$$c_1:=c_1(T_{\sS}),$$
then $c_1(K_{\sS})=-c_1$. 
We use the same formula for the Chern classes:
$$c_1(T_{\widetilde{S}}\otimes K_{\sS}^{2})=\widetilde{c}_1-4c_1, \quad  c_2(T_{\widetilde{S}}\otimes K_{\sS}^{2})=\widetilde{c}_2-2\widetilde{c}_1\cdot c_1+4 c_1^2;$$ 
$$c_1(T_{\widetilde{S}}\otimes K_{\sS}^{-1})=\widetilde{c}_1+2c_1, \quad c_2(T_{\widetilde{S}}\otimes K_{\sS}^{-1})=\widetilde{c}_2+\widetilde{c}_1\cdot c_1+c_1^2 $$
and have 
\begin{align}\label{eqn_final_integral_ADE}
\int_{\widetilde{S}}\frac{1}{e(N^{\vir})}
&=(-2)^{-\dim}\int_{\widetilde{S}}\frac{(1-\widetilde{c}_1+\widetilde{c}_2)4(1+\frac{1}{2}(\widetilde{c}_1-4c_1)+\frac{1}{4}(\widetilde{c}_2-2\widetilde{c}_1\cdot c_1+4c_1^2))}
{1-(\widetilde{c}_1+2c_1)+\widetilde{c}_2+ \widetilde{c}_1\cdot c_1+c_1^2} \\
&=(-2)^{-\dim}\int_{\widetilde{S}}
(\widetilde{c}_2+14 \widetilde{c}_1\cdot c_1+ 4(\widetilde{c}_1)^2).  \nonumber
\end{align}

\subsubsection{Comparison with the Euler characteristic of the Hilbert scheme of points}

Let $\sS$ be a surface with finite ADE singularities $P_1, \cdots, P_s$.  The generating function of the Euler characteristic of the  Hilbert scheme of points on $\sS$ has been studied in \cite{GNS}, \cite{Toda}.  Let us recall the formula for the surface $\sS$ with $A_n$ singularities from \cite{Toda}, and 
let $P_1,\cdots, P_s$ have singularity type $A_{n_1}, \cdots, A_{n_s}$. 
Let $\sS\to S$ be the map to its coarse moduli space and $\sigma: \widetilde{S}\to S$ be the minimal resolution. 
Toda uses wall crossing formula to calculate that 
\begin{equation}\label{eqn_generating_Hilbert_scheme}
\sum_{n\geq 0}\chi(\Hilb^n(\sS))q^{n-\frac{\chi(\widetilde{S})}{24}}=\eta(q)^{-\chi(\widetilde{S})}\cdot \prod_{i=1}^{s}\Theta_{n_i}(q),
\end{equation}
where $\eta(q)=q^{\frac{1}{24}}\prod_{n\geq 1}(1-q^n)$ is the Dedekind eta function,  and 
$$
\Theta_{n}(q)=\sum_{(k_1,\cdots, k_n)\in\zz^n}q^{\sum_{1\leq i\leq j\leq n}k_i k_j}
e^{\frac{2\pi \sqrt{-1}}{n+2}(k_1+2k_2+\cdots+nk_n)}$$

The series $\Theta_n(q)$ 
is a $\qq$-linear combination of the theta series determined by some integer valued positive definite quadratic forms on $\zz^n$ and   $\Theta_n(q)$ is a modular form of weight $n/2$.  So 
the generating series (\ref{eqn_generating_Hilbert_scheme}) is a Fourier development of a meromorphic modular form of weight $-\chi(S)/2$ for some congruence subgroup in $SL_2(\zz)$.
This implies that it should be related to the S-duality conjecture for such surface DM stacks $\sS$. 
But the Euler characteristic of $\Hilb^n(\sS)$ is not the same as the contribution of it to the Vafa-Witten invariants $\VW(\sS)$, which is the integration over its virtual fundamental cycle in (\ref{eqn_rM_integral_isolated}). 

\begin{rmk}
Note that if the surface DM stack $\sS$ is a smooth projective surface $S$, the formula (\ref{eqn_generating_Hilbert_scheme}) is reduced to the G\"ottsche formula
$$\sum_{n\geq 0}\chi(\Hilb^n(S))q^{n-\frac{\chi(S)}{24}}=\eta(q)^{-\chi(S)}.
$$
This was the reason why in  \cite{TT1}, \cite{TT2} Tanaka-Thomas  thought in the beginning that the small Vafa-Witten invariants 
$\vw(S)$ defined using the Behrend function are the right Vafa-Witten invariants, but eventually changed their mind by calculations that the virtual localized invariants $\VW(S)$ are the correct choice for the Vafa-Witten invariants.
\end{rmk}

\subsection{Discuss on the crepant resolutions}

Let $\sS$ be a quintic surface with Type $ADE$ singularities $P_1, \cdots, P_m$, where at the point $P_i$, $\sS$ has type $A_{n_i}, D_{n_i}$ or $E_6, E_7, E_8$-singularity for $n_i\in \zz_{>0}$. Let $\sigma: \sS\to S$ be the map to its coarse moduli space.  We fix 
$$f: \widetilde{S}\to S$$
to be the minimal resolution of singularities. Then $\widetilde{S}$ is  smooth projective surface, and it also satisfies 
$$K_{\widetilde{S}}\cong f^*K_{S}$$
since the resolution of the type ADE singularities does not affect the canonical divisor, see \cite{Reid}.  So $f$ is a {\em crepant resolution}. 

Look at the following diagram:
\begin{equation}\label{eqn_diagram}
\xymatrix{
\widetilde{S}\ar[dr]_{f}&& \sS\ar[dl]^{\sigma}\\
&S&
}
\end{equation}
It is interesting to compare the Vafa-Witten invariants of $\widetilde{S}$ and the Vafa-Witten invariants of $\sS$. 

Let us fix topological data $(\rk, \widetilde{c}_1, \widetilde{c}_2)$ for $\widetilde{S}$ and let 
$\N_{(\rk, \widetilde{c}_1, \widetilde{c}_2)}(\widetilde{S})$ be the moduli space of stable Higgs pairs with topology data $(\rk, \widetilde{c}_1, \widetilde{c}_2)$.
Similarly, fixing a $K$-group class $\mathbf{c}\in K_0(\sS)$, and let $\N_{\mathbf{c}}(\sS)$ be the moduli space of stable Higgs pairs with  $K$-group class $\mathbf{c}$.  Recall from \S \ref{subsec_CStar_fixed_locus},  there are two $\cc^*$-fixed part for the moduli spaces $\N_{(r, \widetilde{c}_1, \widetilde{c}_2)}(\widetilde{S})$ and  $\N_{\mathbf{c}}(\sS)$.  Let 
$$\rM_{(\rk, \widetilde{c}_1, \widetilde{c}_2)}(\widetilde{S})$$
and 
$$\rM_{\mathbf{c}}(\sS)$$
be the first type fixed locus which are the instanton branches.  Then from (\ref{eqn_virtual_Euler_number}), the Vafa-Witten invariants contributed from the instanton branch are given by virtual Euler numbers
$$\VW^{\instan}_{(\rk, \widetilde{c}_1, \widetilde{c}_2)}(\widetilde{S})
=\int_{[\rM_{(\rk, \widetilde{c}_1, \widetilde{c}_2)}(\widetilde{S})]^{\vir}}c_{\vd}(E_{\rM(\widetilde{S})}^{\bullet})$$
and 
$$\VW^{\instan}_{\mathbf{c}}(\sS)
=\int_{[\rM_{\mathbf{c}}(\sS)]^{\vir}}c_{\vd}(E_{\rM(\sS)}^{\bullet})$$
where $E_{\rM(\widetilde{S})}^{\bullet}$ and $E_{\rM(\sS)}^{\bullet}$ are the perfect  obstruction theories of 
$\rM_{(\rk, \widetilde{c}_1, \widetilde{c}_2)}(\widetilde{S})$
and 
$\rM_{\mathbf{c}}(\sS)$ respectively. 
It is in general interesting to compare these two virtual Euler numbers. 

The second type of components of the moduli spaces of Higgs pairs  are denoted by 
$$\rM^{(2)}_{(\rk, \widetilde{c}_1, \widetilde{c}_2)}(\widetilde{S})$$
and 
$$\rM^{(2)}_{\mathbf{c}}(\sS).$$
From Proposition \ref{prop_second_fixed_loci_Hilbert_scheme_isolated}, in the rank $2$ case, 
we fix a $K$-group class 
$\mathbf{c}\in K_0(\sS)$ such that $\left(\widetilde{\Ch}_{\xi}\right)^1(\mathbf{c})=-c_1(\sS)$, then the component $\rM^{(2)}_{\mathbf{c}}(\sS)$ is a disjoint union of nested Hilbert schemes.  Similar result also holds for $\widetilde{S}$, see \cite[Lemma 8.3]{TT1}.  Theorem \ref{thm_generating_function_ADE_quintic} gives the generating function of one component
invariant  in $\rM^{(2)}_{\mathbf{c}}(\sS)$ (the special case of the Hilbert scheme of $n=c_2$ points on $\sS$).
Since for the smooth surface $\widetilde{S}$,  $K_{\widetilde{S}}$ is isomorphic to the canonical class of $S$, which is isomorphic to $K_{\sS}$ under the pullback of $\sigma$,  the surface $\widetilde{S}$ is also of general type and satisfies the condition that 
the line bundle $\sO_S$ is the only line bundle $L$ satisfying $0\leq \deg L\leq \frac{1}{2}\deg K_{\widetilde{S}}$ where the degree is defined  by 
$\deg L=c_1(L)\cdot c_1(\sO_{\widetilde{S}}(1))$.
Hence the result in Proposition 8.22 of \cite{TT1} holds for the component of Hilbert scheme of $n=\widetilde{c}_2$ points on $\widetilde{S}$.  Therefore we verify:
\begin{prop}\label{prop_crepant_resolution}
Within the monopole branches $\rM^{(2)}_{(\rk, \widetilde{c}_1, \widetilde{c}_2)}(\widetilde{S})$ and 
$\rM^{(2)}_{\mathbf{c}}(\sS)$, when putting into the generating function the Vafa-Witten  invariants  of the component coming from the Hilbert scheme of 
$\widetilde{c}_2=n$ on $\widetilde{S}$ and $\sS$ are equal.  $\square$
\end{prop}

\begin{rmk}
It is very interesting to formulate a conjecture to compare the Vafa-Witten invariants $\VW$ of $\widetilde{S}$ and $\sS$ in the more general settings.   This is unknown to the authors. 

If we consider the invariants $\vw$ by using weighted Euler characteristics by the Berhend function, then it is hoped that the invariants $\vw$ will be put into the Joyce-Song \cite{JS} wall crossing formula of changing stability conditions in the derived categories of coherent sheaves on $\widetilde{S}$ and $\sS$ in Diagram (\ref{eqn_diagram}) or the derived categories of coherent sheaves on the total spaces of the canonical line bundles Tot($K_{\widetilde{S}}$) and Tot($K_{\sS}$).  We leave this as a future work. 

If $\sS$ is a quintic surface with other type of singularities rather than ADE singularities, for instance, Wahl singularities, then the minimal resolution of the coarse moduli space $S$ will not be crepant.  It is more difficult in this case to calculate the Vafa-Witten invariants. 
\end{rmk}




\appendix

\section{The perfect obstruction theory, following Tanaka-Thomas}\label{Appendix_POT}

In this appendix we collect some basic materials for the perfect obstruction theory on the moduli space of Higgs pairs for surface DM stacks $\sS$. Our main reference is Sections 3, 5 of  \cite{TT1}.   For the perfect obstruction theory and symmetric obstruction theory we use the intrinsic normal cones in \cite{BF}, \cite{Behrend}.  Since \cite{BF} deals with DM stacks, it is reasonable that the construction of perfect obstruction theory on the Higgs pairs  works for DM stacks. 

\subsection{The perfect obstruction theory for $U(\rk)$-invariants}

First for the moduli spaces $\N$ and $\rM$, we have the exact sequence for the full cotangent complexes:
$$\Pi^*\ll_{\rM/B}\stackrel{\Pi^*}{\longrightarrow}\ll_{\N/B}\longrightarrow \ll_{\N/\rM}.$$

\subsubsection{Atiyah class and obstruction theory}

We recall the Atiyah class for a coherent sheaf $\sF$ on a $B$-DM stack $\sZ$. Let 
$$\sZ[\sF]:=\spec_{\sO_{\sZ}}(\sO_{\sZ}\oplus \sF)\stackrel{q}{\longrightarrow}\sZ$$
be the trivial square-zero thickening of $\sZ$ by $\sF$.  Here $\sO_{\sZ}\oplus \sF$ is a $\sO_{\sZ}$-algebra and the product is given by
$$(f,s)\cdot (g,t)=(fg, (ft+gs)).$$
There is a $\cc^*$-action fixing $\sO_{\sZ}$ and have weight one on $\sF$. Then we have an exact triangle of $\cc^*$-equivariant cotangent complexes:
\begin{equation}\label{eqn_A1}
q^*\ll_{\sZ/B}\to \ll_{\sZ[\sF]/B}\to \ll_{\sZ[\sF]/\sZ}\to q^*\ll_{\sZ/B}[1].
\end{equation}
Applying $q_*$ and taking the weight one part is an exact functor. We use the last two terms of (\ref{eqn_A1}) and apply $q_*$ to get:
$$\sF\to \sF\otimes \ll_{\sZ/B}[1]$$
and this belongs to $\Ext^1(\sF, \ll_{\sZ/B})$. We call this morphism the ``Atiyah class" of $\sF$. 

Apply the Atiyah class construction to the universal sheaf $\rE$ on $\N\times_{B}\sX$ we get 
\begin{equation}\label{eqn_Atiyah_rE}
\rE\to \rE\otimes \ll_{\N\times_{B}\sX/B}[1]
\end{equation}
Since  $\ll_{\N\times_{B}\sX/B}= \ll_{\N/B}\oplus  \ll_{\sX/B}$ (ignoring the pullbacks), we project it to the first summand and get the partial Atiyah class
\begin{equation}\label{eqn_Atiyah_rE2}
\At_{\rE, \N}: \rE\to \rE\otimes \ll_{\N/B}[1].
\end{equation}
Similarly, for $\E$ on $\N\times_{B}\sS$ we do: 
\begin{equation}\label{eqn_Atiyah_E}
\At_{\E, \N}: \E\to \E\otimes \ll_{\N/B}[1].
\end{equation}
Same proof as in \cite[Proposition 3.5]{TT1} gives:
\begin{equation}\label{eqn_Atiyah_pi_rE}
\pi_*\At_{\rE, \N}: \pi_*\rE\to \pi_*\rE\otimes \ll_{\N/B}[1]
\end{equation}
is (\ref{eqn_Atiyah_E}). 
\begin{rmk}
Tanaka-Thomas use the $\cc^*$-action on $\sX[\rE]$ and $\sS[\E]$ such that it has weight one on $\rE$, $\E$ respectively, but fixes $\sX$ and $\sS$. Then applying the definition of partial Atiyah class the statement is proved. 
\end{rmk}

Now from (\ref{eqn_Atiyah_rE2}), 
$$\At_{\rE, \N}\in \Ext_{\N\times_{B}\sX}^1(\rE,\rE\otimes p_{\sX}^*\ll_{\N/B})
=\Ext_{\N\times_{B}\sX}^1(R\cHom(\rE,\rE), p_{\sX}^*\ll_{\N/B}).$$
Let us project it to $p_{\sX}: \N\times_{B}\sX\to \N$ and use the relative Serre duality:
$$\At_{\rE, \N}\in \Ext_{\N\times_{B}\sX}^2(Rp_{\sX *}R\cHom(\rE,\rE\otimes K_{\sX/B}), \ll_{\N/B}).$$
In order to do the calculation we put the $\cc^*$-action to $K_{\sX/B}$.  Then he canonical line bundle 
$K_{\sX/B}\cong \sO_{\sX/B}$ is not equivariant trivial. The $\cc^*$ acts on $K_{\sX/B}$ with weight $-1$, so
$K_{\sX/B}\cong \sO\otimes \mathfrak{t}^{-1},$
where $\Tt^{-1}$ is the standard one-dimensional representation of $\cc^*$. 
 Hence we get: 
\begin{equation}\label{eqn_Atiyah_deformation_obstruction1}
\At_{\rE, \N}: R\cHom_{p_{\sX}}(\rE,\rE)[2]\Tt^{-1}\to  \ll_{\N/B}.
\end{equation}

\begin{prop}\label{prop_SPOT}
The truncated morphism in (\ref{eqn_Atiyah_deformation_obstruction1})
\begin{equation}\label{eqn_Atiyah_deformation_obstruction2}
\tau^{[-1,0]}(R\cHom_{p_{\sX}}(\rE,\rE)[2])\Tt^{-1}\to  \ll_{\N/B}
\end{equation}
is a relative symmetric perfect obstruction theory of amplitude $[-1,0]$ for $\N$. 
\end{prop}
\begin{proof}
We combine the results in \cite[Theorem 3.11, Corollary 3.12]{TT1}. 
We first need to check Condition (3)  in Theorem 4.5 of \cite{BF}, so that 
the morphism in (\ref{eqn_Atiyah_deformation_obstruction1}) 
$$R\cHom_{p_{\sX}}(\rE,\rE)[2]\Tt^{-1}\to  \ll_{\N/B}.$$
is an obstruction theory.  Behrend-Fantechi's construction is for DM stacks, therefore if working locally in \'etale topology, the arguments in \cite[Theorem 3.11]{TT1} works for DM stacks. 

Let $T$ be a $B$-scheme and $T\subset \overline{T}$ a square zero extension with ideal sheaf $I$ and $I^2=0$. 
Let $g: T\to \N$ be a morphism. Then the pullback of (\ref{eqn_Atiyah_deformation_obstruction1}) is:
$$g^*R\cHom_{p_{\sX}}(\rE,\rE)[2]\to g^*\ll_{\N/B};$$
There is natural morphism 
$$g^*\ll_{\N/B}\to \ll_{T/B};$$
and 
$$\ll_{T/B}\to \ll_{T/\overline{T}}\to \tau^{\geq -1}\ll_{T/\overline{T}}=I[1],$$
which is the  Kodaira-Spencer class.
Compose these three morphisms we get an element 
$$\ob\in \Ext^{-1}_{T}(g^*R\cHom_{p_{\sX}}(\rE,\rE), I).$$
To check Condition (3) in Theorem 4.5 of \cite{BF}, we need to show that $\ob$ vanishes if and only if there exists an extension $\overline{T}$ of $T$ and a map of $g$ is $B$-morphism; and also when $\ob=0$ the set of extensions is a torsor under $\Ext^{-2}_{T}(g^*R\cHom_{p_{\sX}}(\rE,\rE), I)$. 
This is from the proof of the second half of \cite[Theorem 3.11]{TT1}, by showing that the pullback of the partial Atiyah class $\At_{\rE, \N}$ through 
$$\overline{g}: \id\times g: \sX\times_{B}T\to \sX\times_{B}\N,$$
gives the partial Atiyah class $\At_{\overline{g}^*\rE, T}$, then composed with the Kodaira-Spencer class gives 
$$\ob\in  \Ext^{-1}_{T}(g^*R\cHom_{p_{\sX}}(\rE,\rE), I)\cong 
\Ext^{2}_{\sX\times_{B}T}(\overline{g}^*\rE, \overline{g}^*\rE\otimes I).$$
Then apply \cite[Proposition III. 3.1.5]{Illusie} to get the Condition (3) in Theorem 4.5 of \cite{BF}.

Next we prove that the truncation $\tau^{[-1,0]}R\cHom_{p_{\sX}}(\rE,\rE)[2]$ gives a symmetric perfect obstruction theory.  First note that $R\cHom_{p_{\sX}}(\rE,\rE)$ is perfect of amplitude $[0,3]$, since we are working on $\sX$ which is dimension three.  The stable sheaves are simple and have automorphism group $\cc$. Then taking cones
$$\Cone\left(\sO_{\N}\stackrel{\id}{\longrightarrow}R\cHom_{p_{\sX}}(\rE,\rE)\right)=
\tau^{\geq 1}R\cHom_{p_{\sX}}(\rE,\rE)$$
is perfect of amplitude $[1,3]$. 

Serre duality gives $\Ext^3_{\sX}(\rE_t, \rE_t)\cong \Hom_{\sX}(\rE_t, \rE_t)=\cc$, therefore 
$$\Cone\left(\tau^{\geq 1}R\cHom_{p_{\sX}}(\rE,\rE)\stackrel{\tr}{\longrightarrow}\sO_{\N}(-3)\right)[-1]=
\tau^{[1,2]}R\cHom_{p_{\sX}}(\rE,\rE)$$
is perfect of amplitude $[1,2]$. Then the obstruction theory (\ref{eqn_Atiyah_deformation_obstruction1}) factors through these two truncations and give rise to the relative perfect obstruction theory
(\ref{eqn_Atiyah_deformation_obstruction2}).  It is symmetric since it is self-dual from Proposition \ref{prop_self_dual}. 
\end{proof}

\subsection{Deformation of Higgs fields}

We mimic the construction in  \cite[\S 5]{TT1} to fix the determinant of $E$ in $(E,\phi)$, and make $\phi$ trace-free. 
This corresponds to the $SU(\rk)$-Higgs bundles in Gauge theory. Since almost all the construction works for DM stacks, we only review the essential steps and leave its details to  \cite[\S 5]{TT1}.

From the appendix,  
$$\pi_*\At_{\rE, \N}: \pi_{*}\rE\to \pi_{*}\otimes\ll_{\N}[1]$$
is $\At_{\E, \N}$. We have the following commutative diagram:
\[
\xymatrix{
\T_{\N}\ar[r]^{\pi_{*}}\ar[d]_{\At_{\rE, \N}}& \Pi^*\T_{\rM}\ar[d]^{\At_{\E,\N}}\\
R\cHom_{p_{\sX}}(\rE,\rE)[1]\ar[r]^{\Pi_{*}}& R\cHom_{p_{\sS}}(\E,\E)[1].
}
\]
We dualize the above diagram and use exact triangle (\ref{eqn_deformation2}), 
\begin{equation}\label{diagram_A.1-1}
\xymatrix{
R\cHom_{p_{\sS}}(\E,\E\otimes K_{\sS})[1]\ar[r]\ar[d]_{\At_{\E, \N}}& R\cHom_{p_{\sX}}(\rE,\rE)[2]\ar[r]\ar[d]_{\At_{\sE,\N}}&  R\cHom_{p_{\sS}}(\E,\E)[2]\ar[d]_{\At_{\sE, \N/\rM}}\\
\Pi^*\ll_{\rM}\ar[r]^{\Pi^*}& \ll_{\N}\ar[r] & \ll_{\N/\rM}.
}
\end{equation}
We think of the right hand side vertical arrow as the projection of $\At_{\rE, \N}$ from $\ll_{\N}$ to $\ll_{\N/\rM}$.

\subsubsection{Deformation of quotient $\pi^*\E\to \rE$} 

Fro the morphism $\Pi: \N\to \rM$, the fibre over any $E\in\rM$ is the space of Higgs fields $\phi$, and we can take 
$\sE_\phi$ as the quotient: $\pi^*E\to \sE_\phi\to 0$ as in (\ref{eqn_rE_quotient}). So $\N/\rM$ is part of Quot scheme of $\pi^*E$.  From exact sequence (\ref{eqn_deformation1}), the deformation and obstruction are governed by 
\begin{equation}\label{eqn_deformation_quot}
\Hom_{\sX}(\pi^*E\otimes K_{\sS}^{-1},\sE)\cong \Hom_{\sS}(E, E\otimes K_{\sS})
\end{equation}
and 
\begin{equation}\label{eqn_obstruction_quot}
\Ext^1_{\sS}(E, E\otimes K_{\sS}).
\end{equation}
These are the cohomologies of $(R\cHom(E,E)[2])^{\vee}$.  We want to explain that the diagram exactly gives the usual obstruction theory of Quot scheme $\N/\rM$. 

We recall the reduced Atiyah class of the quotient of Illusie. Consider $\N\times \sX[\rE]$ with its projection to 
$\N\times \sX/\rM\times \sX$, and similarly  $\N\times \sX[\pi^*\E]$ and $\N\times\sX[\rE]$ with its projection to 
$\N\times \sX/\rM\times \sX$.  The morphism
$$\N\times\sX[\rE]\hookrightarrow \N\times\sX[\pi^*\E]$$
is an embedding, since $\pi^*\E\to \rE$ is a quotient with ideal $\pi^*\E\otimes K_{\sS}^{-1}$. 
We have the commutative diagram of exact triangles:
\begin{equation}\label{diagram_A.1-2}
\xymatrix{
\ll_{\N\times\sX[\rE]/\rM\times\sX}\ar[r]\ar[d]& \ll_{\N\times\sX[\rE]/\N\times\sX}\ar[r]\ar[d]&  
\ll_{\N\times\sX/\rM\times\sX}\ar[d]\\
\ll_{\N\times\sX[\rE]/\rM\times\sX[\pi^*\E]}\ar[r]& \ll_{\N\times\sX[\rE]/\N\times\sX[\pi^*\E]}\ar[r] &
\ll_{\N\times\sX[\pi^*\E]/\rM\times\sX[\pi^*\E]}.
}
\end{equation}
on $\N\times \sX[\rE]$.  Taking the degree $1$ part and pushdown to $\N\times\sX$ we get the right square of the above diagram
\[
\xymatrix{
\rE\ar[r]\ar[d] & \rE\otimes \ll_{\N/\rM}[1]\ar@{=}[d]\\
\pi^*\E\otimes K_{\sS}^{-1}[1]\ar[r]& \rE\otimes\ll_{\N/\rM}[1].
}
\]
The reduced Atiyah class of the quotient $\pi^*\E\to \rE$ is given by the bottom arrow above:
\begin{equation}\label{eqn_A.1-3}
\At_{\Phi}^{\red}\in \Hom(\pi^*\E\otimes K_{\sS}^{-1}, \rE\otimes \ll_{\N/\rM})
\cong \Hom(\E\otimes K_{\sS}^{-1},\E\otimes\ll_{\N/\rM})
\end{equation}
since $\pi_{*}\rE=\E$. 
\begin{prop}(\cite[Proposition 5.8]{TT1})\label{prop_two_Atiyah_agree}
The right hand side vertical arrow $\At_{\rE, \N/\rM}$ in (\ref{diagram_A.1-1}) is $\At_{\Phi}^{\red}$.
\end{prop}

\begin{rmk}
The obstruction theory on $\N/\rM$ induced from taking $\N$ as the moduli space of coherent sheaves $\sE$ on 
$\sX$ is the same as the standard obstruction theory for quotients $\pi^*E\to \sE\to 0$ by the Atiyah class. 
\end{rmk}

\subsubsection{Higgs fields deformations}

Recall that for the map $\Pi: \N\to \rM$, the fibre over a fixed sheaf $E\in\rM$ is the space of Higgs fields $\phi$.  Then the tangent of such a Higgs field is given by 
\begin{equation}\label{eqn_A.1-local_deformation}
\Hom(E,E\otimes K_{\sS})
\end{equation}
and the obstruction is given by
\begin{equation}\label{eqn_A.1-local_obstruction}
\Ext^1(E,E\otimes K_{\sS}).
\end{equation}
So there will have a relative perfect obstruction theory by putting these local deformations and obstructions together:
\begin{equation}\label{eqn_A.1-4}
R\cHom(\E,\E\otimes K_{\sS})^{\vee}\to \ll_{\N/\rM}.
\end{equation}
We argue as in \cite[\S 5]{TT1} that (\ref{eqn_A.1-4}) is the same as (\ref{eqn_A.1-3}) given by the reduced Atiyah class. 

In  \cite[\S 5]{TT1} the authors check that these two obstruction theories are the same in three steps. 
First restrict the moduli space $\rM$ to a point $E$, where the moduli space of Higgs fields $\phi$ on $E$ is the linear vector space 
$$\sfH:=\Hom(E, E\otimes K_{\sS})$$
i.e., the tangent space at a point $\phi\in \sfH$, as a linear space and as a space of quotients
$$\Hom(\pi^*(E\otimes K_{\sS}^{-1}), \sE)$$
are the same. 
This case works for surface DM stacks, since one can always pick up a geometric point $E$ in the moduli space. 

Next we check the two obstruction theories for Higgs bundles on $\sS$. 
So let $\rM$ be the moduli space of vector bundles on $\sS$, where we shrink $\rM$ if necessary to achieve this. 
Let $\E$ be the universal bundle on $\rM\times \sS$ and let 
$$\sfH:=\cHom(\E, \E\otimes K_{\sS})\stackrel{\rho}{\longrightarrow}\rM\times \sS.$$
Then over $\widetilde{\sfH}:=\sfH\times_{\rM\times \sS}\rM\times\sX\stackrel{\pi}{\longrightarrow}\sfH$ we have a universal Higgs field $\Phi$ and a universal quotient 
$$0\to \pi^*(\rho^*\E\otimes K_{\sS}^{-1})\longrightarrow \pi^*(\rho^*\E)\longrightarrow \rE\to 0.$$
Then using the linear structure of the fibre of $\rho$, we have:
\begin{equation}\label{eqn_A.1-5}
\rho^*\sfH\cong \T_{\sfH/\rM\times \sS}\stackrel{\At_{\Phi}^{\red}}{\longrightarrow}
\pi_{*}\cHom(\pi^*(\rho^*\E\otimes K_{\sS}^{-1}),\rE)\cong \cHom(\rho^*\E, \rho^*\E\otimes K_{\sS})\cong \rho^*\sfH.
\end{equation}
Then the composition morphism in (\ref{eqn_A.1-5}) is identity, see \cite[Lemma 5.8]{TT1}. 
Now take $\N/\rM$ as the moduli space of sections of $\sfH\to \rM\times\sS$, the graph of $\Phi$ gives an embedding 
\begin{equation}\label{eqn_A.1-6}
\N\times\sS\stackrel{\Phi}{\hookrightarrow}\Pi^*\sfH.
\end{equation}
where $\Pi: \N\times\sS\to \rM\times\sS$ is the projection.  The normal bundle of $\N\times\sS$ in $\Pi^*\sfH$ is the fiberwise tangent bundle of 
$$\Pi^*\sfH\to \N\times\sS$$
and is just $\Pi^*\sfH$ by the linear structure.  Therefore let 
\begin{equation}\label{eqn_A.1-7}
N_{\Phi}:=\Cone\left(\T_{\N\times\sS}\stackrel{D\Phi}{\longrightarrow}\Pi^*\T_{\Pi^*\sfH}\right)
\cong \Phi^*\T_{\Pi^*\sfH/\N\times\sS}\cong \Pi^*\sfH.
\end{equation}
Then consider 
$$p_{\sS}^*\T_{\N/\rM}\cong \T_{\N\times\sS/\rM\times\sS}\stackrel{D\Phi}{\longrightarrow}\Phi^*\T_{\Pi^*\sfH/\N\times\sS}$$
and applying $Rp_{\sS*}$ we get:
\begin{equation}\label{eqn_A.1-8}
\T_{\N/\rM}\longrightarrow Rp_{\sS*}(N_{\Phi})\cong Rp_{\sS*}(\Pi^*\sfH)\cong R\cHom_{p_{\sS}}(\E, \E\otimes K_{\sS}).
\end{equation}

\begin{prop}\label{prop_POT_relative}
The relative perfect obstruction theory  
$$R\cHom_{p_{\sS}}(\E, \E\otimes K_{\sS})^{\vee}\to \ll_{\N/\rM}$$
by taking dual of (\ref{eqn_A.1-8}) is the same as the right hand side arrow in (\ref{diagram_A.1-1}).
\end{prop}
\begin{proof}
This is from (\ref{eqn_A.1-5}).
\end{proof}

\subsubsection{Trace}

We put the trace of the section $\Phi$ in (\ref{eqn_A.1-6}) and get:
$$\N\times\sS\stackrel{\Phi}{\hookrightarrow}\Pi^*\sfH\stackrel{\tr}{\longrightarrow}\N\times K_{\sS}.$$
Then using the same analysis as in (\ref{eqn_A.1-7}), (\ref{eqn_A.1-8}) for the section $\tr \Phi$ we get a relative obstruction theory
$$\T_{\rM\times H^0(K_{\sS})/\rM}\to Rp_{\sS*}K_{\sS}$$
for $\rM\times \Gamma(K_{\sS})\to \rM$, and a commutative diagram:
\begin{equation}\label{eqn_diagram_A.1-9}
\xymatrix{
\T_{\N/\rM}\ar[r]^--{D(\tr\Phi)}\ar[d]_{\At_{\Phi}^{\red}}& (\tr\Phi)^*\T_{\rM\times\Gamma(K_{\sS})/\rM}\ar[d]\\
R\cHom_{p_{\sS}}(\E, \E\otimes K_{\sS})\ar[r]^--{\tr} & Rp_{\sS*}K_{\sS}.
}
\end{equation}
The diagram is compatible with the morphisms
$$\N\to \rM\times\Gamma(K_{\sS}); \quad  (E,\phi)\mapsto (E, \tr\phi).$$
Also $\T_{\rM\times\Gamma(K_{\sS})/\rM}\cong \Gamma(K_{\sS})\otimes \sO_{\rM}$, and the right hand side arrow in  (\ref{eqn_diagram_A.1-9}) is the canonical embedding:
$$\Gamma(K_{\sS})\otimes \sO_{\rM}\stackrel{H^0}{\longrightarrow}R\Gamma(K_{\sS})\otimes \sO_{\rM}
\cong Rp_{\sS*}K_{\sS}.$$
Then taking co-cones of (\ref{eqn_diagram_A.1-9}) we get the commutative diagram of exact triangles:
\begin{equation}\label{eqn_diagram_A.1-10}
\xymatrix{
\T_{\N/\rM\times\Gamma(K_{\sS})}\ar[r]\ar[d]_{\At_0^{\red}}&\T_{\N/\rM}\ar[r]^--{\tr\Phi}\ar[d]_{\At_{\Phi}^{\red}}& (\tr\Phi)^*\T_{\rM\times\Gamma(K_{\sS})/\rM}\ar[d]\\
R\cHom_{p_{\sS}}(\E, \E\otimes K_{\sS})_0\ar[r]& R\cHom_{p_{\sS}}(\E, \E\otimes K_{\sS})\ar[r]^--{\tr} & Rp_{\sS*}K_{\sS}
}
\end{equation}
where $\At_0^{\red}$ is the trace-free component of $\At_{\Phi}^{\red}$ in the splitting of the top row. 

Then from \cite[Lemma 5.28]{TT1}, the dual of the left hand arrow gives a perfect obstruction theory. 
We combine the diagram (\ref{eqn_diagram_A.1-10})  to (\ref{diagram_A.1-1}), therefore the diagram 
\[
\xymatrix{
\T_{\N/\rM\times\Gamma(K_{\sS})}\ar@{->}[r]\ar[d]& \T_{\N/\rM}\ar@{->}[r]^--{}_--{}\ar[d] & \T_{\Gamma(K_{\sS})}\ar@{=}[d] \\
\T_{\N/\Gamma(K_{\sS})}\ar@{->}[r]\ar[d]& \T_{\N}\ar@{->}[r]\ar[d]& \T_{\Gamma(K_{\sS})}\\
\T_{\rM}\ar@{=}[r] & \T_{\rM}
}
\] 
maps to the diagram:
\[
\xymatrix@C=2em@R1.8em{
R\cHom_{p_{\sS}}(\E, \E\otimes K_{\sS})_0\ar@{<->}[r]\ar[d]& R\cHom_{p_{\sS}}(\E, \E\otimes K_{\sS})\ar@{<->}[r]^--{\tr}_--{\id}\ar[d] & R\Gamma(K_{\sS})\ar@{=}[d] \\
R\cHom_{p_{\sX}}(\rE, \rE)^0[1]\ar@{<->}[r]\ar[d]& R\cHom_{p_{\sX}}(\rE, \rE)[1]\ar@{<->}[r]\ar[d]& R\Gamma(K_{\sS})\\
R\cHom_{p_{\sS}}(\E,\E)[1]\ar@{=}[r] & R\cHom_{p_{\sS}}(\E,\E)[1].
}
\] 
Therefore,
\begin{equation}\label{eqn_A.1-11}
\T_{\N/\Gamma(K_{\sS})}\to
R\cHom_{p_{\sX}}(\rE, \rE)^0[1]
\end{equation}
is a perfect obstruction theory. 

\subsubsection{Determinant}

The method of fixing the determinant is standard, and from (5.32), (5.33) of \cite{TT1}, the diagram: 
\[
\xymatrix{
\T_{\N/\rM\times\Gamma(K_{\sS})}\ar@{=}[r]\ar[d]& \T_{\N/\rM\times\Gamma(K_{\sS})}\ar[d]& \\
\T_{\N/\Gamma(K_{\sS})\times\Pic(\sS)}\ar@{->}[r]\ar[d]& \T_{\N/\Gamma(K_{\sS})}\ar@{->}[r]^--{\det_{*}}_--{}\ar[d] & \T_{\Pic(\sS)}\ar@{=}[d] \\
\T_{\rM/\Pic(\sS)}\ar@{->}[r]& \T_{\rM}\ar@{->}[r]& \T_{\Pic(\sS)}
}
\] 
maps to the diagram:
\[
\xymatrix@C=2em@R1.8em{
R\cHom_{p_{\sS}}(\E, \E\otimes K_{\sS})_0\ar@{=}[r]\ar[d]& R\cHom_{p_{\sS}}(\E, \E\otimes K_{\sS})_0\ar[d] &\\
R\cHom_{p_{\sX}}(\rE, \rE)_{\perp}[1]\ar@{<->}[r]\ar[d]& R\cHom_{p_{\sX}}(\rE, \rE)^0[1]\ar@{<->}[r]\ar[d]& R\Gamma(\sO_{\sS})[1]\ar@{=}[d]\\
R\cHom_{p_{\sS}}(\E,\E)_0[1]\ar@{<->}[r] & R\cHom_{p_{\sS}}(\E,\E)[1]\ar@{<->}[r]& R\Gamma(\sO_{\sS})[1].
}
\] 
We have the splitting of the Atiyah class 
$$\At_{\rE, \N}=(\At_{\rE, \N}^{\perp}, \At_{\det \rE, \N})$$
from the central row of the above diagram.

\begin{prop}\label{prop_symmetric_POT_Perp}
The morphism 
$$\At_{\rE, \N}^{\perp}: R\cHom_{p_{\sX}}(\rE, \rE)_{\perp}[2]\Tt^{-1}\longrightarrow \ll_{\N/\Gamma(K_{\sS})\times\Pic(\sS)}$$
is a $2$-term symmetric relative obstruction theory. 
\end{prop}
\begin{proof}
See \cite[Proposition 5.34]{TT1}. 
\end{proof}
\begin{rmk}
In general, for the moduli space of Higgs sheaves $\N$, one can take two-term locally free resolution 
$$0\to \E_2\to \E_1\to \E\to 0$$
(since $\E$ is torsion free and had homological dimension less than or equal to $1$ on $\sS$), the similar analysis as in \cite[\S 5.7]{TT1} shows that Proposition \ref{prop_symmetric_POT_Perp} gives a symmetric relative obstruction theory for the moduli space $\N$. 
\end{rmk}


\subsection*{}


\begin{thebibliography}{12}  
\bibitem{Alexeev} V. Alexeev, Boundedness and K2 for log surfaces, \newblock {\em International Journal of Mathematics} 5 (1994), no. 06, 779-810.
\bibitem{Anchouche}B.  Anchouche, \newblock Bogomolov inequality for Higgs parabolic bundles
{\em Manuscripta Mathematica}, 100(4) (1999), 423-436.
\bibitem{ACGH}E. Arbarello, M. Cornalba, P.A. Griffiths and J. Harris, \newblock {\em Geometry of algebraic curves}, Volume I, Springer-Verlag (1985). 
\bibitem{Behrend}  K. Behrend, \newblock Donaldson-Thomas invariants via microlocal geometry, 
{\em Ann. Math.} (2009), Vol. 170, No.3, 1307-1338, math.AG/0507523.
\bibitem{BF}   K. Behrend and B. Fantechi, \newblock The intrinsic normal cone,
                   alg-geom/9601010, {\em Invent. Math}. 128 (1997), no. 1, 45-88.
                     
\bibitem{Bryan-Graber}J. Bryan and T. Graber,  \newblock The crepant resolution conjecture, In Algebraic geometry-Seattle 2005. Part 1, volume 80 of Proc. Sympos. Pure Math., pages 23-42. Amer. Math. Soc., Providence, RI, 2009.
\bibitem{BO}J. Bryan and G. Oberdieck,  \newblock CHL Calabi-Yau threefolds: Curve counting, Mathieu moonshine and Siegel modular forms, arXiv:1811.06102.

\bibitem{BD}J. Bryan and D. Steinberg, \newblock Curve counting invariants for crepant resolutions, arXiv:1208.0884. 
\bibitem{Bryan-Young} J. Bryan and B. Young , \newblock Generating functions for colored 3D Young diagrams and the Donaldson-Thomas invariants of orbifolds, {\em Duke Mathematical Journal}, 2008 , 152 (1) :115-153. 
\bibitem{Cadman}C. Cadman, \newblock Using stacks to impose tangency conditions on curves, {\em Amer. J. Math.},
 Vol. 129, (2007), 405-427.  

\bibitem{Calabrese}J. Calabrese, \newblock On the Crepant Resolution Conjecture for Donaldson-Thomas Invariants, 
{\em J. Algebraic Geom.} 25 (2016), no. 1, 1-18.
arXiv:1206.6524.
\bibitem{GP}T.~Graber and R.~Pandharipande, \newblock Localization of virtual classes, {\em Invent.~Math.}~\textbf{135} 487-518, 1999. alg-geom/9708001.

\bibitem{GJK} A. Gholampour, Y. Jiang, and Martijn Kool, \newblock Sheaves on weighted projective planes and modular forms, {\em Advances in Theoretical and Mathematical Physics}, Vol. 21, NO. 6, (2017) 1455-1524, arXiv:1209.3922.
\bibitem{GT}A. Gholampour and R.P. Thomas, \newblock  Degeneracy loci, virtual cycles and nested Hilbert schemes, arXiv:1709.06105. 
\bibitem{Gottsche}L. G\"ottsche, \newblock 
Modular forms and Donaldson invaraints for 4-manifolds with $b_+=1$,  {\em Journal of the American Mathematical Society},   Vol. 9, No. 3 (1996), 827-843.

\bibitem{GK}L. G\"ottsche and M. Kool, \newblock Virtual refinements of the Vafa-Witten formula,
arXiv:1703.07196.
\bibitem{GNS} A. Gyenge,  A. N\'emethi and B. Szendroi, \newblock  Euler characteristics of Hilbert schemes of points on simple surface singularities, published online, European Journal of Mathematics. 

\bibitem{Horikava}E. Horikawa, \newblock On deformations of quintic surfaces, {\em  Inventiones mathematicae}, 
(1975), Vol 31, Iss 1,  43-85.
\bibitem{Gallardo}P.  Gallardo, \newblock On the GIT Quotient Space of Quintic Surfaces, arXiv:1310.3534.
\bibitem{Huybrechts}D. Huybrechts, \newblock {\em Lectures on K3 surfaces}, Cambridge Studies in Advanced Mathematics, 158. Cambridge University Press, Cambridge, 2016.
 \bibitem{HL}D. Huybrechts and M. Lehn, \newblock {\em The geometry of moduli spaces of sheaves}, Aspects of Mathematics, E31, Friedr. Vieweg \& Sohn, Braunschweig, 1997. MR MR1450870 (98g:14012).
 \bibitem{Lieblich}M. Lieblich, \newblock Moduli of twisted orbifold sheaves, {\em Advances in Mathematics}, 
Vol. 226, Iss. 5, (2011), 4145-4182.

\bibitem{Illusie} L. Illusie, \newblock \emph{Complexe cotangent et d\'eformations I},
Lec. Notes Math. \textbf{239}, Springer-Verlag, 1971.

\bibitem{JT}Y. Jiang and R. Thomas, \newblock Virtual signed Euler characteristics, {\em Journal of Algebraic Geometry}, 26 (2017) 379-397, arXiv:1408.2541.
\bibitem{Jiang}Y. Jiang, \newblock Note on MacPherson's local Euler obstruction, {\em Michigan Mathematical Journal}, 68 (2019), 227-250, arXiv:1412.3720.
\bibitem{Jiang2}Y. Jiang, \newblock Donaldson-Thomas invariants of Calabi-Yau orbifolds under flops, {\em Illinois Journal Math.} Volume 62, No 1-4 (2018), 61-97, arXiv:1512.00508.

\bibitem{Jiang3}Y. Jiang,  \newblock The Tanaka-Thomas's Vafa-Witten invariants for surface DM stacks II: Root stacks and parabolic Higgs pairs, in preparation. 
\bibitem{Jiang_2019}Y. Jiang, \newblock Counting twisted sheaves and S-duality, arXiv:1909:04241. 
\bibitem{JKT}Y. Jiang, and P. Kundu, \newblock  The Bogomolov inequality for surface Deligne-Mumford stacks, in preparation. 
\bibitem{JS}D. Joyce and Y. Song, \newblock A theory of generalized Donaldson-Thomas invariants, 
{\em Memoris of the AMS},  217 (2012),  1-216,  
arXiv:0810.5645. 

\bibitem{Joyce} D. Joyce, \newblock A classical model for derived critical locus,  {\em Journal of Differential Geometry}, 101 (2015), 289-367, arXiv:1304.4508.
\bibitem{Kapranov}M. Kapranov, \newblock The elliptic curve in the S-duality theory and Eisenstein series for Kac-Moody groups, arXiv:math/0001005v2 [math.AG]. 
\bibitem{KL}Y.-H. Kiem and J. Li, \newblock Localizing virtual cycles by cosections, {\em Jour. A.M.S.} 26
1025-1050, (2013),  arXiv:1007.3085.
 \bibitem{KM}D. Kotschick and  J. W. Morgan, \newblock 
SO(3)-Invariants for 4-manifolds with $ b_2^+= 1$ II, {\em  J. Differential  Geometry},  39 (1994) 433-456.

\bibitem{KW}A. Kapustin and E. Witten, \newblock Electric-Magnetic Duality
And The Geometric Langlands Program, 	arXiv:hep-th/0604151.

\bibitem{KSB}J. Kollar and N. I. Shepherd-Barron, \newblock  Threefolds and deformations of surface singularities, {\em Invent. Math.} 91 (1988), no. 2, 299-338. 

\bibitem{Kool}M. Kool, \newblock Fixed point loci of moduli spaces of sheaves on toric varieties, {\em Adv. Math.}, 
Vol. 227, Iss. 4, (2011), 1700-1755.
\bibitem{Laarakker}T. Laarakker,  \newblock Monopole contributions to refined Vafa-Witten invariants, arXiv:1810.00385.
\bibitem{LMB}G. Laumon and L. Moret-Bailly, \newblock  {\em Champs Algebriques},   Ergebnisse der Mathematik und ihrer Grenzgebiete. 3. Folge / A Series of Modern Surveys in Mathematics.

\bibitem{LT} J. Li and G. Tian, \newblock  Virtual moduli cycles and Gromov-Witten 
invariants of algebraic varieties, {\em J. Amer. Math. Soc.}, 11, 119-174, 1998, math.AG/9602007.    
\bibitem{MY}M. Maruyama and K. Yokogawa, \newblock Moduli of parabolic stable sheaves, {\em Math. Ann.} 293 (1992), no. 1, 77-99.         
\bibitem{MT}D. Maulik and R. P. Thomas, \newblock Sheaf counting on local K3  surfaces,  arXiv:1806.02657.    
\bibitem{MW} G. Moore and E. Witten, \newblock Integration over the u-plane in Donaldson theory,  {\em Adv. Theor. Math. Phys.} 1 (1997), no. 2, pages 298-387. 
\bibitem{NW}W. Nahm and K. Wendland, \newblock Mirror Symmetry on Kummer Type K3 Surfaces, 
{\em Commun.Math.Phys.} 243 (2003) 557-582, 
arXiv:hep-th/0106104.      

\bibitem{Nironi}F.  Nironi, \newblock Moduli Spaces of Semistable Sheaves on Projective Deligne-Mumford Stacks, 
arXiv:0811.1949. 
\bibitem{OS03}M. Olsson and J. Starr, \newblock Quot functors for Deligne-Mumford stacks, {\em Comm. Algebra} 31 (2003), no. 8, 4069-4096, Special issue in honor of Steven L. Kleiman. 
\bibitem{PTVV}T. Pantev, B. Toen, M. Vaquie and G. Vezzosi, \newblock Shifted symplectic struc-
tures, {\em Publ. Math. I.H.E.S.}117 (2013), 271-328. arXiv:1111.3209.
\bibitem{Rana}J. Rana, \newblock  A boundary divisor in the moduli space of stable quintic surfaces,  arXiv:1407.7148.

\bibitem{Reid}M. Reid, \newblock Du Val singularities, https://homepages.warwick.ac.uk/~masda/surf/more/DuVal.pdf/
\bibitem{Ruan}Y. Ruan, \newblock The cohomology ring of crepant resolutions of orbifolds,  In {\em Gromov-Witten theory of spin curves and orbifolds}, volume 403 of Contemp. Math., pages 117-126. Amer. Math. Soc., Providence, RI, 2006.
\bibitem{Stack_Project} Stack Project: https://stacks.math.columbia.edu/download/stacks-morphisms.pdf.
\bibitem{St1}R. Stanley, \newblock {\em Enumerative conbinatorics Vol 1},  Cambridge University Press (2012). 
\bibitem{St2}R. Stanley, \newblock {\em Enumerative conbinatorics Vol 2},  Cambridge University Press (1999). 
\bibitem{TT1} Y. Tanaka and R. P. Thomas, \newblock Vafa-Witten invariants for projective surfaces I: stable case, 
arXiv.1702.08487.
\bibitem{TT2}Y. Tanaka and R. P. Thomas, \newblock Vafa-Witten invariants for projective surfaces II: semistable case, 
arXiv.1702.08488.
\bibitem{Thomas} R. P. Thomas, \newblock  A holomorphic Casson invariant for Calabi-Yau 3-folds,
                      and bundles on K3 fibrations, {\em J. Differential Geom.}, 54, 367-438, 2000.
                       math.AG/9806111.
\bibitem{Thomas2} R. P. Thomas, \newblock  Equivariant K-theory and refined Vafa-Witten invariants, preprint, 
arXiv:1810.00078.
                                            
                       
\bibitem{Toda} Y. Toda, \newblock   S-Duality for surfaces with   An -type singularities, {\em Mathematische Annalen} 
(2015), Vol. 363, Issue 1-2,  679-699. 
                 
\bibitem{Toda2}Y. Toda, \newblock    Curve counting theories via stable objects II: DT/ncDT flop formula, {J. reine angew. Math.}  675 (2013), 1-51.   
\bibitem{VW} C.~Vafa and E.~Witten, \emph{A strong coupling test of S-duality}, Nucl. Phys. \textbf{B\,431} 3--77, 1994. hep-th/9408074.


\end{thebibliography}
\end{document}